\documentclass[12pt]{article}
\usepackage[a4paper, total={5.5in, 8in}]{geometry}
\usepackage[utf8]{inputenc}
\usepackage{mathtools}
\usepackage[hyphens]{url}
\usepackage{hyperref}
\hypersetup{
    breaklinks=true,
    colorlinks=true,
    linkcolor=blue,
    filecolor=blue,      
    urlcolor=blue,
    citecolor=blue,
    pdftitle={interpretations},
    }
\usepackage[hyphenbreaks]{breakurl}
\usepackage{amsmath}
\usepackage{extpfeil}
\usepackage{amsthm}
\usepackage{amssymb}
\usepackage{booktabs}
\usepackage{adjustbox}
\usepackage{ltablex,array}
\usepackage{longtable}
\usepackage{makecell}
\usepackage{marvosym}
\newcommand*{\faktor}[2]{
  \raisebox{0.5\height}{\ensuremath{#1}}
  \mkern-5mu\diagup\mkern-4mu
  \raisebox{-0.5\height}{\ensuremath{#2}}
} 
\newcommand\restr[2]{{
\left.
\kern-
\nulldelimiterspace 
#1 
\right|_{#2} 
}}
\usepackage{tikz}
\usepackage{tikz-cd}
\tikzcdset{scale cd/.style={every label/.append style={scale=#1},
    cells={nodes={scale=#1}}}}
\usepackage{quiver}
\usetikzlibrary{positioning} \usetikzlibrary{decorations.pathreplacing}
\usetikzlibrary{arrows}
\tikzset{>=stealth}

\newtheorem{theorem}{Theorem}[section]
\newtheorem{proposition}[theorem]{Proposition}

\newtheorem{corollary}[theorem]{Corollary}
\newtheorem{question}[theorem]{Question}
\theoremstyle{definition}
\newtheorem{remark}[theorem]{Remark}
\newtheorem{definition}[theorem]{Definition}
\newtheorem{notation}[theorem]{Notation}
\newtheorem{example}[theorem]{Example}
\theoremstyle{definition}
\usepackage[
backend=biber,
style=alphabetic,
sorting=nyt
]{biblatex}

\begin{document}
\title{Positive model theory of interpretations}
\author{Kristóf Kanalas}
\date{}

\maketitle
\begin{abstract}
    We prove analogues of model theory results for $\mathcal{C}\to \mathcal{D}$ coherent functors, including variants of the omitting types theorem and some results on ultraproduct constructions. We introduce a distributive lattice valued invariant of $\mathcal{C}\to \mathbf{Set}$ coherent functors that vanishes precisely on positively closed models, then we study its functorial properties.
\end{abstract}

\tableofcontents

\section{Introduction}

One of the many arguments towards categorical logic is that the syntactic treatment of interpretations between theories yields complicated definitions which are hard to handle. On that basis it seems to be unclear why we would need to define maps between interpretations, how such a category could be proved to have 2-categorical limits and colimits (e.g.~how should we glue theories together along interpretations) (see \cite{adjoint}), or how natural model theoretic constructions (e.g.~the formation of type spaces) are 2-functorial wrt.~interpretations and their homomorphisms (discussed in Section 4). To treat such questions we identify theories with algebraic structures (coherent categories) and their interpretations with the structure-preserving maps (coherent functors), following a simple idea explained in the Overview (based on \cite{makkai}).

The second (related) argument is that this identification tells us that interpretations are the same as models in some exotic categories (in small coherent categories instead of $\mathbf{Set}$), hence we can hope to prove model theoretic results for morphisms of theories. The third argument is that category theory has a considerable toolkit which can be applied once we translate logic problems to questions about coherent categories and coherent functors. There are many more. (Namely we can go further from the coherent category associated to the theory, and arrive at first to its pretopos completion (\cite{makkai}), then to the classifying topos (\cite{makkai},\cite{elephant}). Moreover from any of these syntactic presentations we can go to the semantics level: the accessible category of models (\cite{rosicky}). Each stage has its own theory and the transfer between these stages is also well-understood.) In this paper we will elaborate on the first three arguments.

\section{Overview}
We start with a quick summary of some chapters from the book \cite{makkai}, which will motivate our setting. In particular we describe how the study of $\mathcal{C}\to \mathbf{Set}$ coherent functors (a special case of our investigations) equals many-sorted positive model theory, discussed e.g.~in \cite{poizat}, \cite{poizatyaacov}, \cite{posmodth} and \cite{haykazyan}.

Given an $L$-structure $M$, we can think of it as a collection of sets (the interpretation of $L$-formulas) and a collection of set-functions (the ones whose graph is a definable set). That is, an $L$-structure is a subcategory of $\mathbf{Set}$, "parametrized" with formulas and function-like formulas between them. 

However, if we take a homomorphism $M\to M'$ it will not preserve these definable sets in general, only those which correspond to positive existential formulas (the ones built up from atomic formulas, $\bot $, $\vee $, $\top $, $\wedge $ and $\exists $). This suggests that when we treat $M$ as a subcategory of $\mathbf{Set}$ (the functorial image of something), we should only care about definable sets and functions given by positive existential formulas.

When we look at this reduced subcategory we can see not only whether a positive existential formula is valid in $M$, i.e.~whether it equals $[x_1=x_1\wedge \dots \wedge x_k=x_k]^M$, but we can also see containment between various definable sets. (In other terms: we can see the validity of one implication but not the validity of a formula with nested implications.) This motivates the definition of the coherent fragment $L_{\omega \omega }^g\subseteq L_{\omega \omega }$, which consists of formulas of the form $\forall \vec{x}(\varphi \to \psi )$ where $\varphi $ and $\psi $ are positive existential. We use the notation $\varphi \Rightarrow \psi $ and call it a coherent sequent.

This fragment is sufficiently rich: if we are given an arbitrary $L_{\omega \omega }$ theory, we can replace it with an $L_{\omega \omega }^g$-theory, whose models are the same, but the homomorphisms are the maps which were previously the elementary embeddings. (But after all this is what usually happens: either we take a simple (coherent) theory and then looking at homomorphisms or we take a theory with high complexity but then we only care about the elementary maps. This fragment captures both cases.) The construction is called Morleyization: we introduce a relation symbol for each formula $R_{\varphi }\subseteq X^n$ and then write things like $R_{\varphi }\wedge R_{\neg \varphi }\Rightarrow \bot $ and $\top \Rightarrow R_{\varphi }\vee R_{\neg \varphi }$, $R_{\exists x \varphi }\Leftrightarrow \exists x R_{\varphi }(x,\dots )$, etc. Using $\forall = \neg \exists \neg $ we ensured that for each $\varphi $: $\varphi ^M=R_{\varphi }^M$, and then we just add $\top \Rightarrow R_{\psi }(\vec{x})$ whenever $\psi $ is a formula in $T$.

The above mentioned parametrization is the so-called syntactic category: given a coherent theory $T\subseteq L_{\omega \omega }^g$, we construct its syntactic category $\mathcal{C}_T$ as follows: objects are positive existential $L$-formulas up to (allowed) change of free variables, and a morphism $[\varphi (\vec{x})]\to [\psi (\vec{y})]$ is given by a positive existential formula $\mu (\vec{x},\vec{y})$ such that in every model $M$ of $T$: $\mu ^M\subseteq \varphi ^M\times \psi ^M$ is the graph of a function (this has a syntactic characterization as well). We take a morphism to be such a $\mu $ up to $T$-provable equivalence.

This category will have some exactness properties, coming from the syntax, e.g.~we have finite products: The product of $[\varphi (x)]$ and $[\psi (x')]$ is $[\varphi (x)\wedge \psi (x')]$ where we change $x'$ to $x''$ if it was the same variable as $x$. (To verify this we would have to define the induced map with a formula, then prove its functionality, the commutativity of the triangles and uniqueness of the induced map: i.e.~we would have to prove that certain formulas are $T$-provably equivalent.) Note that the interpretation of this conjunct in $M$ is the direct product of $\varphi ^M$ and $\psi ^M$, so $M$, as a functor preserves finite products. If we list all these properties we get: finite limits, pullback-stable finite unions and pullback-stable image factorizations. Categories having all of these are called coherent.

Then we get what we want: $M:\mathcal{C}_T\to \mathbf{Set}$ coherent functors are the same as $T$-models in $\mathbf{Set}$ and $M\Rightarrow N$ natural transformations are the same as the model homomorphisms. Actually the situation is even better: given any coherent category $\mathcal{D}$ we can define what is an $L$-structure in $\mathcal{D}$: we choose an underlying object $A$, for function symbols we choose arrows $f^M:A^n\to A$ and for relation symbols subobjects $P^M\subseteq A^m$. It turns out that inductively we can define $\varphi ^M$ for any positive existential $\varphi $: for $t_1=t_2$ it is the equalizer of the interpretation of the terms, for $P(t_1,\dots t_m)$ it is a pullback, $\wedge $ and $\vee $ is given by intersection and union of subobjects, finally $\exists x\varphi $ is given by the image factorization of $\varphi ^M\hookrightarrow A^{n+1}\xrightarrow{\pi } A^n$. Homomorphisms are expressed with commutative squares. (E.g.~if $T$ is the theory of groups then these are group objects and their homomorphisms in $\mathcal{D}$.) It is still true that $\mathcal{D}$-models of $T$ correspond to $\mathcal{C}_T\to \mathcal{D}$ coherent functors and homomorphisms to natural transformations between those.

This gives two natural ideas: all results in (positive) model theory are just statements about $\mathcal{C}\to \mathbf{Set}$ coherent functors. So if we could reprove these theorems in the categorical language, using just a few properties of $\mathbf{Set}$, then we would get a theorem about $\mathcal{C}\to \mathcal{D}$ coherent functors for any coherent category $\mathcal{D}$ (at least if it satisfies a few extra categorical properties). The first possibility is to take $\mathcal{D}$ to be something interesting e.g.~compact Hausdorff spaces, and then study group objects/ ZFC objects, etc.~there, using model theory. The second possibility is to note that syntactic categories are also coherent and hence study $T\to T'$ interpretations = $\mathcal{C}_T\to \mathcal{C}_{T'}$ coherent functors this way. 

Actually all (small) coherent categories are syntactic categories of (small, many-sorted, coherent) theories, as the whole process has an inverse: given a coherent category $\mathcal{C}$ we can form its canonical signature $L_{\mathcal{C}}$ whose sorts are the objects in $\mathcal{C}$ and it only has unary function symbols; the arrows of $\mathcal{C}$. (Alternatively we could also take an extended canonical signature, where we introduce relation symbols for subobjects.) Then the theory is just writing down everything we see (=identities, commutative triangles, finite limits, unions, effective epimorphisms (images)), as if we were in $\mathbf{Set}$. E.g.~if we see that the function symbol $\widetilde{f}:\widetilde{X}\to \widetilde{Y}$ is coming from a monomorphism $f:X\to Y$ of $\mathcal{C}$, then we write $\widetilde{f}(x)=\widetilde{f}(x')\Rightarrow x=x'$ as an axiom. The trick is, that when we interpret the two formulas in a coherent category we do it by universal constructions (equalizers in this case), and hence this sequence is valid in a model $M$ iff $\widetilde{f}^M$ is a monomorphism (i.e.~satisfies the defining universal property: injectivity makes no sense in general). So $Th(\mathcal{C})$-models in $\mathcal{D}$ are the same as $\mathcal{C}\to \mathcal{D}$ coherent functors (and homomorphisms are the same as natural transformations). From this $\mathcal{C}_{Th(\mathcal{C})}\simeq \mathcal{C}$ also follows and we can summarize these facts as "coherent categories, coherent functors and natural transformations are the same as many-sorted coherent theories, their interpretations (=models), and the homomorphisms between interpretations".

Finally note that the following are synonyms: coherent theory, $\aleph _0$-geomet\-ric theory, h-inductive theory (what we called a coherent sequent is what is called a simple h-inductive sentence in \cite{posmodth}. An h-inductive theory is a set of finite conjunctions of simple h-inductive sentences, which is the same as a set of simple h-inductive sentences).

We mentioned one possible use, namely that we can generalize model theory statements to make theorems about $\mathcal{C}\to \mathcal{D}$ coherent functors. The first three sections are mainly dealing with this approach: 
\begin{itemize}
    \item We recall the definition of elementary natural transformations (Lecture 18X, Definition 2.~of \cite{lurie}), that correspond to those homomorphisms which reflect the validity of positive existential formulas. (Definition \ref{elementarynattr}) In \cite{posmodth} these are called immersions.
    \item Given a coherent functor $M:\mathcal{C}\to \mathcal{D}$ we study when a collection of subobjects $(Nx\hookrightarrow Mx)_{x\in \mathcal{C}}$ gives rise to an elementary coherent subfunctor $N\Rightarrow M$. (Theorem \ref{TV})
    \item Following \cite{haykazyan} we define positively closed coherent functors (the positive analogue of existentially closed models). (Definition \ref{poscldef})
    \item Still following \cite{haykazyan} we define types (as prime filters over the subobject lattices $Sub_{\mathcal{C}}(x)$), and prove variants of the omitting types theorem (for coherent functors with various codomains). (Theorem \ref{omitting} and \ref{omitting3})
    \item It turns out that $\lambda $-pretoposes have filtered colimits of size $<\lambda $ which commute with finite limits. (Theorem \ref{hasfiltcolim} and \ref{commfinlim})
    \item We repeat (and slightly extend) the definition of a $(\lambda ,\kappa )$-coherent category and a $(\lambda ,\kappa )$-pretopos from \cite{inflog}. These correspond to $L_{\lambda \kappa }^g$ theories. (Definition \ref{klpretop})
    \item We prove that a $(\lambda ,\kappa )$-pretopos $\mathcal{D}$ (satisfying two extra conditions: it is weakly Boolean and $\{0\}\subseteq Sub_{\mathcal{D}}(1)$ is a $\kappa $-prime filter) is a rich-enough structure to perform $<\kappa $ ultraproduct constructions with $\mathcal{C}\to \mathcal{D}$ coherent functors (for any coherent category $\mathcal{C}$). One might read it as "infinite quantifier theories are so weird that the interpretations of finitary theories inside them are closed under ultraproducts". (Theorem \ref{klpretopult})
    \item The Łoś-lemma holds. (Theorem \ref{los})
    \item We define semantically complete coherent categories (any two models satisfy the same coherent sequents). (Definition \ref{semancplte}) (We refer to this property as $M$ and $N$ are elementary equivalent. In \cite{posmodth} the term "$M$ and $N$ are companions" is used.)
    \item $\mathcal{C}$ is semantically complete iff it is weakly Boolean and 2-valued. (Theorem \ref{intsemancplt})
    \item If $\mathcal{C}$ is semantically complete then $\mathbf{Coh}_e(\mathcal{C},\mathbf{Set})$ (i.e.~$\mathcal{C}\to \mathbf{Set}$ coherent functors and elementary natural transformations) has the joint embedding property. (Theorem \ref{semancplthasjep})
\end{itemize}

There is an equally important second possibility: now that we have a clear algebraic understanding of interpretations between theories (coherent functors) and homomorphisms between interpretations (natural transformations) we can make use of the functoriality of certain model-theoretic constructions, and apply the toolkit of category theory:
\begin{itemize}
    \item Given the subobject functor $Sub_{\mathcal{C}}:\mathcal{C}^{op}\to \mathbf{DLat}$ we study its left Kan-extension along the Yoneda-embedding. Its restriction to coherent functors is denoted by $L_{\mathcal{C}}:\mathbf{Coh}(\mathcal{C},\mathbf{Set})\to \mathbf{DLat}$. (Proposition \ref{lcdef})
    \item This construction is also functorial in $\mathcal{C}$, stated precisely in Definition \ref{ldef}.
    \item We give an explicit description for $LM=L_{\mathcal{C}}M$. $M$ is positively closed iff $LM=2$. (Proposition \ref{lmcomp} and \ref{poscl2})
    \item We give several equivalent conditions for a model $M$ to be positively closed. (Theorem \ref{posclequivalent}) ($1\Leftrightarrow 4\Leftrightarrow 5$ can be found in \cite{posmodth}, for the one-sorted case, in the language of model theory.)
    \item The above left Kan-extension always preserves finite products of coherent functors. (Theorem \ref{lanfinpr})
\end{itemize}


\section{Tarski-Vaught test}

\begin{notation}
$\mathcal{C}$, $\mathcal{D}$, $\mathcal{E}$ are coherent categories, $M$, $M'$, $N$ are coherent functors, $\eta$, $\nu$, $\alpha$ are natural transformations.
\end{notation}

If $\mathcal{C}$ is a pretopos and $\mathcal{E}=\mathbf{Set}$ then the following is Definition 2.~in Lecture 18X of \cite{lurie}:

\begin{definition}
\label{elementarynattr}
We say that a natural transformation
\[
\adjustbox{scale=1.1}{
\begin{tikzcd}
	{\mathcal{C}} && {\mathcal{E}}
	\arrow[""{name=0, anchor=center, inner sep=0}, "N", curve={height=-12pt}, from=1-1, to=1-3]
	\arrow[""{name=1, anchor=center, inner sep=0}, "{M}"', curve={height=12pt}, from=1-1, to=1-3]
	\arrow["\eta", shorten <=3pt, shorten >=3pt, Rightarrow, from=0, to=1]
\end{tikzcd}
}
\]
is elementary, if for each $i:a\hookrightarrow x$ monomorphism in $\mathcal{C}$

\[
\adjustbox{scale=0.95}{
\begin{tikzcd}
	Na && Nx \\
	\\
	{Ma} && {Mx}
	\arrow["Ni", hook, from=1-1, to=1-3]
	\arrow["{\eta _x}", from=1-3, to=3-3]
	\arrow["{Mi}", hook, from=3-1, to=3-3]
	\arrow["{\eta _a}"', from=1-1, to=3-1]
\end{tikzcd}
}
\]
is a pullback. The category of $\mathcal{C}\to \mathcal{E}$ coherent functors and elementary natural transformations is denoted by $\mathbf{Coh}_e(\mathcal{C},\mathcal{E})$.
\end{definition}

\begin{proposition}
If $\eta :N\Rightarrow M$ is an elementary natural transformation, then for each $x\in \mathcal{C}$ the component $\eta _x$ is monic.
\label{mono}
\end{proposition}

\begin{proof}
Let $r, s :A\to N(x)$ be maps for which $\eta _x r =\eta _x s$. We have an induced dashed arrow in

\[
\adjustbox{scale=0.9}{
\begin{tikzcd}
	A \\
	& {M(x)} && {M(x)\times M(x)} \\
	\\
	& {N(x)} && {N(x)\times N(x)}
	\arrow["{M(\Delta )}", from=2-2, to=2-4]
	\arrow["{\eta _{x\times x}}", from=2-4, to=4-4]
	\arrow["{\eta _x}"', from=2-2, to=4-2]
	\arrow["{N(\Delta )}"', from=4-2, to=4-4]
	\arrow["{\langle \alpha ,\beta \rangle}", curve={height=-6pt}, from=1-1, to=2-4]
	\arrow["{\eta _x\circ \alpha}"', curve={height=6pt}, from=1-1, to=4-2]
	\arrow[dashed, from=1-1, to=2-2]
\end{tikzcd}
}
\]
from which $r =s $ follows.
\end{proof}

\begin{definition}
Let $M:\mathcal{C}\to \mathcal{E}$ be a coherent functor and $(\alpha_x: N(x) \hookrightarrow M(x))_{x\in \mathcal{C}}$ be a collection of subobjects. We say that it satisfies the Tarski-Vaught criterion, if for each $i:\varphi \to x\times y$ monomorphism of $\mathcal{C}$ the map $f\circ j$ in
\[
\adjustbox{scale=0.9}{
\begin{tikzcd}
	{M(\varphi )} && {M(x)\times M(y)} \\
	\\
	s && {M(x)\times N(y)} && {N(y)} \\
	&&& p \\
	q && {N(x)\times N(y)}
	\arrow["{M(i)}"{description}, hook, from=1-1, to=1-3]
	\arrow["{1\times \alpha _y}"{description}, hook, from=3-3, to=1-3]
	\arrow[hook, from=3-1, to=1-1]
	\arrow[hook, from=3-1, to=3-3]
	\arrow[from=3-3, to=3-5]
	\arrow["{\alpha _x\times 1}"{description}, hook, from=5-3, to=3-3]
	\arrow[hook, from=5-1, to=5-3]
	\arrow["j"{description}, hook, from=5-1, to=3-1]
	\arrow["\ulcorner"{anchor=center, pos=0.125, rotate=-90}, draw=none, from=1-3, to=3-1]
	\arrow["\ulcorner"{anchor=center, pos=0.125, rotate=-90}, shift right=2, draw=none, from=3-3, to=5-1]
	\arrow["f"{description, pos=0.3}, two heads, from=3-1, to=4-4]
	\arrow[hook, from=4-4, to=3-5]
	\arrow[two heads, from=5-1, to=4-4]
\end{tikzcd}
}
\]
is surjective.
\label{TVdef}
\end{definition}

\begin{proposition}
If $(\alpha_x: N(x) \hookrightarrow M(x))_{x\in \mathcal{C}}$ satisfies the Tarski-Vaught criterion then $\alpha : N\Rightarrow M$ is a subfunctor. 
\end{proposition}

\begin{proof}
We use the terminology of \cite{makkai}. We have the following diagram:

\[\begin{tikzcd}
	{M(x)} && {M(x)\times M(y)} && {M(y)} \\
	\\
	s && {N(x)\times M(y)} \\
	\\
	q && {N(x)\times N(y)}
	\arrow["{M(f)\circ p_1}", shift left=2, from=1-3, to=1-5]
	\arrow["{p_2}"', shift right=2, from=1-3, to=1-5]
	\arrow["{\langle 1,M(f)\rangle}", hook, from=1-1, to=1-3]
	\arrow["{\alpha _x \times 1}", hook, from=3-3, to=1-3]
	\arrow[hook, from=3-1, to=3-3]
	\arrow[hook, from=3-1, to=1-1]
	\arrow["\ulcorner"{anchor=center, pos=0.125, rotate=-90}, draw=none, from=1-3, to=3-1]
	\arrow["{1\times \alpha _y}", hook, from=5-3, to=3-3]
	\arrow[hook, from=5-1, to=5-3]
	\arrow[hook, from=5-1, to=3-1]
	\arrow["\ulcorner"{anchor=center, pos=0.125, rotate=-90}, draw=none, from=3-3, to=5-1]
\end{tikzcd}\]

It would be enough to see that the subobject $q\hookrightarrow N(x)\times N(y)$ is functional (from $N(x)$ to $N(y)$), as in this case we had a commutative square
\[\begin{tikzcd}
	{graph(M(f))} && {M(x)\times M(y)} \\
	\\
	{graph (f')} && {N(x)\times N(y)}
	\arrow[hook, from=1-1, to=1-3]
	\arrow["{\alpha _x \times \alpha _y}"', from=3-3, to=1-3]
	\arrow[hook, from=3-1, to=3-3]
	\arrow[from=3-1, to=1-1]
\end{tikzcd}\]
from which the commutativity of
\[\begin{tikzcd}
	{M(x)} && {M(y)} \\
	\\
	{N(x)} && {N(y)}
	\arrow["{M(f)}", from=1-1, to=1-3]
	\arrow["{f'}", from=3-1, to=3-3]
	\arrow["{\alpha _y}"', hook, from=3-3, to=1-3]
	\arrow["{\alpha _x}"', hook, from=3-1, to=1-1]
\end{tikzcd}\]
follows.

$s\hookrightarrow N(x)\times M(y)$ is functional, as it is the equalizer of $M(f)\circ \alpha _x \circ p_1$ and $p_2: N(x)\times M(y)\to M(y)$. We need to conclude the functionality of $q$, which is, that $Th(\mathcal{E})$ proves $q(x,y)\wedge q(x,y')\Rightarrow y\approx y'$ and $\top \Rightarrow \exists y q(x,y)$. With diagrams it is verified as follows:

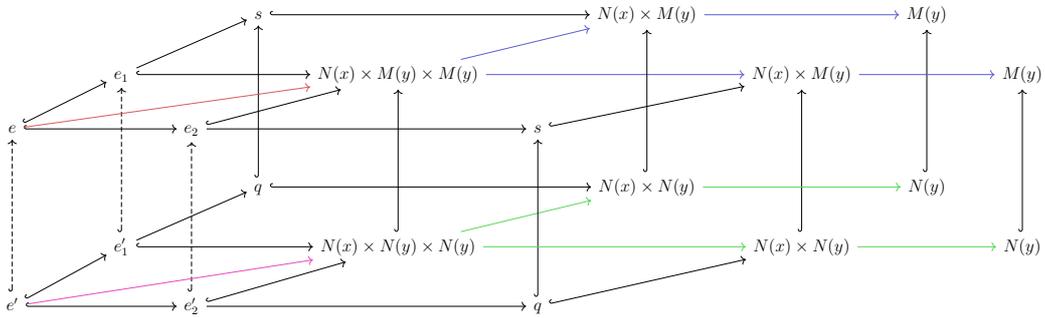
\[
\adjustbox{width=\textwidth}{
\begin{tikzcd}
	&&&& s &&& {N(x)\times M(y)} && {M(y)} \\
	&& {e_1} &&& {N(x)\times M(y)\times M(y)} &&& {N(x)\times M(y)} && {M(y)} \\
	e &&& {e_2} &&& s \\
	&&&& q &&& {N(x)\times N(y)} && {N(y)} \\
	&& {e_1'} &&& {N(x)\times N(y)\times N(y)} &&& {N(x)\times N(y)} && {N(y)} \\
	{e'} &&& {e_2'} &&& q
	\arrow[hook, from=3-1, to=2-3]
	\arrow[from=2-3, to=1-5]
	\arrow[hook, from=1-5, to=1-8]
	\arrow[color={rgb,255:red,92;green,92;blue,214}, from=2-6, to=1-8]
	\arrow[hook, from=2-3, to=2-6]
	\arrow[hook, from=3-4, to=2-6]
	\arrow[hook, from=3-1, to=3-4]
	\arrow[color={rgb,255:red,92;green,92;blue,214}, from=2-6, to=2-9]
	\arrow[hook, from=3-7, to=2-9]
	\arrow[from=3-4, to=3-7]
	\arrow[dashed, hook, from=6-1, to=3-1]
	\arrow[hook, from=6-1, to=6-4]
	\arrow[hook, from=6-1, to=5-3]
	\arrow[hook, from=5-3, to=5-6]
	\arrow[hook, from=6-4, to=5-6]
	\arrow[dashed, hook, from=5-3, to=2-3]
	\arrow[dashed, hook, from=6-4, to=3-4]
	\arrow[hook, from=5-6, to=2-6]
	\arrow[from=5-3, to=4-5]
	\arrow[hook, from=4-5, to=1-5]
	\arrow[hook, from=4-5, to=4-8]
	\arrow[color={rgb,255:red,92;green,214;blue,92}, from=5-6, to=4-8]
	\arrow[hook, from=4-8, to=1-8]
	\arrow[hook, from=6-7, to=3-7]
	\arrow[from=6-4, to=6-7]
	\arrow[color={rgb,255:red,92;green,214;blue,92}, from=5-6, to=5-9]
	\arrow[hook, from=6-7, to=5-9]
	\arrow[hook, from=5-9, to=2-9]
	\arrow[color={rgb,255:red,92;green,214;blue,92}, from=4-8, to=4-10]
	\arrow[color={rgb,255:red,92;green,92;blue,214}, from=1-8, to=1-10]
	\arrow[hook, from=4-10, to=1-10]
	\arrow[color={rgb,255:red,92;green,214;blue,92}, from=5-9, to=5-11]
	\arrow[color={rgb,255:red,92;green,92;blue,214}, from=2-9, to=2-11]
	\arrow[hook, from=5-11, to=2-11]
	\arrow[color={rgb,255:red,214;green,92;blue,92}, from=3-1, to=2-6]
	\arrow[color={rgb,255:red,232;green,74;blue,190}, from=6-1, to=5-6]
\end{tikzcd}
}
\]
As the red map equalizes the blue arrows, the pink map equalizes the green ones. The second axiom follows, as in

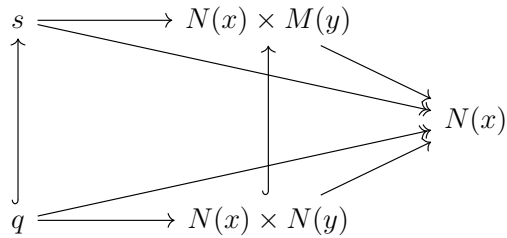
\[
\adjustbox{scale=0.9}{
\begin{tikzcd}
	s && {N(x)\times M(y)} \\
	&&& {N(x)} \\
	q && {N(x)\times N(y)}
	\arrow[hook, from=3-1, to=1-1]
	\arrow[from=1-1, to=1-3]
	\arrow[hook, from=3-3, to=1-3]
	\arrow[from=3-1, to=3-3]
	\arrow[from=1-3, to=2-4]
	\arrow[from=3-3, to=2-4]
	\arrow[two heads, from=1-1, to=2-4]
	\arrow[two heads, from=3-1, to=2-4]
\end{tikzcd}
}
\]
if $s\to N(x)$ is surjective then $q\to N(x)$ is also, by our assumption.
\end{proof}

\begin{proposition}
If the subfunctor $\alpha : N\Rightarrow M$ satisfies the Tarski-Vaught test then it is elementary.
\end{proposition}

\begin{proof}
The map $s\twoheadrightarrow p$ is the pullback of $1$, hence it is an effective epi (iso). By assumption this implies that the dashed arrow is eff.~epi, and since it is also monic, it is an isomorphism.
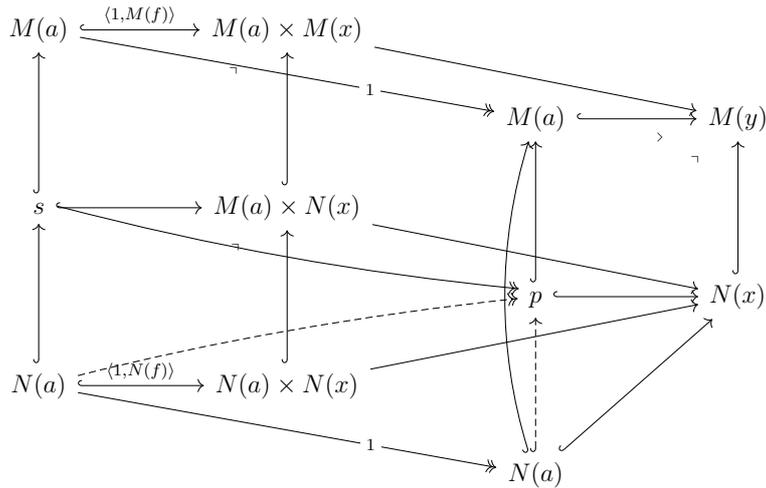
\[
\adjustbox{scale=0.8}{
\begin{tikzcd}
	{M(a)} && {M(a)\times M(x)} \\
	&&&& {M(a)} && {M(y)} \\
	s && {M(a)\times N(x)} \\
	&&&& p && {N(x)} \\
	{N(a)} && {N(a)\times N(x)} \\
	&&&& {N(a)} && {}
	\arrow["{\langle 1, M(f)\rangle}", hook, from=1-1, to=1-3]
	\arrow[hook, from=5-3, to=3-3]
	\arrow[hook, from=3-3, to=1-3]
	\arrow[hook, from=3-1, to=1-1]
	\arrow[hook, from=3-1, to=3-3]
	\arrow[hook, from=5-1, to=3-1]
	\arrow["{\langle 1, N(f)\rangle}", hook, from=5-1, to=5-3]
	\arrow["\ulcorner"{anchor=center, pos=0.125, rotate=-90}, draw=none, from=1-3, to=3-1]
	\arrow["\ulcorner"{anchor=center, pos=0.125, rotate=-90}, draw=none, from=3-3, to=5-1]
	\arrow["1"{description, pos=0.7}, two heads, from=1-1, to=2-5]
	\arrow[from=1-3, to=2-7]
	\arrow[hook, from=2-5, to=2-7]
	\arrow[hook, from=4-7, to=2-7]
	\arrow[hook, from=4-5, to=2-5]
	\arrow[hook, from=4-5, to=4-7]
	\arrow["\ulcorner"{anchor=center, pos=0.125, rotate=-90}, draw=none, from=2-7, to=4-5]
	\arrow[from=3-3, to=4-7]
	\arrow[from=5-3, to=4-7]
	\arrow[shift left=1, curve={height=6pt}, two heads, from=3-1, to=4-5]
	\arrow["\ulcorner"{anchor=center, pos=0.125, rotate=-135}, draw=none, from=2-7, to=3-3]
	\arrow[curve={height=-18pt}, hook, from=6-5, to=2-5]
	\arrow[hook, from=6-5, to=4-7]
	\arrow["1"{description, pos=0.7}, two heads, from=5-1, to=6-5]
	\arrow[dashed, hook, from=6-5, to=4-5]
	\arrow[curve={height=-6pt}, dashed, two heads, from=5-1, to=4-5]
\end{tikzcd}
}
\qedhere
\]
\end{proof}

\begin{definition}
The collection $(\alpha _x:N(x)\hookrightarrow M(x))_{x\in \mathcal{C}}$ preserves finite products if $\alpha _1: N(1)\hookrightarrow M(1)=1_{\mathcal{D}}$ is an isomorphism and there is an isomorphism $\varphi $ making

\[\begin{tikzcd}
	{M(x)\times M(y)} && {M(x\times y)} \\
	\\
	{N(x)\times N(y)} && {N(x\times y)}
	\arrow["{\alpha _{x\times y}}"', hook, from=3-3, to=1-3]
	\arrow["\varphi", dashed, from=3-3, to=3-1]
	\arrow["\cong"', from=1-3, to=1-1]
	\arrow["{\alpha _x \times \alpha _y}", hook, from=3-1, to=1-1]
\end{tikzcd}\]
commute.
\label{presprod}
\end{definition}

\begin{theorem}
Let $M:\mathcal{C}\to \mathcal{D}$ be a coherent functor and let $(\alpha _x:N(x)\hookrightarrow M(x))_{x\in \mathcal{C}}$ be a collection of subobjects. The following are equivalent:
\begin{enumerate}
    \item $(\alpha _x)_{x\in \mathcal{C}}$ satisfies the Tarski-Vaught criterion and preserves finite products,
    \item it extends (uniquely) to a coherent elementary subfunctor $\alpha :N \Rightarrow M$.
\end{enumerate}
\label{TV}
\end{theorem}

\begin{proof}
$1\to 2$. It remains to prove that $N$ is coherent. As 
\[
\adjustbox{scale=0.8}{
\begin{tikzcd}
	{M(x)} && {M(x)\times M(y)} && {M(x\times y)} \\
	\\
	{N(x)} && {N(x)\times N(y)} && {N(x\times y)}
	\arrow["{\alpha _{x\times y}}"', hook, from=3-5, to=1-5]
	\arrow["\varphi", dashed, from=3-5, to=3-3]
	\arrow[from=3-3, to=3-1]
	\arrow["{\alpha _x}", hook, from=3-1, to=1-1]
	\arrow["\cong"', from=1-5, to=1-3]
	\arrow["{\alpha _x \times \alpha _y}", hook, from=3-3, to=1-3]
	\arrow[from=1-3, to=1-1]
\end{tikzcd}
}
\]
commutes and $N(p_1)$ is uniquely determined, the preservation of finite products follows. As $\alpha $ is elementary $N(i)$ in
\[
\adjustbox{scale=0.85}{
\begin{tikzcd}
	{M(a)} && {M(x)} && {M(y)} \\
	\\
	{N(a)} && {N(x)} && {N(y)}
	\arrow["{M(i)}", hook, from=1-1, to=1-3]
	\arrow["{M(f)}", shift left=1, from=1-3, to=1-5]
	\arrow["{M(g)}"', shift right=1, from=1-3, to=1-5]
	\arrow["{\alpha _a}", hook, from=3-1, to=1-1]
	\arrow["{\alpha _x}", hook, from=3-3, to=1-3]
	\arrow["{N(i)}", hook, from=3-1, to=3-3]
	\arrow["\ulcorner"{anchor=center, pos=0.125, rotate=-90}, draw=none, from=1-3, to=3-1]
	\arrow["{N(f)}", shift left=1, from=3-3, to=3-5]
	\arrow["{N(g)}"', shift right=1, from=3-3, to=3-5]
	\arrow["{\alpha _y}"', hook, from=3-5, to=1-5]
\end{tikzcd}
}
\]
is an equalizer hence $N$ preserves finite limits.

The preservation of unions is immediate, as unions are pullback-stable in $\mathcal{D}$. Finally let $f: x\to y$ be an effective epimorphism in $\mathcal{C}$. In

\[
\adjustbox{scale=0.85}{
\begin{tikzcd}
	{M(x)} && {M(x)\times M(y)} && {M(y)} \\
	\\
	s && {M(x)\times N(y)} && {N(y)} \\
	\\
	{N(x)} && {N(x)\times N(y)}
	\arrow["{\langle 1,M(f)\rangle }", hook, from=1-1, to=1-3]
	\arrow["{1\times \alpha _y}"', hook, from=3-3, to=1-3]
	\arrow[hook, from=3-1, to=3-3]
	\arrow[hook, from=3-1, to=1-1]
	\arrow["{\alpha _x\times 1}"', hook, from=5-3, to=3-3]
	\arrow["{\langle 1, N(f) \rangle}"', hook, from=5-1, to=5-3]
	\arrow[hook, from=5-1, to=3-1]
	\arrow[from=1-3, to=1-5]
	\arrow["{\alpha _y}"', hook, from=3-5, to=1-5]
	\arrow[from=3-3, to=3-5]
	\arrow["\ulcorner"{anchor=center, pos=0.125, rotate=-90}, draw=none, from=1-3, to=3-1]
	\arrow["\ulcorner"{anchor=center, pos=0.125, rotate=-90}, draw=none, from=1-5, to=3-3]
	\arrow["\ulcorner"{anchor=center, pos=0.125, rotate=-90}, draw=none, from=3-3, to=5-1]
\end{tikzcd}
}
\]
every square is a pullback. By pullback-stability $s\to N(y)$ is an effective epi, then by assumption $N(f):N(x)\to N(y)$ is also.

$2\to 1$. By the universal property of $M(x)\times M(y)$, the inner square in

\[
\adjustbox{scale=0.9}{
\begin{tikzcd}
	& {M(x)\times M(y)} \\
	{M(x)} && {M(x\times y)} && {M(y)} \\
	\\
	& {N(x)\times N(y)} \\
	{N(x)} && {N(x\times y)} && {N(y)}
	\arrow["{\alpha _x \times \alpha _y}"{description}, from=4-2, to=1-2]
	\arrow["\cong"{description}, dashed, from=2-3, to=1-2]
	\arrow["{\alpha _{x\times y}}"{description}, from=5-3, to=2-3]
	\arrow["\cong"{description}, dashed, from=5-3, to=4-2]
	\arrow[curve={height=6pt}, from=1-2, to=2-1]
	\arrow[curve={height=-6pt}, from=1-2, to=2-5]
	\arrow["{M(p_2)}"{description}, from=2-3, to=2-5]
	\arrow["{M(p_1)}"{description, pos=0.7}, from=2-3, to=2-1]
	\arrow["{N(p_2)}"{description}, from=5-3, to=5-5]
	\arrow["{N(p_1)}"{description}, from=5-3, to=5-1]
	\arrow[curve={height=6pt}, from=4-2, to=5-1]
	\arrow[curve={height=-6pt}, from=4-2, to=5-5]
	\arrow["{\alpha _y}"{description}, hook, from=5-5, to=2-5]
	\arrow["{\alpha _x}"{description}, from=5-1, to=2-1]
\end{tikzcd}
}
\]
commutes, hence $(\alpha _x)_{x\in \mathcal{C}}$ preserves products in the sense of Definition \ref{presprod}.

For the Tarski-Vaught criterion observe
\[
\adjustbox{scale=0.9}{
\begin{tikzcd}
	{M(\varphi)} && {M(x)\times M(y)} \\
	&&& {M(\exists x\varphi )} && {M(y)} \\
	s && {M(x)\times N(y)} \\
	&&& {N(\exists x\varphi )} && {N(y)} \\
	{N(\varphi )} && {N(x)\times N(y)}
	\arrow["{M(i)}"{description}, hook, from=1-1, to=1-3]
	\arrow[from=1-3, to=2-6]
	\arrow["{1\times \alpha _y}", hook, from=3-3, to=1-3]
	\arrow[hook, from=3-1, to=3-3]
	\arrow[hook, from=3-1, to=1-1]
	\arrow[from=3-3, to=4-6]
	\arrow["{\alpha _y}"{description}, hook, from=4-6, to=2-6]
	\arrow[two heads, from=1-1, to=2-4]
	\arrow["{M(j)}"{description}, hook, from=2-4, to=2-6]
	\arrow[from=4-4, to=2-4]
	\arrow["{N(j)}"{description}, hook, from=4-4, to=4-6]
	\arrow[two heads, from=3-1, to=4-4]
	\arrow["{\alpha _x \times 1}", hook, from=5-3, to=3-3]
	\arrow["{N(i)}"{description}, hook, from=5-1, to=5-3]
	\arrow[hook, from=5-1, to=3-1]
	\arrow[two heads, from=5-1, to=4-4]
\end{tikzcd}
}
\]
By the pullback-stability of effective epi-mono factorizations $s \twoheadrightarrow N(\exists x \varphi ) \xhookrightarrow{N(j)} N(y)$ is a factorization for $s\to M(x)\times N(y)\to N(y)$.
\end{proof}

The extra assumption on "preserving finite products" is needed, e.g.~the collection $(N(x)=\emptyset \hookrightarrow M(x))_x$ trivially satisfies the Tarski-Vaught condition. The reason for this is that we only considered subobjects $\varphi \hookrightarrow x\times y$ secretly assuming that if $y=y_0\times \dots y_k$ then $N$ will preserve this product. We can modify the definition to get:

\begin{definition}
Let $\mathcal{C}, \mathcal{D}$ be coherent categories, $M:\mathcal{C}\to \mathcal{D}$ be a coherent functor and $(Nx\xhookrightarrow{\alpha _x} Mx)_{x\in \mathcal{C}}$ be a collection of subobjects. We say that it satisfies the Tarski-Vaught test in the stronger sense if for any subobject $\varphi \hookrightarrow x_0\times \dots \times x_{n-1}\times y_0\times \dots y_{m-1}$ we have that $f\circ j$ in

\[
\adjustbox{width=\textwidth}{
\begin{tikzcd}
	M\varphi && {Mx_0\times \dots \times Mx_{n-1}\times My_0\times \dots My_{m-1}} \\
	\\
	s && {Mx_0\times \dots \times Mx_{n-1}\times Ny_0\times \dots Ny_{m-1}} && {Ny_0\times \dots Ny_{m-1}} \\
	&&&& \bullet \\
	q && {Nx_0\times \dots \times Nx_{n-1}\times Ny_0\times \dots Ny_{m-1}}
	\arrow[hook, from=1-1, to=1-3]
	\arrow["{id\times \alpha _{y_0}\times \dots \times \alpha _{y_{m-1}}}"{description}, hook, from=3-3, to=1-3]
	\arrow["{\alpha _{x_0}\times \dots \alpha _{x_{n-1}}\times id}"{description}, hook, from=5-3, to=3-3]
	\arrow[from=3-3, to=3-5]
	\arrow[hook, from=3-1, to=1-1]
	\arrow[hook, from=3-1, to=3-3]
	\arrow[hook, from=5-1, to=5-3]
	\arrow["j"{description}, hook, from=5-1, to=3-1]
	\arrow["\lrcorner"{anchor=center, pos=0.125, rotate=-90}, draw=none, from=3-3, to=5-1]
	\arrow["\lrcorner"{anchor=center, pos=0.125, rotate=-90}, draw=none, from=1-3, to=3-1]
	\arrow["f"{description, pos=0.8}, two heads, from=3-1, to=4-5]
	\arrow[hook, from=4-5, to=3-5]
\end{tikzcd}
}
\]
is effective epi (for $n>0, m\geq 0$, where we define $\prod _{\emptyset }Ny_i \to  \prod _{\emptyset }My_i $ to be $1\to 1$).
\end{definition}

\begin{theorem}
Let $M:\mathcal{C}\to \mathcal{D}$ be a coherent functor and $(\alpha _x:Nx\hookrightarrow Mx)_{x\in \mathcal{C}}$ be a collection of subobjects. The following are equivalent:
\begin{enumerate}
    \item It satisfies the Tarski-Vaught test in the stronger sense.
    \item It extends uniquely to an elementary coherent subfunctor $\alpha :N\Rightarrow M$.
\end{enumerate}
\end{theorem}

\begin{proof}
$2\to 1$ is clear from Theorem $\ref{TV}$, as the component at $x_0\times \dots \times x_{n-1}$ must be $\alpha _{x_0}\times \dots \times \alpha _{x_{n-1}}$. Assuming $1$ we shall prove that the collection preserves products.

Terminal object: applying the condition to $1\to 1$ we get that in
\[\begin{tikzcd}
	{M(1)} && {M(1)\times 1} && 1 \\
	{N(1)} && {N(1)\times 1}
	\arrow[from=1-1, to=1-3]
	\arrow[from=1-3, to=1-5]
	\arrow["{\alpha _1}"', hook, from=2-3, to=1-3]
	\arrow["{\alpha _1}"', hook, from=2-1, to=1-1]
	\arrow[from=2-1, to=2-3]
\end{tikzcd}\]
$\alpha _1$ is effective epi, hence iso.

Binary products: taking $x\times y \xhookrightarrow{ \Delta } (x\times y)\times x\times y$ yields that $q\to Nx\times Ny$ in 
\[\begin{tikzcd}
	{M(x\times y)} && {M(x\times y)\times Mx\times My} && {Mx\times My} \\
	\\
	{Nx\times Ny} && {M(x\times y)\times Nx\times Ny} && {Nx\times Ny} \\
	\\
	q && {N(x\times y)\times Nx\times Ny}
	\arrow[from=1-1, to=1-3]
	\arrow[from=1-3, to=1-5]
	\arrow["\cong"{description}, curve={height=-12pt}, from=1-1, to=1-5]
	\arrow["{id\times \alpha _x\times \alpha _y}"{description}, hook, from=3-3, to=1-3]
	\arrow["{\alpha _x\times \alpha _y}"{description}, hook, from=3-5, to=1-5]
	\arrow[from=3-3, to=3-5]
	\arrow["\lrcorner"{anchor=center, pos=0.125, rotate=-90}, draw=none, from=1-5, to=3-3]
	\arrow[from=3-1, to=3-3]
	\arrow["{\alpha _x\times \alpha _y}"{description}, hook, from=3-1, to=1-1]
	\arrow["\lrcorner"{anchor=center, pos=0.125, rotate=-90}, draw=none, from=1-3, to=3-1]
	\arrow["{\alpha _{x\times y}\times id}"{description}, hook, from=5-3, to=3-3]
	\arrow["\cong"{description}, hook, from=5-1, to=3-1]
	\arrow[from=5-1, to=5-3]
\end{tikzcd}\]
is iso, and therefore $Nx\times Ny\leq N(x\times y)$:

\[\begin{tikzcd}
	& {M(x\times y)} && {M(x\times y)\times Mx\times My} \\
	\\
	& {N(x\times y)} && {N(x\times y)\times Mx\times My} \\
	{Nx\times Ny} && q & {N(x\times y)\times Nx \times Ny}
	\arrow["{\alpha _{x\times y}\times id}", hook, from=3-4, to=1-4]
	\arrow["{\alpha _{x\times y}}"', hook, from=3-2, to=1-2]
	\arrow[from=1-2, to=1-4]
	\arrow[from=3-2, to=3-4]
	\arrow["\lrcorner"{anchor=center, pos=0.125, rotate=-90}, draw=none, from=1-4, to=3-2]
	\arrow["{\alpha _x\times \alpha _y}", hook, from=4-1, to=1-2]
	\arrow[dashed, from=4-1, to=3-2]
	\arrow[from=4-1, to=4-3]
	\arrow[from=4-3, to=4-4]
	\arrow[from=4-4, to=3-4]
\end{tikzcd}\]
The symmetric argument completes the proof.
\end{proof}

In model theory one does not specify subsets of each definable set to describe a substructure, instead a single subset of the underlying set is used. As our setting is equivalent to the classical setup we should be able to do the same, via identifying the sorts among general objects (i.e.~the objects $[x\approx x]$ if $\mathcal{C}$ is a syntactic category). This tool will be used later as well.

\begin{definition}
Let $\mathcal{C}$ be a coherent category. A filtration of $\mathcal{C}$ is a set $\mathbb{S}\subseteq Ob(\mathcal{C})$ together with a specified monomorphism $\iota _x: x\hookrightarrow s_{x,1}\times \dots s_{x,n_x}$ (where $s_i\in \mathbb{S}$) for every object $x$ of $\mathcal{C}$, such that given any monomorphism $a\xhookrightarrow{j}s_1\times \dots s_n$ we have a commutative triangle

\[\begin{tikzcd}
	a &&& {s_1\times \dots \times s_n} \\
	\\
	{a'}
	\arrow["j", hook, from=1-1, to=1-4]
	\arrow["\cong"', from=1-1, to=3-1]
	\arrow["{\iota _{a'}}"', hook, from=3-1, to=1-4]
\end{tikzcd}\]
(More precisely: $\{s_1,\dots s_n\}=\{s_{a',1},\dots s_{a',n_{a'}}\}$ and the triangle is a square with the induced isomorphism on the right side.) Here we allow $n_x=0$, i.e.~to choose $x\xhookrightarrow{\iota _x} 1$.
\end{definition}

\begin{remark}
In other terms for each object $a$ we choose finitely many objects $s_1,\dots s_n\in \mathbb{S}$ and maps $\iota _{a,k}:a\to s_k$ satisfying the above conditions, that do not depend on which product object $s_1\times \dots \times s_n$ we use.
\end{remark}

\begin{proposition}
Every coherent category admits a filtration.
\end{proposition}

\begin{proof}
This property is invariant under equivalence, indeed if $F:\mathcal{C}\to \mathcal{D}$ is an equivalence and $(\mathbb{S},(\iota _x)_x)$ is a filtration on $\mathcal{C}$ then we can take $\mathbb{S}'=F[\mathbb{S}]$ and if $u\in \mathcal{D}$ equals $Fx$ for some $x\in \mathcal{C}$ then take $\iota '_u=F(\iota _x)$ otherwise take $\iota '_u:u\cong F(x)\xrightarrow{F\iota _x} F(s_1)\times \dots \times F(s_n)$ with an arbitrary isomorphism $u\cong Fx$. This gives a filtration on $\mathcal{D}$.

From \cite{makkai} we know that every coherent category is the syntactic category of something. But then we can take $\mathbb{S}$ to be the set $\{ [x=x] : X \text{ is a sort}\}$ and for any $[\varphi (x_1,\dots x_n)]$ we take
\begin{multline*}
    \iota _{[\varphi ]}:[\varphi (x_1,\dots x_n)]\xrightarrow{[\varphi (x_1,\dots x_n)\wedge x_1\approx x_1'\wedge \dots x_n\approx x_n']} [x_1'\approx x_1' \wedge \dots \wedge x_n'\approx x_n']=\\ [x_1'\approx x_1']\times \dots \times [x_n'\approx x_n']
\end{multline*}
(of course $=$ is the induced isomorphism between the two product objects). In particular when $\varphi $ is closed we take $[\varphi]\xrightarrow{[\varphi ]}[\top]$. The condition is satisfied, if $[\varphi (\vec{x})]\xrightarrow{[\mu (\vec{x},\vec{y})] } [\vec{y}\approx\vec{y}]$ is a monomorphism, then we can factor it through $[\exists \vec{x} \mu (\vec{x},\vec{y})]$ which is a canonical subobject of $[\vec{y}\approx \vec{y}]$, as required. (This also works when $\vec{y}$ is the empty sequence, i.e.~when we are given a monomorphism $[\varphi (x)]\xrightarrow{[\mu (x)]} [\top ]$, which is the same as $[\varphi (x)]\xrightarrow{[\varphi (x)]} [\top ]$ as $[\top ]$ is terminal.)
\end{proof}

\begin{proposition}
\label{filtration}
Let $\mathcal{C}$ be a coherent category and $\mathbb{S}\subseteq Ob(\mathcal{C})$, $(\iota _x: x\hookrightarrow s_{x,1}\times \dots \times s_{x,n_x})_{x\in \mathcal{C}}$ be a filtration. Let $N,M:\mathcal{C}\to \mathcal{D}$ be coherent functors and $(\alpha _s: Ns\to Ms)_{s\in \mathbb{S}}$ be a collection of arrows in $\mathcal{D}$. If for any $x\in \mathcal{C}$ the dashed arrow in

\[\begin{tikzcd}
	Nx &&& {Ns_{x,1}\times \dots \times Ns_{x,n_x}} \\
	\\
	Mx &&& {Ms_{x,1}\times \dots \times Ms_{x,n_x}}
	\arrow["{M\iota _x}", hook, from=1-1, to=1-4]
	\arrow["{\alpha _x}"', dashed, from=1-1, to=3-1]
	\arrow["{M\iota _x}"', hook, from=3-1, to=3-4]
	\arrow["{\alpha _{s_{x,1}}\times \dots \times \alpha _{s_{x,n_x}}}", from=1-4, to=3-4]
\end{tikzcd}\]
exists then $\alpha =(\alpha _x)_x$ is a natural transformation. If moreover these squares are pullbacks then $\alpha $ is elementary.
\end{proposition}

\begin{proof}
In
\[
\adjustbox{width=\textwidth}{
\begin{tikzcd}
	Nx && {Nx\times Ny} && {(Ns_1\times \dots \times Ns_n)\times (Ns_1'\times \dots \times Ns_m')} \\
	Na \\
	Ma \\
	Mx && {Mx\times My} && {(Ms_1\times \dots \times Ms_n)\times (Ms_1'\times \dots \times Ms_m')}
	\arrow["{N(\langle 1,f\rangle)}", hook, from=1-1, to=1-3]
	\arrow["{N(\iota _x \times \iota _y)}", hook, from=1-3, to=1-5]
	\arrow["\cong"', from=1-1, to=2-1]
	\arrow["{N\iota _a}"{description, pos=0.7}, hook, from=2-1, to=1-5]
	\arrow["\cong"', from=4-1, to=3-1]
	\arrow["{M\iota _a}"{description, pos=0.7}, hook, from=3-1, to=4-5]
	\arrow["{\prod \alpha _{s_i}\times \prod {\alpha _{s_{j}'}}}", from=1-5, to=4-5]
	\arrow[from=2-1, to=3-1]
	\arrow["{\alpha _x\times \alpha _y}"{description}, from=1-3, to=4-3]
	\arrow["{M(\iota _x \times \iota _y)}"', hook, from=4-3, to=4-5]
	\arrow["{M(\langle 1,f\rangle )}"', hook, from=4-1, to=4-3]
\end{tikzcd}
}
\]
the commutativity of the rectangle and that of the right square make the left square commute (as $M(\iota _x \times \iota _y)$ is mono). Taking the analogous diagram with the monomorphisms $N(f)$ and $M(f)$ (instead of their graphs), and using the pasting law for pullbacks yields the other half of the claim. 
\end{proof}

\begin{theorem}
Let $\mathcal{C}$ be coherent category with filtration $(\mathbb{S},(\iota _x)_x)$ and let $M:\mathcal{C}\to \mathcal{D}$ be a coherent functor. For a collection of subobjects $(\alpha _s: N(s)\hookrightarrow M(s))_{s\in \mathbb{S}}$ the following are equivalent:
\begin{enumerate}
    \item $(\alpha _s)_s$ satisfies the Tarski-Vaught condition, i.e.~given $\varphi \xhookrightarrow{i} (s_1\times \dots \times s_n)\times (s_1'\times \dots \times s_m')$ the map $f\circ j$ in
\[
\adjustbox{scale=0.8}{
\begin{tikzcd}
	{M(\varphi )} && {Ms_1\times \dots \times Ms_n\times Ms_1'\times \dots Ms_m'} \\
	\\
	p && {Ms_1\times \dots \times Ms_n\times Ns_1'\times \dots Ns_m'} && {Ns_1'\times \dots Ns_m'} \\
	&&& \bullet \\
	q && {Ns_1\times \dots \times Ns_n\times Ns_1'\times \dots Ns_m'}
	\arrow["{M(i)}", hook, from=1-1, to=1-3]
	\arrow["{1\times \dots \times 1\times \alpha _{s_1'}\times \dots \times \alpha _{s_m'}}"{description}, hook, from=3-3, to=1-3]
	\arrow[hook, from=3-1, to=1-1]
	\arrow[hook, from=3-1, to=3-3]
	\arrow[from=3-3, to=3-5]
	\arrow["{\alpha _{s_1}\times \dots \times \alpha _{s_n}\times 1\times \dots \times 1}"{description}, hook, from=5-3, to=3-3]
	\arrow[hook, from=5-1, to=5-3]
	\arrow["j"{description}, hook, from=5-1, to=3-1]
	\arrow["\ulcorner"{anchor=center, pos=0.125, rotate=-90}, draw=none, from=1-3, to=3-1]
	\arrow["\ulcorner"{anchor=center, pos=0.125, rotate=-90}, shift right=2, draw=none, from=3-3, to=5-1]
	\arrow["f"{description, pos=0.3}, two heads, from=3-1, to=4-4]
	\arrow[hook, from=4-4, to=3-5]
	\arrow[two heads, from=5-1, to=4-4]
\end{tikzcd}
}
\]
is surjective. (For $n>0, m\geq 0$. Note that if $n=0$ the condition trivially holds.)
    \item There is a unique elementary coherent subfunctor $\alpha :N\Rightarrow M$ with component $\alpha _s$ at $s\in \mathbb{S}$. (Unique as a subobject of $M$ in $\mathcal{D}^{\mathcal{C}}$.)
\end{enumerate}
\label{TV1}
\end{theorem}

\begin{proof}
$2\to 1$ is clear. It is also immediate that if $(\alpha _s)_s$ extends to a coherent elementary subfunctor then this extension is unique, and the components are calculated by the pullbacks 
\[\begin{tikzcd}
	Ma && {Ms_{a,1}\times \dots \times Ms_{a,n_a}} \\
	\\
	Na && {Ns_{a,1}\times \dots \times Ns_{a,n_a}}
	\arrow["{M\iota _a}", hook, from=1-1, to=1-3]
	\arrow["{\alpha _{s_{a,1}}\times \dots \times \alpha _{s_{a,n_a}}}"', hook, from=3-3, to=1-3]
	\arrow["{N\iota _a}"', hook, from=3-1, to=3-3]
	\arrow["{\alpha_a}"', hook, from=3-1, to=1-1]
\end{tikzcd}\]
(when $n_a=0$ this means $\alpha _a:Na=Ma\xhookrightarrow{id}Ma$.)

Now assume 1. Let $(\alpha _a)_{a\in \mathcal{C}}$ be as above. The following diagram shows that our collection satisfies the Tarski-Vaught condition in the stronger sense.
\[
\adjustbox{width=\textwidth}{
\begin{tikzcd}
	{M\varphi } && {Mx_0\times \dots Mx_{n-1}\times My_0\times \dots My_{m-1}} && {M\vec{s_0}\times \dots M\vec{s_{n-1}}\times M\vec{s'_0}\times \dots } \\
	\\
	p && {Mx_0\times \dots Mx_{n-1}\times Ny_0\times \dots Ny_{m-1}} && {M\vec{s_0}\times \dots M\vec{s_{n-1}}\times N\vec{s'_0}\times \dots } \\
	& \bullet && {Ny_0\times \dots} && {N\vec{s_0'}\times \dots} \\
	q && {Nx_0\times \dots Nx_{n-1}\times Ny_0\times  \dots Ny_{m-1}} && {N\vec{s_0}\times \dots N\vec{s_{n-1}}\times N\vec{s_0'}\times \dots}
	\arrow[hook, from=1-1, to=1-3]
	\arrow["{M\iota _{x_0} \times \dots}", hook, from=1-3, to=1-5]
	\arrow[""{name=0, anchor=center, inner sep=0}, hook, from=3-3, to=3-5]
	\arrow["{id\times \alpha _{\vec{s_0'}}\times \dots }"{description}, hook, from=3-5, to=1-5]
	\arrow["{id\times \alpha _{y_0}\times \dots \alpha _{y_{m-1}}}"{description}, hook, from=3-3, to=1-3]
	\arrow[hook, from=3-1, to=3-3]
	\arrow[hook, from=3-1, to=1-1]
	\arrow["\lrcorner"{anchor=center, pos=0.125, rotate=-90}, draw=none, from=1-3, to=3-1]
	\arrow["{\alpha _{x_0}\times \dots \alpha _{x_{n-1}}\times id}"{description, pos=0.3}, hook, from=5-3, to=3-3]
	\arrow[hook, from=5-1, to=3-1]
	\arrow[hook, from=5-1, to=5-3]
	\arrow["{\alpha _{\vec{s_0}}\times \dots \times id}"{description, pos=0.3}, hook, from=5-5, to=3-5]
	\arrow[""{name=1, anchor=center, inner sep=0}, hook, from=5-3, to=5-5]
	\arrow["\lrcorner"{anchor=center, pos=0.125, rotate=-90}, draw=none, from=3-3, to=5-1]
	\arrow[two heads, from=3-1, to=4-2]
	\arrow[from=3-3, to=4-4]
	\arrow[from=3-5, to=4-6]
	\arrow[hook, from=4-4, to=4-6]
	\arrow[hook, from=4-2, to=4-4]
	\arrow["\lrcorner"{anchor=center, pos=0.125, rotate=-90}, draw=none, from=1-5, to=0]
	\arrow["\lrcorner"{anchor=center, pos=0.125, rotate=-90}, draw=none, from=3-5, to=1]
\end{tikzcd}
}
\]
\end{proof}

Now we will discuss the categorical analogue of positively closed models, based on \cite{haykazyan}.

\begin{definition}
\label{poscldef}
Let $\mathcal{C}$ be coherent and let $\mathcal{D}$ be $|\mathcal{C}|^+$-geometric. The coherent functor $N:\mathcal{C}\to \mathcal{D}$ is positively closed if for each subobject $\varphi \hookrightarrow x$ in $\mathcal{C}$ we have $N(x)=N\varphi \cup \bigcup _{\psi \cap \varphi =\emptyset }N\psi $.
\end{definition}

\begin{remark}
If $\mathcal{C}$ has filtration $(\mathbb{S},(\iota _x)_x)$ then it is enough to check the assumption for the subobjects $\varphi \hookrightarrow s_0\times \dots \times s_{n-1}$ ($n\geq 0$). (In the general case just take intersection with $x\hookrightarrow s_{x,0}\times \dots s_{x,n_x-1}$ and with its $N$-image.)
\label{pos1}
\end{remark}

\begin{proposition}
If $N:\mathcal{C}\to \mathcal{D}$ is positively closed then each natural transformation $\alpha :N\Rightarrow M$ is elementary.
\label{allel}
\end{proposition}

\begin{proof}
We have
\[
N(y)=N\varphi \sqcup \bigcup _{\varphi \cap \psi =\emptyset } N\psi \leq \alpha _y^*(M\varphi ) \sqcup \bigcup _{\psi \cap \varphi =\emptyset }\alpha _y^*(M\psi )\leq N(y)
\]
and both summands on the left are contained in the corresponding one on the right from which $N\varphi =\alpha _y^*(M\varphi )$ follows. 
\end{proof}

\begin{remark}
For $\mathcal{C}\to \mathbf{Set}$ functors the converse is proved as part of Theorem \ref{posclequivalent}.
\end{remark}

The following is the analogue of Proposition 4.1.~in \cite{haykazyan}:

\begin{proposition}
Let $\mathcal{C}$ be a coherent category with filtration $(\mathbb{S},(\iota _x)_x)$, let $\mathcal{D}$ be $|\mathcal{C}|^+$-geometric and let $M:\mathcal{C}\to \mathcal{D}$ be a coherent functor. Assume that we have subobjects $\alpha _s:N(s)\hookrightarrow M(s)$ for $s\in \mathbb{S}$ such that for each $\varphi \hookrightarrow s_0\times \dots \times s_{n-1}\times s_0'\times \dots \times s_{m-1}'$ in

\[
\adjustbox{width=\textwidth}{
\begin{tikzcd}
	M\varphi && {Ms_0\times \dots Ms_{n-1}\times Ms_0'\times \dots Ms_{m-1}'} && {Ms_0'\times \dots Ms_{m-1}'} \\
	&&& {\exists _{\pi }N\varphi} \\
	N\varphi && {Ns_0\times \dots Ns_{n-1}\times Ns_0'\times \dots Ns_{m-1}'} && {\bigcup _{\substack{\chi \hookrightarrow s_0'\times \dots s_{m-1}'\\ \chi \cap \exists _{\pi }\varphi =\emptyset}}M\chi} \\
	&&& {Ns_0'\times \dots Ns'_{m-1}}
	\arrow["{\alpha _{s_0}\times \dots}", hook, from=3-3, to=1-3]
	\arrow[hook, from=1-1, to=1-3]
	\arrow[hook, from=3-1, to=3-3]
	\arrow[hook, from=3-1, to=1-1]
	\arrow["\ulcorner"{anchor=center, pos=0.125, rotate=-90}, draw=none, from=1-3, to=3-1]
	\arrow[from=1-3, to=1-5]
	\arrow[two heads, from=3-1, to=2-4]
	\arrow[hook, from=2-4, to=1-5]
	\arrow[hook, from=3-5, to=1-5]
	\arrow["{\alpha _{s_0'}\times \dots}"{description, pos=0.7}, hook, from=4-4, to=1-5]
\end{tikzcd}
}
\]
$Ns_0'\times \dots Ns'_{m-1}\leq \exists _{\pi }N\varphi  \cup \bigcup _{\chi } M\chi $ (for $n,m\geq 0$).

Then there is a unique positively closed coherent subfunctor $\alpha :N\Rightarrow M$ with component $\alpha _x$ at $x$.
\label{26}
\end{proposition}

\begin{proof}
First we verify the Tarski-Vaught condition. We have to prove that the dashed arrow in
\[
\adjustbox{width=\textwidth}{
\begin{tikzcd}
	M\varphi && {Ms_0\times \dots Ms_{n-1}\times Ms_0'\times \dots Ms_{m-1}'} && {Ms_0'\times \dots Ms_{m-1}'} \\
	& pb && {M(\exists _{\pi } \varphi)} && {\bigcup_{ \chi \cap \exists _{\pi} \varphi =\emptyset } M\chi} \\
	p && {Ms_0\times \dots Ms_{n-1}\times Ns_0'\times \dots Ns_{m-1}'} && {Ns_0'\times \dots Ns_{m-1}'} \\
	& pb && r && t \\
	q && {Ns_0\times \dots Ns_{n-1}\times Ns_0'\times \dots Ns_{m-1}'} & {r'}
	\arrow[hook, from=1-1, to=1-3]
	\arrow["{id\times \alpha _{s_0'}\times \dots}"{description}, hook, from=3-3, to=1-3]
	\arrow[hook, from=3-1, to=1-1]
	\arrow[hook, from=3-1, to=3-3]
	\arrow[from=3-3, to=3-5]
	\arrow["{\alpha _{s_0}\times \dots \times id}"{description}, hook, from=5-3, to=3-3]
	\arrow[hook, from=5-1, to=5-3]
	\arrow[hook, from=5-1, to=3-1]
	\arrow[two heads, from=3-1, to=4-4]
	\arrow[hook, from=4-4, to=3-5]
	\arrow[""{name=0, anchor=center, inner sep=0}, "{\alpha _{s_0'}\times \dots}"{description}, hook, from=3-5, to=1-5]
	\arrow[from=1-3, to=1-5]
	\arrow[two heads, from=1-1, to=2-4]
	\arrow[hook, from=2-4, to=1-5]
	\arrow[hook, from=4-4, to=2-4]
	\arrow[dashed, hook, from=5-4, to=4-4]
	\arrow[curve={height=18pt}, two heads, from=5-1, to=5-4]
	\arrow[hook, from=5-4, to=3-5]
	\arrow[hook', from=2-6, to=1-5]
	\arrow[hook', from=4-6, to=3-5]
	\arrow[""{name=1, anchor=center, inner sep=0}, hook', from=4-6, to=2-6]
	\arrow["pb"{description}, draw=none, from=0, to=1]
\end{tikzcd}
}
\]
is an isomorphism. We have $r\cap t=\emptyset $.
By assumption there's $Ns_0'\times \dots Ns_{m-1}' \leq r'\cup \bigcup_{ \chi } M(\chi )$ hence $Ns_0'\times \dots Ns_{m-1}' =r\cup t \leq r'\cup t$. Taking intersection with $r$ we get $r\leq r'$. By Theorem \ref{TV1} this implies that $(\alpha _s)_s$ extends (uniquely) to an elementary coherent subfunctor $\alpha :N\Rightarrow M$. Our assumption with $m=0$ gives that of Remark \ref{pos1} which implies that $N$ is positively closed. 
\end{proof}




\section{Omitting types}

\begin{definition}
$Sub_{\mathcal{C}}:\mathcal{C}\to \mathbf{DLat}^{op}$ is the functor which maps an object $x$ to the distributive lattice $Sub(x)$ and an arrow $f:x\to y$ to the homomorphism $f^*:Sub(y)\to Sub(x)$.
Given a coherent functor $M:\mathcal{C}\to \mathcal{D}$ we have a natural transformation $Sub_M:Sub_{\mathcal{D}}\circ M \Rightarrow Sub_{\mathcal{C}}$ whose components are given by the homomorphisms $M:Sub(x)\to Sub(Mx)$.
\end{definition}

\begin{remark}
To see that $Sub_M$ is indeed a natural transformation we need the commutativity of
\[
\adjustbox{scale=0.9}{
\begin{tikzcd}
	{Sub_{\mathcal{C}}(x)} && {Sub_{\mathcal{C}}(y)} \\
	\\
	{Sub_{\mathcal{D}}(Mx)} && {Sub_{\mathcal{D}}(My)}
	\arrow["{f^*}"', from=1-3, to=1-1]
	\arrow["M", from=1-1, to=3-1]
	\arrow["M", from=1-3, to=3-3]
	\arrow["{(Mf)^*}"', from=3-3, to=3-1]
\end{tikzcd}
}
\]
which follows as $M$ preserves pullbacks.
\label{subm}
\end{remark}

\begin{proposition}
\label{subcomm}
For elementary $\alpha$ the following commutes:
\[
\adjustbox{scale=0.9}{
\begin{tikzcd}
	{\mathcal{C}} \\
	&&&& {\mathbf{DLat}^{op}} \\
	{\mathcal{D}}
	\arrow[""{name=0, anchor=center, inner sep=0}, "N"{description}, curve={height=18pt}, from=1-1, to=3-1]
	\arrow[""{name=1, anchor=center, inner sep=0}, "M"{description}, curve={height=-18pt}, from=1-1, to=3-1]
	\arrow[""{name=2, anchor=center, inner sep=0}, "{Sub_{\mathcal{C}}}", from=1-1, to=2-5]
	\arrow[""{name=3, anchor=center, inner sep=0}, "{Sub_{\mathcal{D}}}"', from=3-1, to=2-5]
	\arrow["\alpha", shorten <=7pt, shorten >=7pt, Rightarrow, from=0, to=1]
	\arrow["{Sub_N}", curve={height=-6pt}, shorten <=5pt, shorten >=5pt, Rightarrow, from=3, to=2]
	\arrow["{Sub_M}"', curve={height=6pt}, shorten <=5pt, shorten >=5pt, Rightarrow, from=3, to=2]
\end{tikzcd}
}
\]
\end{proposition}

\begin{proof}
We need the commutativity of
\[\begin{tikzcd}
	& {Sub_{\mathcal{C}}(x)} \\
	{Sub_{\mathcal{D}}(Mx)} && {Sub_{\mathcal{D}}(Nx)}
	\arrow["M"', from=1-2, to=2-1]
	\arrow["{\alpha _x^*}"', from=2-1, to=2-3]
	\arrow["N", from=1-2, to=2-3]
\end{tikzcd}\]
which follows as $\alpha $ is elementary.
\end{proof}

\begin{theorem}[Stone-duality for distributive lattices, see e.g.~\cite{spectral}]
The functor $Spec: \mathbf{DLat}^{op}\to \mathbf{Spec}$ which takes a distributive lattice $L$ to the spectral space of its prime filters, and a homomorphism $h:L\to L'$ to $Spec(h):Spec(L')\to Spec(L)$, which is pulling back along $h$, is an equivalence of categories.
\end{theorem}

\begin{remark}
The basic opens (equivalently the compact opens) in $Spec\ L$ are of the form $[\varphi ]=\{p\in Spec\ L : \varphi \in p\}$ for some element $\varphi \in L$. Hence the closure of a prime filter $p$ is given as $\bigcap _{\psi \not \in p} \{r : \psi \not \in r\}$, i.e.~$q\in \overline{p}$ iff $q\leq p$.
\label{basic}
\end{remark}

\begin{definition}
We set $S_{\mathcal{C}}=Spec \circ Sub_{\mathcal{C}}$ and $S_M=Spec \circ Sub_M$. 
Elements of $S_{\mathcal{C}}(x)$ are called $x$-types over $\mathcal{C}$. An $x$-type $p$ is called somewhere dense if $int(\overline{p})\neq \emptyset$. An $x$-type $p$ is realized by $M:\mathcal{C}\to \mathcal{D}$ if $S_{M,x}^{-1}(p)$ contains a somewhere dense $Mx$-type, otherwise it is omitted. 
\label{typespacedef}
\end{definition}

\begin{remark}
If we start with the syntactic category of a 1-sorted coherent theory $\mathcal{C}=\mathcal{C}_T$ such that the single sort corresponds to the object $x\in \mathcal{C}$, then restricting $S_{\mathcal{C}}$ to the subcategory of finite powers of $x$ and morphisms whose coordinate functions are projections: $x^k\xrightarrow{\langle p_{i_1},\dots p_{i_n}\rangle } x^n$ yields the type space functor $\mathbf{FinOrd}^{op}\to \mathbf{Spec}$, studied e.g.~in \cite{kamsma}.
\end{remark}

\begin{proposition}
Every somewhere dense type is a maximal filter.
\label{maximal}
\end{proposition}

\begin{proof}
Assume $\emptyset \neq [\varphi ]\subseteq \overline{p}$. We claim $\varphi \in p$. If not, then $p\in S_{\mathcal{C}}(x)\setminus [\varphi ] $ and therefore $\overline{p}\subseteq S_{\mathcal{C}}(x)\setminus [\varphi ] $ as the latter is closed. Then we have $[\varphi ]\subseteq S_{\mathcal{C}}(x)\setminus [\varphi ]$ meaning $[\varphi ]=\emptyset $. But then for each $p\leq p'$ we have $\varphi \in p'$ from which $p'\leq p$ follows.
\end{proof}

\begin{proposition}
Take $p\in S_{\mathcal{C}}(x)$ with $\mathcal{C}\not \simeq *$. The following are equivalent:
\begin{enumerate}
    \item $\emptyset \neq [\varphi ] \subseteq \overline{p}$
    \item $\psi \in p \Leftrightarrow \varphi \cap \psi \neq \emptyset $
\end{enumerate}
\label{realeq}
\end{proposition}

\begin{proof}
$1\Rightarrow 2$. In the previous proposition we have seen $\varphi \in p$, hence $\psi \in p$ implies $\varphi \cap \psi \in p$ and hence $\varphi \cap \psi \neq \emptyset $. Conversely, since $\varphi \cap \psi \neq \emptyset $ it is contained in a prime filter $q\in S_{\mathcal{C}}(x)$. But then by Remark \ref{basic} we have $\psi , \varphi \in q\leq p$.

$2\Rightarrow 1$. As $\mathcal{C}\not \simeq *$ we have $\varphi \neq \emptyset$ and therefore $[\varphi ]\neq \emptyset $. Assume $\varphi \in q\in S_{\mathcal{C}}(x)$. Then $\psi \in q$ implies $q\ni \psi \cap \varphi \neq \emptyset $ and hence $\psi \in p$.
\end{proof}

\begin{proposition}
Take a coherent functor $*\not \simeq \mathcal{C}\xrightarrow{M}\mathcal{D}$ and take $p\in S_{\mathcal{C}}(x)$. The following are equivalent:
\begin{enumerate}
    \item $p$ is realized by $M$
    \item there is a subobject $a\in Sub_{\mathcal{D}}(Mx)$ such that $\varphi \in p \Leftrightarrow M\varphi \cap a \neq \emptyset $ and $b_1,b_2 \leq a$, $b_1\cap b_2=\emptyset $ implies $b_1=\emptyset $ or $b_2=\emptyset $.
\end{enumerate}
\label{29}
\end{proposition}

\begin{proof}
The type $p$ is realized by $M$ iff there is a subobject $a$ and a prime filter $u$ with $\emptyset \neq [a]\subseteq \overline{u}$ (equivalently: a subobject $a$ such that $\{b:a\cap b\neq \emptyset \}$ is a prime filter, and this is equivalent to our condition on $Sub_{\mathcal{D}}(a)$) and with $M^{-1}(u)=p$, i.e.~with $\varphi \in p$ iff $M\varphi \in u$ iff $M\varphi \cap a \neq \emptyset $. 
\end{proof}

\begin{definition}
A coherent category $\mathcal{D}$ is weakly Boolean if for any pair of subobjects $a\hookrightarrow x\hookleftarrow b$ either $a\leq b$ or there is $\emptyset \neq u\leq a$ such that $u\cap b=\emptyset $.
\end{definition}

\begin{definition}
An object $a\in \mathcal{C}$ is an atom if $|Sub_{\mathcal{C}}(a)|=2$. $\mathcal{C}$ is 2-valued if $1\in \mathcal{C}$ is an atom.
\end{definition}

\begin{definition}
We say that a type $p\in S_{\mathcal{C}}(x)$ is strongly realized by $M:\mathcal{C}\to \mathcal{D}$ if the object $a$ of Proposition \ref{29} is an atom. It is weakly omitted if it is not strongly realized. When $\mathcal{D}$ is weakly Boolean the two notions coincide.
\end{definition}

\begin{proposition}
Let $M:\mathcal{C}\to \mathcal{D}$ be positively closed with $\mathcal{D}$ being $|\mathcal{C}|^+$-geometric and let $\mathcal{C}$ be 2-valued, i.e.~$Sub_{\mathcal{C}}(1)\cong \mathbf{2}$. Assume that the set of somewhere dense points is dense in $S_{\mathcal{D}}(Mx)$. Then every somewhere dense type $p\in S_{\mathcal{C}}(x)$ is realized by $M$. 
\end{proposition}

\begin{proof}
Assume $\emptyset \neq [\varphi ]\subseteq \overline{p}$. Any coherent functor with 2-valued domain reflects the initial object. Indeed if $x\neq 0$ then $\exists _!x \neq 0$ since any map to the initial object is an isomorphism. But then we must have $\exists _!x=1$ and hence $\exists _!Mx=1$ from which $Mx\neq \emptyset $ follows.

By Proposition \ref{realeq} there's
\[
\psi \in p \Leftrightarrow \psi \cap \varphi \neq \emptyset \Leftrightarrow M\psi \cap M\varphi \neq \emptyset
\]
as $M$ reflects the initial object. For $\psi \leq x$ we have $Mx=M\psi \sqcup \bigcup_{\psi '\cap \psi =\emptyset } M\psi '$. Take $u\in S_{\mathcal{D}}(Mx)$ with $M\varphi \in u$. Then
\begin{multline*}
   \psi \in p \Rightarrow \forall \psi ': \psi '\cap \psi =\emptyset \to \psi '\not \in p \Rightarrow \\ M\psi '\cap M\varphi =\emptyset \Rightarrow M\varphi \cap \bigcup_{\psi '\cap \psi =\emptyset } M\psi ' =\emptyset \Rightarrow  M\varphi \leq M\psi \Rightarrow M\psi \in u
\end{multline*}
and
\begin{multline*}
\psi \not \in p \Rightarrow M\psi \cap M\varphi =\emptyset \Rightarrow \exists \psi ': \psi \cap \psi '=\emptyset \text{ and } M\varphi  \cap M\psi ' \neq \emptyset \Rightarrow \\ \psi '\in p \Rightarrow M\psi '\in u \Rightarrow M\psi \not \in u
\end{multline*}
meaning $M^{-1}(u)=p$ (and therefore $u\in S_M^{-1}(p)$). By assumption there is a somewhere dense prime filter $u\in  [M\varphi ]\neq \emptyset$ which completes the proof.
\end{proof}

\begin{proposition}
Let $f:x\to y$ be an arrow of $\mathcal{C}$. Then $S_{\mathcal{C}}(f)$ is an open map. If $f$ is mono/ effective epi then $S_{\mathcal{C}}(f)$ is injective/ surjective.
\label{arrows}
\end{proposition}

\begin{proof}
Image-factorization $\exists _f :Sub(x)\to Sub(y)$ is left adjoint to the pullback map $f^*:Sub(y)\to Sub(x)$, i.e.~we have that $\exists _f(r)\leq a$ iff $r\leq f^*(a)$. Hence we have $f^*(a)\wedge r \leq f^*(b)$ iff $\exists _f(f^*a\wedge r)\leq b$ iff $a\wedge \exists _f(r)\leq b$, where the last equivalence follows from Proposition 3.2.7.~of \cite{makkai}. By Theorem 8.~of \cite{spectral} this implies that $S_{\mathcal{C}}(f)=Spec(f^*)$ is open.

The second part follows from Theorem 6.~of \cite{spectral}, as if $f$ is monic then $f^*$ is surjective, while if $f$ is effective epi then $f^*$ is injective. 
\end{proof}

\begin{definition}
Let $\mathcal{C}$ be a coherent category and $\overline{\mathbb{S}}$ be a list of objects (typically a filtration but every element has $\omega $-many occurrences in the list). We write $S_{\mathcal{C}}^{\omega }=S_{\mathcal{C}}^{\omega }(\overline{\mathbb{S}})$ for the cofiltered limit $lim_{X\in [\overline{\mathbb{S}}]^{<\omega }} S_{\mathcal{C}}(\prod _{s\in X} s)$ (i.e.~the limit is taken over finite products from $\overline{\mathbb{S}}$ and the projection maps). Points of $S_{\mathcal{C}}^{\omega }$ are called $\omega $-types over $\mathcal{C}$ (even if $\overline{\mathbb{S}}$ is uncountable). 

$S_{\mathcal{C}}^{\omega }(\overline{\mathbb{S}})$ is the $Spec$-image of the filtered colimit $colim _{X\in ([\overline{\mathbb{S}}]^{<\omega })^{op}}Sub_{\mathcal{C}}(\prod _{s\in X} s)$. In particular any coherent functor $M:\mathcal{C}\to \mathcal{D}$ yields an induced map 
\[
colim _{X\in ([\overline{\mathbb{S}}]^{<\omega })^{op}}Sub_{\mathcal{C}}(\prod _{s\in X} s)\to colim _{X\in ([\overline{\mathbb{S}}]^{<\omega })^{op}}Sub_{\mathcal{D}}(\prod _{s\in X} Ms)
\]
whose $Spec$-image we call $S^{\omega }_M:S_{\mathcal{D}}^{\omega }(M\overline{\mathbb{S}})\to S_{\mathcal{C}}^{\omega }(\overline{\mathbb{S}})$.

An $\omega $-type $q\in S_{\mathcal{C}}^{\omega }(\overline{\mathbb{S}})$ is realized by $M:\mathcal{C}\to \mathcal{D}$ if there is a point $u\in S_{\mathcal{D}}^{\omega }(M\overline{\mathbb{S}})$ with $S^{\omega }_M(u)=q$ and for which each projection $u_X\in S_{\mathcal{D}}(\prod _{s\in X} Ms)$ is somewhere dense. 
\label{real}
\end{definition}

\begin{remark}
By Lemma 5.24.5.~of \cite{stacks-project} $S_{\mathcal{C}}^{\omega }$ is also the limit in $\mathbf{Top}$.
\label{toplim}
\end{remark}

\begin{proposition}
Let $*\not \simeq \mathcal{C}\xrightarrow{M} \mathcal{D}$ be a coherent functor and take $q\in S_{\mathcal{C}}^{\omega }(\overline{\mathbb{S}})$. The following are equivalent:
\begin{enumerate}
    \item $q$ is realized by $M$
    \item there are subobjects $b_X\hookrightarrow M(\prod _{s\in X} s)$ such that \begin{itemize}
        \item $\exists _{\langle \pi _{i_1},\dots \pi _{i_k}\rangle } (b_{\{s_1,\dots s_n\}}) \cap b_{\{s_{i_1},\dots s_{i_k}\}} \neq \emptyset $
        \item $a,a'\leq b_X$, $a\cap a'=\emptyset $ $\Rightarrow $ $a=\emptyset $ or $a'=\emptyset $
        \item $\varphi \in q_X \Leftrightarrow M\varphi \cap b_X \neq \emptyset $
    \end{itemize}
\end{enumerate}
\label{infreal}
\end{proposition}

\begin{proof}
As elements of the limit correspond to compatible sequences, $q$ is realized by $M$ iff there's $u_X\in S_{\mathcal{D}}(\prod _{s\in X} Ms)$ with $\emptyset \neq [b_X] \subseteq \overline{u_X}$ such that $S_{M,\prod _X s}(u_X)=q_X$ and $S_{\mathcal{D}}(\langle \pi _{i_1},\dots \pi _{i_k} \rangle )(u_{\{ s_1,\dots s_n \}})=u_{\{s_{i_1},\dots s_{i_k}\}}$. 

By Proposition \ref{29} the first two properties are equivalent to the last two bullet points while the third one means 
\[c\in u_{\{s_{i_1},\dots s_{i_k}\}} \text{ iff } c\times \prod _{j\in \{1,\dots n \}\setminus \{i_1,\dots i_k\}} Ms_j \in u_{\{s_1,\dots s_n\}}
\]
i.e.
\begin{multline*}
c\cap b_{\{s_{i_1}, \dots s_{i_k}\}} \neq \emptyset \Leftrightarrow c\times \prod _{j\in \{1,\dots n \}\setminus \{i_1,\dots i_k\}} Ms_j \cap b_{\{s_1,\dots s_n\}} \neq \emptyset \Leftrightarrow \\ c\cap \exists _{\langle \pi _{i_1},\dots \pi _{i_k}\rangle}b_{\{s_1,\dots s_n\}} \neq \emptyset 
\end{multline*}
Taking $c= \exists _{\langle \pi _{i_1},\dots \rangle}b_{\{ s_1,\dots \}}$ implies the first bullet point. Conversely, assuming $\exists _{\langle \pi _{i_1},\dots \pi _{i_k}\rangle } (b_{\{s_1,\dots s_n\}}) \cap b_{\{s_{i_1},\dots s_{i_k}\}} \neq \emptyset $ we conclude that $ \exists _{\langle \pi _{i_1},\dots \pi _{i_k}\rangle } b_{\{s_1,\dots s_n\}}\in u_{\{s_{i_1},\dots \}}$ and $b_{\{s_{i_1},\dots  s_{i_k}\}}\times \prod _{j\in \{1,\dots n \}\setminus \{i_1,\dots i_k\}} Ms_j \in u_{\{s_1,\dots \}}$ from which the above equivalence follows.
\end{proof}

\begin{definition}
A type $q\in S_{\mathcal{C}}^{\omega }(\overline{\mathbb{S}})$ is strongly realized by $M:\mathcal{C}\to \mathcal{D}$ if each $b_X$ is an atom. (When $\mathcal{D}$ is weakly Boolean $q$ is realized iff it is strongly realized.)
\end{definition}

\begin{proposition}
Let $\overline{\mathbb{S}}$ be countable and assume that for each $s\in \overline{\mathbb{S}}$ the map $s\to 1$ is effective epi. Then for each finite $X\subseteq \overline{\mathbb{S}}$ the projection map $\pi _{X}:S_{\mathcal{C}}^{\omega }(\overline{\mathbb{S}})\to S_{\mathcal{C}}(\prod _{s\in X} s)$ is an open surjection.
\end{proposition}

\begin{proof}
Writing $\overline{\mathbb{S}}=(s_i)_{i<\omega }$ yields a final subdiagram $s_0\leftarrow s_0\times s_1\leftarrow \dots$. As effective epimorphisms are pullback-stable each projection $s_0\times \dots \times s_{n-1}\leftarrow s_0\times \dots \times s_{n}$ is effective epi, then by Proposition \ref{arrows} each arrow $S_{\mathcal{C}}(s_0\times \dots \times s_{n-1})\xleftarrow{S_{\mathcal{C}}(\pi _{0\dots n-1})} S_{\mathcal{C}}(s_0\times \dots \times s_{n})$ is surjective.

In $\mathbf{Set}$ the transfinite cocomposition of surjections is surjective, hence the same holds in $\mathbf{Top}$ as the forgetful functor preserves limits. By Remark \ref{toplim} it follows that each $\pi _{\{s_0,\dots s_n\}}$ is onto. Hence each $\pi _X$ is onto.

By Lemma 5.14.2.~of \cite{stacks-project} the preimages $\{\pi _{\{s_0,\dots s_n\}}^{-1}(V): n\in \mathbb{N}, V\subseteq  S_{\mathcal{C}}(\prod _{i\leq n} s_i) \text{ open}\}$ form a basis of the limit. As images distribute over unions it suffices to prove that these are mapped to opens by $\pi _{\{s_0,\dots s_k \}}$'s. When $k=n$ we get back $V$ by surjectivity. If $k\geq n$ it follows from $S_{\mathcal{C}}(\prod _{i\leq n} s_i)\leftarrow S_{\mathcal{C}}(\prod _{i\leq k} s_i)$ being continuous. If $k\leq n$ it follows from $S_{\mathcal{C}}(\prod _{i\leq k} s_i)\leftarrow S_{\mathcal{C}}(\prod _{i\leq n} s_i)$ being open.
\end{proof}

Now we rewrite Lemma 4.2.~of \cite{haykazyan} with coherent categories. The following is Theorem 4.2.~in \cite{isbell}:

\begin{theorem}[Isbell]
Any compact sober space (hence every spectral space) is a Baire-space, i.e.~a countable intersection of dense open sets is dense.
\label{isbell}
\end{theorem}

\begin{definition}
Fix a countable coherent category $\mathcal{C}$ with filtration $\mathbb{S}$ such that for each $s\in \mathbb{S}$ the map $s\to 1$ is effective epi. Let $\overline{\mathbb{S}}=(s_i)_{i<\omega }$ be the list of objects from $\mathbb{S}$ where every element appears $\omega $-many times. We will write $\pi _K$ for $\pi _{\{s_0,\dots s_{K-1}\}}:S_{\mathcal{C}}(\overline{\mathbb{S}})\to S_{\mathcal{C}}(\prod _{n<K}s_n)$.

Take $\varphi \hookrightarrow s_{i_0}\times \dots s_{i_{k-1}}\times s_{j_0}\times \dots s_{j_{l-1}}$ with $k>0, l\geq 0$, and a function $f:k\to \omega $ such that $s_{i_n}$ and $s_{f(n)}$ are copies of the same element of $\mathbb{S}$ (we will refer to this as the sort of $s_{i_n}$ is the same as the sort of $s_{f(n)}$). Given $g:l\to \omega$ such that the sort of $s_{j_n}$ coincides with that of $s_{g(n)}$ we denote by $\varphi _{f,g}^* \hookrightarrow s_0\times \dots \times s_{K-1}$ the pullback of $\varphi $ along $\langle \pi_{f(0)},\dots \pi_{f({k-1})},\pi_{g(0)},\dots \pi_{g({l-1})} \rangle:\prod _{n<K}s_n\to s_{i_0}\times \dots s_{i_{k-1}}\times s_{j_0}\times \dots s_{j_{l-1}}$ for some $K> max\{ f(n), g(m) :n<k, m<l\}$. Let $O^f_{\varphi }$ be the following open set:
\begin{multline*}
O^f_{\varphi }=\bigcup_{\substack{g:l\to \omega \\ sort(s_{j_n})=sort(s_{g(n)})}} \pi _{K}^{-1}([\varphi _{f,g}^*]) \cup \\ \bigcup _{\substack{\chi \hookrightarrow s_{i_0}\times \dots s_{i_{k-1}} \times s_{t_0}\times \dots s_{t_{r-1}} \\ (\exists _{\pi }\chi)^*_f \cap (\exists _{\pi }\varphi )^*_f =\emptyset }} \bigcup _{\substack{h:r\to \omega \\ sort(s_{t_n})=sort(s_{h(n)})}} \pi _{K}^{-1}([\chi _{f,h}^*])
\end{multline*}
Finally we set $G=\bigcap _{\varphi } \bigcap _f O_{\varphi }^f \subseteq S_{\mathcal{C}}^{\omega }(\overline{\mathbb{S}})$.
\end{definition}

\begin{remark}
The above definition (hopefully) expresses the following: we fix an $\omega $-sequence of variables $x_i$ such that every sort appears infinitely many times. Then for any formula $\varphi (\vec{x},\vec{y})$ and any identification of $\vec{x}$ with some variables from the sequence ($f\vec{x}$), either we can identify $\vec{y}$ with some other variables ($g\vec{y}$) such that $\varphi (f\vec{x},g\vec{y})$ is in our $\omega $-type or there's $\chi (\vec{x},\vec{z})$ and an identification $h$ such that $\chi (f\vec{x},h\vec{z})$ is in our $\omega $-type and this is $T$-provably an obstruction for $\varphi $ being present, i.e.~$T\vdash \exists\vec{y} \varphi (f\vec{x},\vec{y})\wedge \exists\vec{z} \chi (f\vec{x},\vec{z}) \Leftrightarrow \bot $.
\end{remark}

\begin{proposition}
$G$ is comeagre in $S_{\mathcal{C}}^{\omega }(\overline{\mathbb{S}})$.
\end{proposition}

\begin{proof}
We have to show that each $O_{\varphi }^f$ is dense.  It suffices to prove that for any non-empty basic open $\emptyset \neq [\psi ]\subseteq S_{\mathcal{C}}(s_0\times \dots s_{K-1})$:  $\pi _{K}^{-1}([\psi ])\cap O_{\varphi }^f \neq \emptyset$. We can assume $K> max\{f(n):n<k\}$ (otherwise take the preimage of $V$ in the sequence). Let $g:l\to [K,L)\subseteq \omega $ be an injective function with $sort(s_{j_n})=sort(s_{g(n)})$. If $\psi \times s_{K}\times \dots \times s_{L-1}\cap \varphi _{f,g}^* =\emptyset $ then  $\pi _{K}^{-1}([\psi ]) \subseteq O_{\varphi }^f$ (as $\psi \hookrightarrow s_{f(0)}\times \dots s_{f(k-1)} \times \prod _{f(n)\neq m <K} s_m \xhookrightarrow{\Delta } s_{i_0}\times \dots s_{i_{k-1}}\times \prod _{f(n)\neq m <K} s_m$ is one of the $\chi $'s). If not then $\emptyset \neq \pi_{L}^{-1}([\psi \times s_{K}\times \dots \times s_{L-1}\cap \varphi _{f,g}^*])\subseteq \pi _{K}^{-1}([\psi ])\cap O_{\varphi }^f$.
\end{proof}

\begin{proposition}
Let $\mathcal{C}$, $\mathbb{S}$ and $\overline{\mathbb{S}}$ be as before and let $\mathcal{D}\not \simeq *$ be $\aleph _1$-geometric. Assume that $q\in G\subseteq S_{\mathcal{C}}^{\omega }(\overline{\mathbb{S}})$ is strongly realized by a coherent functor $M:\mathcal{C}\to \mathcal{D}$, and let $b_X\hookrightarrow \prod _{s\in X}Ms$ be as in Proposition \ref{infreal}. Finally take $c_i=\exists _{p_i}b_X \hookrightarrow s_i$ (an atom) and assume that $N(s)=\bigcup _{i:sort(s_i)=sort(s)} c_i \to 1$ is effective epi for each $s\in \mathbb{S}$ (e.g.~because $\mathcal{D}$ is 2-valued). Then there is a (unique) positively closed coherent subfunctor $\alpha :N \Rightarrow M$ with component $N(s)\hookrightarrow M(s)$ at $s\in \mathbb{S}$.
\label{ngeom}
\end{proposition}

\begin{proof}
By pullback-stability we have 
\begin{multline*}
N(s)\times N(s')=(\bigcup _{sort(s_i)=sort(s)}c_i)\times (\bigcup _{sort(s_j)=sort(s')}c_j)=\\ (M(s)\times \bigcup _{j}c_j) \cap (\bigcup _{i}c_i \times M(s')) = \bigcup _{j}(M(s)\times c_j) \cap \bigcup _{i}(c_i \times M(s')) = \bigcup _{i,j}c_i\times c_j
\end{multline*}
and more generally $N(s_{j_1})\times \dots \times N(s_{j_k})=\bigcup _{i_1\dots i_k: sort(s_{i_n})=sort(s_{j_n})} c_{i_1}\times \dots \times c_{i_k}$ for $k\geq 1$.

By Proposition \ref{26} we have to prove 
\begin{multline*}
N(s_{j_0})\times \dots N(s_{j_{l-1}})\leq \exists _{M\pi }(M(\varphi )\cap N(\vec{s_{i}})\times N(\vec{s_{j}})) \cup \bigcup _{\substack{\chi \hookrightarrow s_{j_0}\times \dots s_{j_{l-1}}\\ \chi  \cap \exists _{\pi }\varphi =\emptyset }} M(\chi )
\end{multline*}
for $\varphi \hookrightarrow s_{i_0}\times \dots s_{i_{k-1}}\times s_{j_0}\times \dots s_{j_{l-1}}$ and $k,l\geq 0$.

First assume $k>0, l\geq 0$. Fix $\varphi $ and $f:k\to \omega $ such that $sort(s_{i_n})=sort(s_{f(n)})$. Since $q\in G$ either we have $g:l\to \omega $ with $sort(s_{j_n})=sort(s_{g(n)})$ such that $\varphi ^*_{f,g} \in q_K$ or there is $\chi \hookrightarrow s_{i_0} \times \dots s_{i_{k-1}}\times s_{t_0}\times \dots s_{t_{r-1}}$ with $(\exists _{\pi }\chi)^*_f \cap (\exists _{\pi }\varphi)^*_f =\emptyset $ and there is $h:r\to \omega $ with $sort(s_{t_n})=sort(s_{h(n)})$ such that $\chi _{f,h}^*\in q_L$. In the first case $b_K\cap M(\varphi _{f,g}^*)\neq \emptyset $ hence $b_K\leq M(\varphi _{f,g}^*)$ as $b_K$ is an atom. Therefore $\prod _{\exists n: m=f(n)} c_m\leq \exists _{M\pi }(M(\varphi _{f,g}^*)\cap N(\vec{s_{f(n)}})\times N(\vec{s_{g(m)}}))$ hence $c_{f(0)}\times \dots \times c_{f(k-1)}\leq \exists _{M\pi }(M(\varphi )\cap N(\vec{s_{i}})\times N(\vec{s_{j}}))$. Similarly we get that in the second case $\prod _{\exists n: m=f(n)} c_m \leq \exists _{M\pi }M(\chi _{f,h}^*) = M((\exists _{\pi } \chi )_f^*)$. We have a projection map $\prod _{\exists n: m=f(n)} s_m\leftarrow s_{f(0)}\times \dots \times s_{f(k-1)} $. Let $\widetilde{\chi }$ be the pullback of $\exists _{\pi }\chi $ along this map. Then $c_{f(0)}\times \dots c_{f(k-1)}\leq M(\widetilde{\chi }\cap \Delta )$ and $(\widetilde{\chi }\cap \Delta )\cap \exists _{\pi }\varphi =\emptyset $ (where $\Delta$ is $\prod _{\exists n: m=f(n)}s_m\xhookrightarrow{\langle \pi _{f(0)},\dots \rangle }s_{f(0)}\times \dots s_{f(k-1)}$.

Now assume $k=0, l\geq 0$ and take $\varphi \hookrightarrow s_{j_0}\times \dots s_{j_{l-1}}$. We have to prove that 
\[
1\leq \exists _! (M\varphi \cap N\vec{s_j} ) \cup \bigcup _{\substack{\chi \hookrightarrow 1 \\ \chi \cap \exists _! \varphi =\emptyset }} M\chi
\]
As $Ns\times N\vec{s_j}\to 1$ is effective epi (for any $s$) it is enough to prove that the pullback of the inequality along this map holds, i.e.
\[
Ns\times N\vec{s_j}\leq Ns\times (M\varphi \cap N\vec{s_j}) \cup \bigcup _{\substack{\chi \hookrightarrow 1\\ \chi \cap \exists _!\varphi =\emptyset }} M(\chi \times s\times \vec{s_j})
\]
We have $\chi \cap \exists _!\varphi =\emptyset $ iff $!^{-1}\chi \cap \varphi =\emptyset $ (using $\exists _f(a\cap f^{-1}b)=\exists _fa \cap b$), and therefore this inequality follows by 
\[
Ns\times N\vec{s_j} \leq (M(s\times \varphi ) \cap Ns\times N\vec{s_j}) \cup \bigcup _{\substack{\chi '\hookrightarrow s\times \vec{s_j} \\ \chi ' \cap (s\times \varphi )=\emptyset }} M\chi ' 
\]
\end{proof}

\begin{theorem}
Let $\mathcal{C}$ be a countable coherent category with filtration $\mathbb{S}$ such that for each $s\in \mathbb{S}$: $s\to 1$ is effective epi, and let $\mathcal{D}$ be 2-valued $\aleph _1$-geometric. Let $E_X\subseteq S_{\mathcal{C}}(\prod _{s\in X}s)$ be meagre subsets (for $X\subseteq \overline{ \mathbb{S}}$ finite). Then there is a comeagre set $A\subseteq S_{\mathcal{C}}^{\omega }(\overline{\mathbb{S}})$ such that if any $q\in A$ is strongly realized by some $M:\mathcal{C}\to \mathcal{D}$ coherent functor then $M$ has a positively closed coherent subfunctor $N$ which weakly omits each type from each $E_X$.
\label{omitting}
\end{theorem}

\begin{proof}
Take $A=G\setminus \bigcup _X \pi _{X}^{-1}(E_X)$. As the preimage of a meagre set at an open map is meagre we have that $A$ is comeagre. Assume that $q\in A$ is strongly realized by $M:\mathcal{C}\to \mathcal{D}$, i.e.~there are atomic subobjects $c_i\leq M(s_i)$ such that for $\varphi \hookrightarrow \prod _{s_i\in X}s_i$ we have $\prod _{i:s_i\in X}c_i\cap M(\varphi )\neq \emptyset \Leftrightarrow \prod _{i:s_i\in X}c_i \leq M(\varphi ) \Leftrightarrow \varphi \in q_X$ and let $(N(s))_{s\in \mathbb{S}}$ be as before. By the previous proposition it extends to a positively closed coherent subfunctor $N:\mathcal{C}\to \mathcal{D}$.

Assume that $N$ realizes an $X=\{s_{i_0},\dots s_{i_{m-1}}\}$-type $p\in E_X$. That is, there is an atom $a\leq \prod _{r<m} N(s_{i_r})$ such that $a\cap N\varphi \neq \emptyset $ iff $\varphi \in p$. As $\prod _{r<m}N(s_{i_r})=\bigcup _{\substack{j_0,\dots j_{m-1}\\ sort(s_{j_r})=sort(s_{i_r})}} c_{j_0}\times \dots c_{j_{m-1}}$ there is $j_0,\dots j_{m-1}$ such that $a\cap c_{j_0}\times \dots \times c_{j_{m-1}}\neq \emptyset$. Then the two atoms must coincide hence $p=\pi _{\{s_{j_0},\dots s_{j_{m-1}}\}}(q)$ which is a contradiction.
\end{proof}

One of our next goals is to give examples of pairs $(\mathcal{C},\mathcal{D})$ such that small sets of $\mathcal{C}$-types can be omitted by $\mathcal{D}$-models. Since in the $\mathbf{Set}$-case the realization of $\omega $-types is guaranteed by the compactness theorem, we study ultraproducts.

\section{Ultraproducts}

\begin{definition}
Let $\mathcal{D}$ be a category with products and filtered colimits. Let $(D_i)_{i\in I}$ be a set of objects and $U\subseteq \mathcal{P}(I)$ be an ultrafilter. The ultraproduct of $D_i$'s wrt.~$U$ is $\faktor{\prod _{i\in I} D_i}{U}=colim_{J\in U} \prod _{i\in J} D_i$. (The colimit is taken over the projection maps. The diagram is filtered as $U$ is closed under intersections.)
\end{definition}

The following is the generalization of Theorem 2.1.1 in \cite{ultracat}:

\begin{theorem}[Łoś-lemma]
\label{los}
Let $\kappa $ and $\lambda $ be regular cardinals such that $\gamma <\kappa \Rightarrow 2^{\gamma }<\lambda $ and let $\mathcal{D}$ be a $\lambda $-geometric category such that
\begin{enumerate}
    \item it has $<\kappa $ products,
    \item for any object $x$: in $Sub_{\mathcal{D}}(x)$ the $<\kappa $ meets distribute over binary unions,
    \item if $(x_i\xrightarrow{f_i} y_i)_{i<\gamma <\kappa } $ are effective epimorphisms then $\prod _{i } x_i \xrightarrow{\prod _{i} f_i} \prod _{i} y_i$ is an effective epimorphism,
    \item the $<\kappa $-product of non-initial objects is non-initial,
    \item it has filtered colimits of size $<\lambda $, which commute with finite limits.
    \item it is weakly Boolean
\end{enumerate}
Then $\mathbf{Coh}(\mathcal{C},\mathcal{D})$ is closed under $<\kappa $ ultraproducts in $\mathcal{D}^{\mathcal{C}}$ for any coherent category $\mathcal{C}$.

Moreover given $A_i\hookrightarrow X_i\hookleftarrow B_i$ in $\mathcal{D}$:
\[ 
\faktor{\prod _I A_i }{U}\subseteq \faktor{\prod _I B_i }{U} \text{ iff }\{i: A_i\subseteq B_i\}\in U.
\]
\end{theorem}

\begin{proof}
We have to prove that $\faktor{\prod _I F_i}{U}$ is coherent when each $F_i: \mathcal{C}\to \mathcal{D}$ is. $\prod _J F_i$ preserves finite limits, the initial object (in a coherent category any map $x\to \emptyset $ is an isomorphism as the subobject $\emptyset \hookrightarrow \emptyset $ is pullback-stable, in particular $\prod _J \emptyset \xrightarrow{\pi _1} \emptyset $ is iso), and effective epimorphisms, by assumption. As filtered colimits commute with finite limits and all colimits, $\faktor{\prod _i F_i }{U}$ preserves finite limits, the initial object and effective epimorphisms.

It remains to prove that $\faktor{\prod _i F_i}{U}$ preserves binary unions $a\cup b\hookrightarrow x$, i.e.~that the induced map $f$ in
\[
\adjustbox{scale=0.9}{
\begin{tikzcd}
	{\faktor{\prod _i F_i(a)}{U}} & {\faktor{\prod _i F_i(a)}{U}\cup \faktor{\prod _i F_i(b)}{U}} & {\faktor{\prod _i F_i(b)}{U}} \\
	& {\faktor{\prod _i F_i(a\cup b)}{U}} \\
	& {\faktor{\prod _i F_i(x)}{U}}
	\arrow[curve={height=18pt}, hook, from=1-1, to=3-2]
	\arrow[curve={height=-18pt}, hook', from=1-3, to=3-2]
	\arrow[hook', from=1-3, to=1-2]
	\arrow[hook, from=1-1, to=1-2]
	\arrow["f", dashed, hook', from=1-2, to=2-2]
	\arrow[hook', from=2-2, to=3-2]
\end{tikzcd}
}
\]
is surjective. Take $\emptyset \neq v \hookrightarrow \faktor{\prod _i F_i(a\cup b)}{U}$. For some ultra-set $J\in U$: $v$ intersects the image of the cocone map $\prod _J F_i(a\cup b)\to \faktor{\prod _I F_i(a\cup b)}{U}$ as pullbacks and filtered colimits commute. Let $\emptyset \neq w \hookrightarrow \prod _J F_i(a)\cup F_i(b)$ be a non-empty subobject in the preimage of this intersection. As 

\begin{multline*}
    \prod _J F_i(a)\cup F_i(b) =\bigcap _J F_0(x)\times \dots \times (F_i(a)\cup F_i(b))\times F_{i+1}(x)\times \dots =\\
 \bigcap _J [F_0(x)\times \dots \times F_i(a) \times F_{i+1}(x)\times \dots \ \cup \ F_0(x)\times \dots \times F_i(b)\times F_{i+1}(x)\times \dots  ] =\\
\bigcup _{\varepsilon (i)\in \{a,b\}} \bigcap _{J} F_0(x)\times \dots \times F_i(\varepsilon (i))\times F_{i+1}(x)\times \dots = \bigcup _{\varepsilon } \prod _{J} F_i(\varepsilon (i))
\end{multline*}
(using that in $Sub_{\mathcal{D}}(\prod _J F_i(x))$ the $<\kappa $ meets distribute over binary unions), for some choice-function $\varepsilon$: $w'=w\cap \prod _J F_i(\varepsilon (i))$ is non-empty. But for some ultraset $J'\subseteq J$: $\varepsilon (i)$ is constant, let's say it equals $a$. The commutativity of
\[
\adjustbox{scale=0.8}{
\begin{tikzcd}
	{\prod _{J'}F_i(a)} && {\faktor{\prod _I F_i(a)}{U}} && {\faktor{\prod _I F_i(a)}{U} \cup \faktor{\prod _I F_i(b)}{U}} \\
	&& {\prod _{J'}F_i(a)\cup F_i(b)} \\
	{\prod _J F_i(\varepsilon (i))} && {\prod _JF_i(a)\cup F_i(b)} && {\faktor{\prod _I F_i(a)\cup F_i(b)}{U}} \\
	\\
	{\emptyset \neq w'} && {\emptyset \neq w} && {\emptyset \neq v}
	\arrow[hook, from=5-5, to=3-5]
	\arrow[from=3-3, to=3-5]
	\arrow[hook, from=5-3, to=3-3]
	\arrow[from=5-3, to=5-5]
	\arrow[hook, from=3-1, to=3-3]
	\arrow[hook, from=5-1, to=3-1]
	\arrow[hook, from=5-1, to=5-3]
	\arrow[from=3-1, to=1-1]
	\arrow[from=1-1, to=1-3]
	\arrow[hook, from=1-3, to=1-5]
	\arrow["f", dashed, hook', from=1-5, to=3-5]
	\arrow[hook, from=1-1, to=2-3]
	\arrow[from=3-3, to=2-3]
	\arrow[from=2-3, to=3-5]
	\arrow[dashed, hook, from=1-3, to=3-5]
\end{tikzcd}
}
\]
implies that the image of $f$ intersects $v$, i.e.~it intersects every non-empty subobject of the codomain. As $\mathcal{D}$ was weakly Boolean, we get that $f$ is surjective.

For the second part first assume $J=\{i:A_i\subseteq B_i\}\in U$. Then we get a commutative triangle
\[\begin{tikzcd}
	{\prod _{J'} A_i} \\
	&& {\prod _{J'}X_i} \\
	{\prod _{J'}B_i}
	\arrow["{\prod _{J'}f_i}", hook, from=1-1, to=2-3]
	\arrow["{\prod _{J'}g_i}"', hook, from=3-1, to=2-3]
	\arrow["{\prod _{J'}h_i}"', from=1-1, to=3-1]
\end{tikzcd}\]
for each $U\ni J'\subseteq J$. As such ultrasets with reversed inclusion form a cofinal subdiagram of all ultrasets (with reversed inclusion), this induces the required map $\faktor{\prod _I A_i}{U}\to \faktor{ \prod _i B_i}{U}$. Conversely assume $\faktor{\prod _I A_i}{U}\subseteq \faktor{ \prod _I B_i}{U}$. If $\{i:A_i\subseteq B_i\}\not \in U$ then $J=\{i:\exists \  \emptyset \neq v_i\hookrightarrow B_i: v_i\cap A_i \cap B_i=\emptyset \}\in U$ (as $\mathcal{D}$ is weakly Boolean). For each $U\ni J'\subseteq J$ we get
\[\begin{tikzcd}
	{\emptyset \neq \prod _{J'} v_i} && {\prod _{J'} B_i} \\
	{\emptyset } && {\prod _{J'}A_i\cap B_i}
	\arrow[hook, from=1-1, to=1-3]
	\arrow[""{name=0, anchor=center, inner sep=0}, hook, from=2-3, to=1-3]
	\arrow[""{name=1, anchor=center, inner sep=0}, hook, from=2-1, to=1-1]
	\arrow[hook, from=2-1, to=2-3]
	\arrow["pb"{description}, draw=none, from=1, to=0]
\end{tikzcd}\]
therefore (cofinality + every $\bullet \to \emptyset $ map is iso + filtered colimits commute with pullbacks):

\[\begin{tikzcd}
	{\emptyset \neq \faktor{\prod _{J} v_i}{U}} && {\faktor{\prod _{J} B_i}{U}=\faktor{\prod _{I} B_i}{U}} \\
	{\emptyset } && {\faktor{\prod _{J} A_i\cup B_i}{U}=\faktor{\prod _{I} A_i\cup B_i}{U}}
	\arrow[hook, from=1-1, to=1-3]
	\arrow[""{name=0, anchor=center, inner sep=0}, "\cong"', hook, from=2-3, to=1-3]
	\arrow[""{name=1, anchor=center, inner sep=0}, hook, from=2-1, to=1-1]
	\arrow[hook, from=2-1, to=2-3]
	\arrow["pb"{description}, draw=none, from=1, to=0]
\end{tikzcd}\]
contradiction.
\end{proof}

\begin{remark}
In the list of assumptions we could change $4.$ to
\begin{itemize}
    \item[$4'.$] $\{0\}\subseteq Sub_{\mathcal{D}}(1)$ is a $\kappa $-prime ideal. 
\end{itemize}
Indeed, given objects $(x_i)_{i<\gamma <\kappa }$ we can take the effective epi-mono factorizations $x_i\twoheadrightarrow u_i\hookrightarrow 1$, and as $x_i\neq 0$, also $u_i\neq 0$. By $3.$ map $\prod _i x_i \to \prod _i u_i$ is effective epi and the codomain is $\bigcap _i u_i$ which is not $0$ by our assumption.
\end{remark}

\begin{theorem}
Let $\mathcal{D}$ be as above. The diagonal map $\Delta :M\Rightarrow \faktor{M^I}{U}$ is an elementary natural transformation.
\end{theorem}

\begin{proof}
It is a natural transformation because 
\[\begin{tikzcd}
	Mx && My \\
	{Mx^J} && {My^J} \\
	{\faktor{Mx^I}{U}} && {\faktor{My^I}{U}}
	\arrow["Mf", from=1-1, to=1-3]
	\arrow["\delta"', from=1-1, to=2-1]
	\arrow["\delta", from=1-3, to=2-3]
	\arrow["{(\psi _J)_x}"', from=2-1, to=3-1]
	\arrow["{Mf^I}"', from=2-1, to=2-3]
	\arrow["{(\psi _J)_y}", from=2-3, to=3-3]
	\arrow["{\faktor{Mf^I}{U}}"', dashed, from=3-1, to=3-3]
\end{tikzcd}\]
commutes. Note that the maps $\Delta _x$ are monomorphisms as filtered colimits commute with monomorphisms:

\[
\adjustbox{scale=0.9}{
\begin{tikzcd}
	Mx && Mx & \dots & Mx \\
	\\
	{Mx^J} && {Mx^{J'}} & \dots & {\faktor{Mx^I}{U}}
	\arrow["\delta"', hook', from=1-1, to=3-1]
	\arrow[Rightarrow, no head, from=1-1, to=1-3]
	\arrow["\delta", hook', from=1-3, to=3-3]
	\arrow[from=3-1, to=3-3]
	\arrow[Rightarrow, no head, from=1-3, to=1-4]
	\arrow[from=3-3, to=3-4]
	\arrow["{\Delta _x}", hook', from=1-5, to=3-5]
\end{tikzcd}
}
\]

Given a monomorphism $a\hookrightarrow x$ assume that the induced map in 
\[
\adjustbox{scale=0.9}{
\begin{tikzcd}
	Ma &&& Mx \\
	& R & pb \\
	{\faktor{Ma^I}{U}} &&& {\faktor{Mx^I}{U}}
	\arrow["Mf", hook, from=1-1, to=1-4]
	\arrow["{\faktor{Mf^I}{U}}", hook, from=3-1, to=3-4]
	\arrow["{\Delta _x}"{description}, hook, from=1-4, to=3-4]
	\arrow["{\Delta _a}"{description}, from=1-1, to=3-1]
	\arrow[hook', from=2-2, to=3-1]
	\arrow[hook, from=2-2, to=1-4]
	\arrow[dashed, hook, from=1-1, to=2-2]
\end{tikzcd}
}
\]
is not effective epi. As $\mathcal{D}$ is weakly Boolean there is a subobject $\emptyset \neq v\hookrightarrow R$ which is disjoint from $Ma$. 

There is $J\in U$ such that $v^*$, the pullback of $v\to \faktor{Ma^I}{U}$ along the cocone map $Ma^J\to \faktor{Ma^I}{U}$ is non-empty:

\[\begin{tikzcd}
	{\emptyset \neq v^*} && {Ma^J} \\
	v && {\faktor{Ma^{I}}{U}} \\
	Mx && {\faktor{Mx^I}{U}}
	\arrow[from=1-1, to=1-3]
	\arrow[from=1-1, to=2-1]
	\arrow[from=2-1, to=2-3]
	\arrow[from=1-3, to=2-3]
	\arrow[from=2-1, to=3-1]
	\arrow[from=3-1, to=3-3]
	\arrow[from=2-3, to=3-3]
\end{tikzcd}\]
(And here the image of $v^*\to v\hookrightarrow \faktor{Mx^I}{U}$ is disjoint from $Ma\xhookrightarrow{\Delta _a} \faktor{Ma^I}{U}$.)

From this we get a diagram
\[
\adjustbox{scale=0.8}{
\begin{tikzcd}
	{v^*} && {Ma^J} && {\faktor{Ma^I}{U}} \\
	& {?} && {=} \\
	Mx && {Mx^J} && {\faktor{Mx^I}{U}}
	\arrow[from=1-1, to=3-1]
	\arrow[from=1-1, to=1-3]
	\arrow["{Mf^J}"{description}, from=1-3, to=3-3]
	\arrow["\delta"{description}, from=3-1, to=3-3]
	\arrow[from=3-3, to=3-5]
	\arrow[from=1-3, to=1-5]
	\arrow["{\faktor{Mf^I}{U}}"{description}, from=1-5, to=3-5]
\end{tikzcd}
}
\]
where the rectangle and the right square commutes. If the equalizer of the two $v^*\to Mx^J$ maps is non-empty we replace $v^*$ with that. Otherwise there should be $J\supseteq J'\in U$ such that the equalizer of the two maps post-composed with the projection is non-empty. This is because filtered colimits and equalizers commute:

\[\begin{tikzcd}
	\emptyset && {v^*} && {Mx^{J}} \\
	\emptyset && {v^{*}} && {Mx^{J'}} \\
	&& \dots && \dots \\
	\emptyset && {v^{*}} && {\faktor{Mx^I}{U}}
	\arrow["eq", hook, from=1-1, to=1-3]
	\arrow[curve={height=-6pt}, from=1-3, to=1-5]
	\arrow[curve={height=6pt}, from=1-3, to=1-5]
	\arrow[from=1-5, to=2-5]
	\arrow[Rightarrow, no head, from=1-3, to=2-3]
	\arrow[curve={height=-6pt}, from=2-3, to=2-5]
	\arrow[curve={height=6pt}, from=2-3, to=2-5]
	\arrow[from=2-5, to=3-5]
	\arrow[Rightarrow, no head, from=2-3, to=3-3]
	\arrow["eq", hook, from=2-1, to=2-3]
	\arrow[curve={height=-6pt}, from=4-3, to=4-5]
	\arrow[curve={height=6pt}, from=4-3, to=4-5]
	\arrow["eq"', hook, from=4-1, to=4-3]
\end{tikzcd}\]
So now we have a commutative outer square in
\[\begin{tikzcd}
	{\emptyset \neq v^{**}} & {v^*} && Mx \\
	{v^*} & Ma & pb \\
	{Ma^{J'}} &&& {Mx^{J'}}
	\arrow["e", hook, from=1-1, to=1-2]
	\arrow["e"', hook', from=1-1, to=2-1]
	\arrow[from=1-2, to=1-4]
	\arrow["\delta", from=1-4, to=3-4]
	\arrow[from=2-1, to=3-1]
	\arrow["{Mf^{J'}}"', from=3-1, to=3-4]
	\arrow["Mf"{description}, from=2-2, to=1-4]
	\arrow["\delta"{description}, from=2-2, to=3-1]
	\arrow[dashed, from=1-1, to=2-2]
\end{tikzcd}\]
for some $\emptyset \neq v^{**}\hookrightarrow v^*$ and $J\supseteq J'\in U$. But as the inner square is a pullback we get the induced map. Therefore the images of $v^*\to Mx\to Mx^{J'}\to \faktor{Mx^I}{U}$ and $Ma\xrightarrow{\delta }Ma^{J'}\to Mx^{J'}\to \faktor{Mx^I}{U}$ are not disjoint which is a contradiction.
\end{proof}

As usual Łoś-lemma implies compactness:

\begin{theorem}
\label{compactnesslos}
Let $\mathcal{D}$ be as above. Let $L$ be a signature with $|L|\cdot \aleph _0 <\kappa $ and $T\subseteq L_{\omega \omega }^g$ be a theory. If each finite subset $S\subseteq T$ has a model in $\mathcal{D}$ then $T$ has a model in $\mathcal{D}$.
\end{theorem}

\begin{proof}
Let $\mathcal{C}_S$ be the syntactic category of $S$ (objects are coherent $L$-formulas up to change of free variables, morphisms are $S$-provable equivalence classes of $S$-provably functional formulas, see \cite{makkai}). For $S\subseteq S'$ we have a coherent functor $\mathcal{C}_S\to \mathcal{C}_{S'}$ mapping $[\varphi (x)]\xrightarrow{[\mu (x,y)]}[\psi (y)]$ to $[\varphi (x)]\xrightarrow{[\mu (x,y)]}[\psi (y)]$ (which is well-defined as $S\vdash \mu \Leftrightarrow \mu '$ implies $S'\vdash \mu \Leftrightarrow \mu '$). For each finite $S\subseteq T$ fix $M_S:\mathcal{C}=\mathcal{C}_{\emptyset }\to \mathcal{C}_S \to \mathcal{D}$.

Let $U\subseteq \mathcal{P}([T]^{<\omega })$ be an ultrafilter extending $\{\widehat{S}=\{S'\in [T]^{<\omega }: S\subseteq S'\} : S\in [T]^{<\omega }\}$. Then by Theorem \ref{los} $M=\faktor{\prod _{[T]^{<\omega }} M_S}{U}:\mathcal{C}\to \mathcal{D}$ is a coherent functor with 
\[M([\varphi (\vec{x},\vec{z})\wedge \vec{y}\approx \vec{y}])\subseteq M([\psi (\vec{y},\vec{z}) \wedge \vec{x}\approx \vec{x}]) \text{ for each }\varphi (\vec{x},\vec{z})\Rightarrow \psi (\vec{y},\vec{z}) \in T
\]
(More precisely the subobject above is represented by 
\[ [\varphi (\vec{x},\vec{z})\wedge y\approx y]\xrightarrow{[\varphi (\vec{x},\vec{z}) \wedge \vec{x}\approx \vec{x'} \wedge \vec{y}\approx \vec{y'} \wedge \vec{z}\approx \vec{z'}]} [\vec{x'}\approx \vec{x'}\wedge \vec{y'}\approx \vec{y'}\wedge \vec{z'}\approx \vec{z'}] \]
and similarly for $\psi$.) Hence $M$ factors through $\mathcal{C}_T$.
\end{proof}

\begin{theorem}
\label{omitting3}
Assume the following:
\begin{itemize}
    \item $\mathcal{C}$ is countable with a filtration $\mathbb{S}$ such that for each $s\in \mathbb{S}$: $s\to 1$ is effective epi.
    \item $\mathcal{D}$ is $(2^{\aleph _0})^+$-geometric, weakly Boolean, 2-valued, in which atoms are dense. It has all countable products. It has filtered colimits of size $\leq 2^{\aleph _0}$ which commute with finite limits.
    \item There is a $\mathcal{C}\to \mathcal{D}^I$ conservative coherent functor.
\end{itemize}
Then given meagre subsets $E_X\subseteq S_{\mathcal{C}}(\prod _{s\in X}s)$ for $X\subseteq \overline{\mathbb{S}}$ finite, there is a $\mathcal{C}\to \mathcal{D}$ positively closed coherent functor which omits each type from each $E_X$.
\end{theorem}

\begin{proof}
First we check that the conditions of Theorem \ref{los} are satisfied. Assume that each $f_i:x_i\twoheadrightarrow y_i$ is effective epi and that $\prod _i f_i:\prod _i x_i\to \prod _i y_i$ is not. As $\mathcal{D}$ is weakly Boolean with dense atoms there is an atom $a$ which is disjoint from the image of $\prod _i f_i$. Taking atoms in each preimage and taking their product shows $a$ to lie in the image of $\prod _if_i$.

We also have to prove $\bigcap _{i<\omega } a_{0,i}\cup a_{1,i} =\bigcup _{\varepsilon :\omega \to \{0,1\}} \bigcap _{i<\omega } a_{\varepsilon (i),i}$. The $\geq $ direction is obvious, if this is a proper subobject we get an atom which is disjoint from it, this yields contradiction.

By Theorem \ref{omitting} and by Theorem \ref{isbell} it suffices to prove that every $\omega $-type $p\in S_{\mathcal{C}}^{\omega }(\overline{\mathbb{S}})$ is (strongly) realized by a $\mathcal{C}\to \mathcal{D}$ coherent functor.

Let $L$ be the extended canonical signature $L_{\mathcal{C}}$ of $\mathcal{C}$ plus countably many relation symbols $(R_i\hookrightarrow s_i)_{i<\omega }$ and $T$ be the theory 
\begin{multline*}
Th(\mathcal{C})+\{R_{i_0}(x_0)\wedge \dots R_{i_{k-1}}(x_{k-1}) \Rightarrow \varphi (x_0,\dots x_{k-1}) : \varphi \in \pi _{\{i_0,\dots i_{k-1}\}}(q) \}+\\ \{R_{j_0}(x_0')\wedge \dots \wedge R_{j_{l-1}}(x_{l-1}')\wedge \psi (x_0',\dots x_{l-1}')\Rightarrow \bot : \psi \not \in \pi _{\{j_0,\dots j_{l-1} \}}(q) \}
\end{multline*}
We have to prove that $T$ has a model in $\mathcal{D}$ where each $R_i$ is interpreted as an atomic subobject.

By the previous theorem it suffices to prove that every finite $T_0\subseteq T$ has such a model $M_{T_0}$ in $\mathcal{D}$, as in this case a countable ultraproduct of those is a model of $T$. In this ultraproduct the interpretation of $R_i$ is the filtered colimit of countable products of $R_i^{M_{T_0}}$'s, which is an atom.

Given $Th(\mathcal{C})$ plus finitely many formulas from the rest of the axioms mentioning $(\varphi _i)_{i<m}$ and $(\psi _i)_{i<n}$, adding extra variables (taking pullbacks along the projections) we may assume that these formulas contain the same free variables (are all subobjects of the same finite product of $s_i$'s). As $\bigcap \varphi _i \in \pi _{\vec{s}}(q)$, $\bigvee \psi _j \not \in \pi _{\vec{s}}(q)$ we get $T\not \vdash \bigwedge \varphi _i\Rightarrow \bigvee _i \psi _i$ and by assumption there is a model $M:\mathcal{C}\to \mathcal{D}$ in which $M(\bigwedge \varphi _i)\not \subseteq M(\bigvee \psi _i)$ and as $\mathcal{D}$ is weakly Boolean with dense atoms there is an atom in $M(\bigwedge \varphi _i)\setminus M(\bigvee \psi _i)\subseteq M(\vec{s})$ and its projections interpret the $R_i$'s which appear in the finitely many formulas.
\end{proof}

Our next goal is to see that (most of) the conditions of Theorem \ref{los} are satisfied when $\mathcal{D}$ is the pretopos completion of the syntactic category of some $L_{\lambda \kappa }$-theory. We start with the observation:

\begin{proposition}
\label{notmono}
Let $\mathcal{D}$ be an exact category (e.g.~a pretopos) and $Y\xrightarrow{\langle q_1,q_2\rangle} X\times X$ be a (not necessarily monic) map. If it satisfies the axioms of internal equivalence relations, namely:
\begin{enumerate}
    \item (reflexivity) There is a map $r':X\to Y$ with $q_1r'=q_2r'=1_X$.
    \item (symmetry) There is an effective epimorphism $s':Y\to Y$ such that $\langle q_1,q_2\rangle \circ s' = \langle q_2,q_1\rangle $.
    \item (transitivity) Let 
\[\begin{tikzcd}
	{Y\times _X Y} & Y \\
	Y & X
	\arrow["{q_2'}", from=1-1, to=1-2]
	\arrow["{q_1}", from=1-2, to=2-2]
	\arrow["{q_1'}"', from=1-1, to=2-1]
	\arrow["{q_2}", from=2-1, to=2-2]
\end{tikzcd}\]
    be a pullback diagram. There is an effective epimorphism $t':Y\times _XY\to Y$ such that $\langle q_1q_1',q_2q_2'\rangle =\langle q_1,q_2\rangle \circ t'$.
\end{enumerate}
then the coequalizer of $f$ and $g$ exists.
\end{proposition}

\begin{proof}
We will prove that in the image factorization $\langle q_1,q_2\rangle :Y\xtwoheadrightarrow{p} R\xhookrightarrow{\langle f,g\rangle } X\times X$ the subobject $R$ is an internal equivalence relation on $X$. For reflexivity we can take $r=pr'$. Using that effective epi-mono factorizations are unique we get a map:
\[\begin{tikzcd}
	Y && R && {X\times X} \\
	\\
	Y && R && {X\times X}
	\arrow["{s'}"', two heads, from=1-1, to=3-1]
	\arrow["p", two heads, from=1-1, to=1-3]
	\arrow["p", two heads, from=3-1, to=3-3]
	\arrow["{\langle f,g\rangle}", hook, from=1-3, to=1-5]
	\arrow["{\langle g,f\rangle }"{description}, hook, from=3-3, to=1-5]
	\arrow["s"', dashed, from=1-3, to=3-3]
	\arrow["{\langle f,g\rangle}"', hook, from=3-3, to=3-5]
	\arrow["{\langle \pi _2,\pi _1\rangle }"', from=3-5, to=1-5]
\end{tikzcd}\]
proving symmetry. Finally we have

\[\begin{tikzcd}
	{Y\times _XY} && \bullet && Y \\
	& pb && pb \\
	\bullet && {R\times _XR} && R \\
	& pb && pb \\
	Y && R && X
	\arrow["g", from=5-3, to=5-5]
	\arrow["f"', from=3-5, to=5-5]
	\arrow["p", two heads, from=5-1, to=5-3]
	\arrow["p"', two heads, from=1-5, to=3-5]
	\arrow["{g'}", from=3-3, to=3-5]
	\arrow["{f'}", from=3-3, to=5-3]
	\arrow[two heads, from=3-1, to=3-3]
	\arrow[from=3-1, to=5-1]
	\arrow[two heads, from=1-3, to=3-3]
	\arrow[from=1-3, to=1-5]
	\arrow[two heads, from=1-1, to=1-3]
	\arrow[two heads, from=1-1, to=3-1]
	\arrow["{p'}"{description, pos=0.7}, curve={height=12pt}, dashed, two heads, from=1-1, to=3-3]
	\arrow["{q_1'}"{description}, curve={height=12pt}, from=1-1, to=5-1]
	\arrow["{q_2'}"{description}, curve={height=-12pt}, from=1-1, to=1-5]
	\arrow["{q_2}"{description}, curve={height=12pt}, from=5-1, to=5-5]
	\arrow["{q_1}"{description}, curve={height=-12pt}, from=1-5, to=5-5]
\end{tikzcd}\]
and using the uniqueness of image factorizations again, we get

\[\begin{tikzcd}
	&&&& Y \\
	&&&&& R \\
	&& {R\times _XR} \\
	{Y\times _XY} &&&& {X\times X}
	\arrow["{\langle q_1q_1',q_2q_2'\rangle }"', from=4-1, to=4-5]
	\arrow["{p'}", two heads, from=4-1, to=3-3]
	\arrow["{\langle ff',gg'\rangle}", from=3-3, to=4-5]
	\arrow["{\langle q_1,q_2\rangle}"{description, pos=0.6}, from=1-5, to=4-5]
	\arrow["p", two heads, from=1-5, to=2-6]
	\arrow["{\langle f,g\rangle}", hook, from=2-6, to=4-5]
	\arrow["{t'}"{description}, curve={height=-24pt}, two heads, from=4-1, to=1-5]
	\arrow["t"{description}, curve={height=-12pt}, dashed, from=3-3, to=2-6]
\end{tikzcd}\]
proving transitivity. Therefore the coequalizer of $f$ and $g$ exists. Since $p$ is (effective) epi, this is also the coequalizer of $q_1$ and $q_2$.
\end{proof}

\begin{theorem}
\label{hasfiltcolim}
Let $\mathcal{D} $ be a $\lambda $-pretopos. Then it has filtered colimits of size $<\lambda $, which are pullback-stable. Moreover any $\lambda $-geometric functor between $\lambda $-pretoposes preserves filtered colimits of size $<\lambda $.
\end{theorem}

\begin{proof}
Let $d_{\bullet }:I\to \mathcal{D}$ be a filtered diagram. We will prove that the pair
\[\begin{tikzcd}
	{\bigsqcup _{i\xrightarrow{f} k \xleftarrow{g}j} pb(d_f , d_g)} && {\bigsqcup _{i} d_i}
	\arrow["{\sqcup p_{f,g}}", shift left=1, from=1-1, to=1-3]
	\arrow["{\sqcup q_{f,g}}"', shift right=1, from=1-1, to=1-3]
\end{tikzcd}\]
satisfies the axioms of Proposition \ref{notmono}, and hence its coequalizer exists. Here we take the coproducts over the pullback diagrams

\[\begin{tikzcd}
	{pb(d_f,d_g)} && {d_j} \\
	\\
	{d_i} && {d_k}
	\arrow["{d_f}"', from=3-1, to=3-3]
	\arrow["{d_g}", from=1-3, to=3-3]
	\arrow["{q_{f,g}}"', from=1-1, to=3-1]
	\arrow["{p_{f,g}}", from=1-1, to=1-3]
\end{tikzcd}\]

Reflexivity: Take $\bigsqcup _i d_i \xrightarrow{\bigsqcup 1_{d_i} } \bigsqcup _{i\xrightarrow{f} k \xleftarrow{g}j} pb(d_f ,d_g)$ which maps $d_i$ to the pullback of $(1_{d_i},1_{d_i})$.

Symmetry: Take 
\[
\bigsqcup _{i\xrightarrow{f} k \xleftarrow{g}j} pb(d_f, d_g)\xrightarrow{\bigsqcup 1_{pb(d_f,d_g)} } \bigsqcup _{i\xrightarrow{f} k \xleftarrow{g}j} pb(d_f,d_g) 
\]
which maps $pb(d_f,d_g)$ to $pb(d_g,d_f)$ identically (or more precisely with the induced isomorphism). This is an isomorphism, in particular an effective epi.

Transitivity: By (the infinitary analogue of) Proposition 8, Lecture 7. in \cite{lurie} we have that 
\[
\adjustbox{scale=0.9}{
\begin{tikzcd}
	{\bigsqcup\limits _{i\xrightarrow{f} k \xleftarrow{g}j}\bigsqcup\limits _{i'\xrightarrow{f'} k' \xleftarrow{g'}i}  pb(q_{f , g},p_{{f'} , {g'}})} && {\bigsqcup\limits _{i\xrightarrow{f} k \xleftarrow{g}j} pb(d_f,d_g)} \\
	\\
	{\bigsqcup\limits _{i\xrightarrow{f} k \xleftarrow{g}j} pb(d_f,d_g)} && {\bigsqcup _{i} d_i}
	\arrow["{\sqcup q_{f,g}}"', from=3-1, to=3-3]
	\arrow["{\sqcup p_{f,g}}", from=1-3, to=3-3]
	\arrow["r"', from=1-1, to=3-1]
	\arrow["s", from=1-1, to=1-3]
\end{tikzcd}
}
\]
is a pullback. 

In other terms for each zig-zag $(f',g',f,g)$ we take
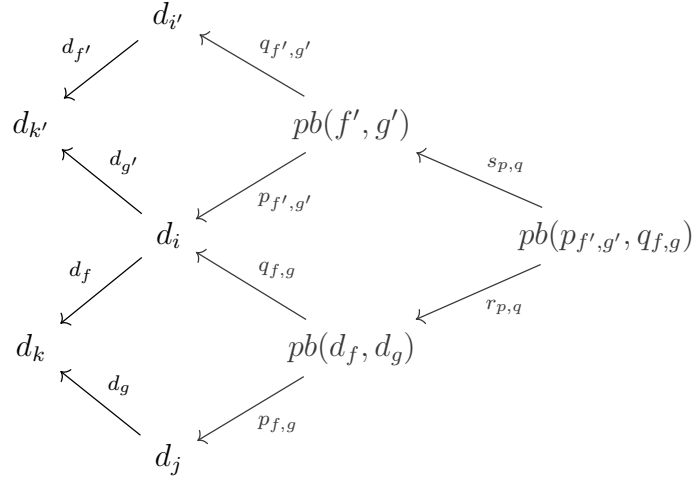
\[
\adjustbox{scale=1}{
\begin{tikzcd}
	& {d_{i'}} \\
	{d_{k'}} && \textcolor{rgb,255:red,56;green,56;blue,56}{pb(f',g')} \\
	& {d_i} && \textcolor{rgb,255:red,56;green,56;blue,56}{pb(p_{{f'},{g'}},q_{f,g})} \\
	{d_k} && \textcolor{rgb,255:red,56;green,56;blue,56}{pb(d_f,d_g)} \\
	& {d_j}
	\arrow["{d_{f'}}"', from=1-2, to=2-1]
	\arrow["{d_{g'}}"', from=3-2, to=2-1]
	\arrow["d_f"', from=3-2, to=4-1]
	\arrow["d_g"', from=5-2, to=4-1]
	\arrow["{q_{{f'},{g'}}}"', color={rgb,255:red,56;green,56;blue,56}, from=2-3, to=1-2]
	\arrow["{p_{{f'},{g'}}}", color={rgb,255:red,56;green,56;blue,56}, from=2-3, to=3-2]
	\arrow["{q_{f,g}}"', color={rgb,255:red,56;green,56;blue,56}, from=4-3, to=3-2]
	\arrow["{p_{f,g}}", color={rgb,255:red,56;green,56;blue,56}, from=4-3, to=5-2]
	\arrow["{s_{p,q}}"', color={rgb,255:red,56;green,56;blue,56}, from=3-4, to=2-3]
	\arrow["{r_{p,q}}", color={rgb,255:red,56;green,56;blue,56}, from=3-4, to=4-3]
\end{tikzcd}
}
\]

We shall define a map $t$ from this coproduct to $\bigsqcup _{i\xrightarrow{f} k \xleftarrow{g}j} pb(d_f,d_g)$. We will do this by just extending the above diagram to

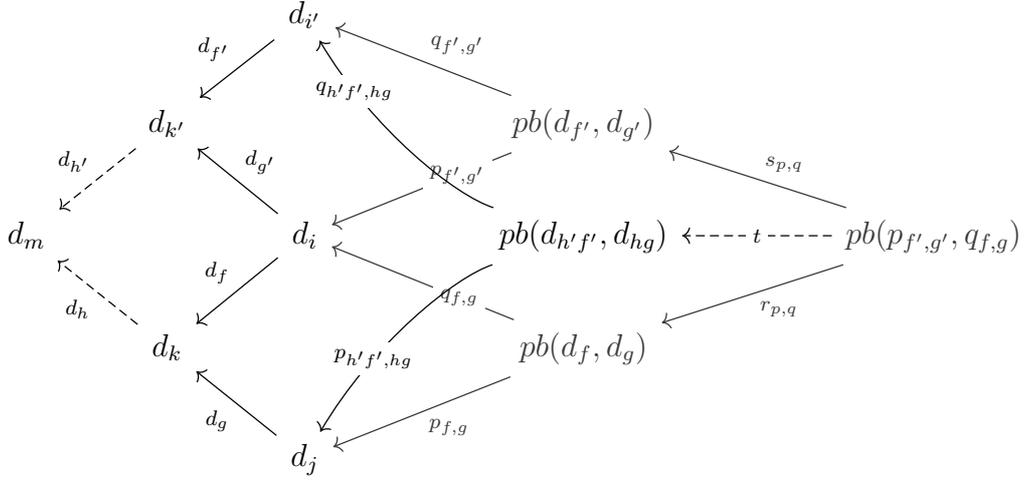
\[
\adjustbox{scale=1}{
\begin{tikzcd}
	&& {d_{i'}} \\
	& {d_{k'}} &&& \textcolor{rgb,255:red,56;green,56;blue,56}{pb(d_{f'},d_{g'})} \\
	{d_m} && {d_i} && {pb(d_{h'f'},d_{hg})} && \textcolor{rgb,255:red,56;green,56;blue,56}{pb(p_{{f'},{g'}},q_{f,g})} \\
	& {d_k} &&& \textcolor{rgb,255:red,56;green,56;blue,56}{pb(d_f,d_g)} \\
	&& {d_j}
	\arrow["{d_{f'}}"', from=1-3, to=2-2]
	\arrow["{d_{g'}}"', from=3-3, to=2-2]
	\arrow["{d_f}"', from=3-3, to=4-2]
	\arrow["{d_g}", from=5-3, to=4-2]
	\arrow["{q_{{f'},{g'}}}"', color={rgb,255:red,56;green,56;blue,56}, from=2-5, to=1-3]
	\arrow["{p_{{f'},{g'}}}"{description, pos=0.3}, color={rgb,255:red,56;green,56;blue,56}, from=2-5, to=3-3]
	\arrow["{q_{f,g}}"{description, pos=0.3}, color={rgb,255:red,56;green,56;blue,56}, from=4-5, to=3-3]
	\arrow["{p_{f,g}}", color={rgb,255:red,56;green,56;blue,56}, from=4-5, to=5-3]
	\arrow["{s_{p,q}}"', color={rgb,255:red,56;green,56;blue,56}, from=3-7, to=2-5]
	\arrow["{r_{p,q}}", color={rgb,255:red,56;green,56;blue,56}, from=3-7, to=4-5]
	\arrow["{d_{h'}}"', dashed, from=2-2, to=3-1]
	\arrow["{d_h}", dashed, from=4-2, to=3-1]
	\arrow["{q_{{h'}{f'},hg}}"{description, pos=0.8}, curve={height=-18pt}, from=3-5, to=1-3]
	\arrow["{p_{{h'}{f'},hg}}"{description, pos=0.7}, curve={height=18pt}, from=3-5, to=5-3]
	\arrow["t"{description}, dashed, from=3-7, to=3-5]
\end{tikzcd}
}
\]
and then take the induced map. The commutativity condition in the axiom is precisely the commutativity of the two triangles in which $t$ appears. To make sure that $t$ is surjective we just have to assume that zig-zags of the form $\bullet \to \bullet = \bullet = \bullet \leftarrow \bullet $ were chosen to be completed with identities.

This proves that the coequalizer
\[\begin{tikzcd}
	{\bigsqcup _{i\xrightarrow{f} k \xleftarrow{g}j} pb(d_f,d_g)} && {\bigsqcup _{i} d_i} && c
	\arrow["{\sqcup p_{f,g}}", shift left=1, from=1-1, to=1-3]
	\arrow["{\sqcup q_{f,g}}"', shift right=1, from=1-1, to=1-3]
	\arrow["\varphi", from=1-3, to=1-5]
\end{tikzcd}\]
exists. To prove that $(\restr{\varphi }{d_i})_i$ is the colimit cocone we need that given $(\psi _i:d_i\to c')_i$ they form a cocone over $d_{\bullet }$ iff $\bigsqcup _i \psi _i$ coequalizes the above pair of maps, i.e.~iff for any span $i\xrightarrow{f}k\xleftarrow{g}j$ we have that the left half of 
\[\begin{tikzcd}
	& {d_i} \\
	{pb(d_f,d_g)} & {c'} & {d_k} \\
	& {d_j}
	\arrow["d_f", from=1-2, to=2-3]
	\arrow["d_g"', from=3-2, to=2-3]
	\arrow["{q_{f,g}}", from=2-1, to=1-2]
	\arrow["{p_{f,g}}"', from=2-1, to=3-2]
	\arrow["{\psi_i}"', from=1-2, to=2-2]
	\arrow["{\psi _j}", from=3-2, to=2-2]
\end{tikzcd}\]
commutes. $\Rightarrow $: In this case the left square is the gluing of two commutative squares. $\Leftarrow $: Take $g=1_{k}$.

This coequalizer is clearly preserved by any $\lambda $-geometric functor between $\lambda $-pretoposes.

We are left to show that filtered colimits are pullback-stable. When we form pullbacks along $\alpha $ in
\[
\adjustbox{scale=1}{
\begin{tikzcd}
	&& {d_i^*} \\
	{pb(d_f,d_g)^*} &&& {d_k^*} && {c'} \\
	& {d_j^*} \\
	&& {d_i} \\
	{pb(d_f,d_g)} &&& {d_k} && c \\
	& {d_j}
	\arrow["{d_g}"{description}, color={rgb,255:red,92;green,92;blue,214}, from=6-2, to=5-4]
	\arrow["{d_f}"{description}, color={rgb,255:red,92;green,92;blue,214}, from=4-3, to=5-4]
	\arrow[curve={height=6pt}, from=6-2, to=5-6]
	\arrow[from=5-4, to=5-6]
	\arrow[curve={height=-6pt}, from=4-3, to=5-6]
	\arrow["\alpha", from=2-6, to=5-6]
	\arrow[draw={rgb,255:red,92;green,92;blue,214}, from=2-4, to=5-4]
	\arrow[draw={rgb,255:red,214;green,153;blue,92}, from=1-3, to=4-3]
	\arrow[draw={rgb,255:red,92;green,92;blue,214}, from=3-2, to=6-2]
	\arrow["{d_g^*}"{description, pos=0.3}, color={rgb,255:red,92;green,92;blue,214}, from=3-2, to=2-4]
	\arrow["{d_f^*}"{description}, color={rgb,255:red,92;green,92;blue,214}, from=1-3, to=2-4]
	\arrow[curve={height=-6pt}, from=1-3, to=2-6]
	\arrow[from=2-4, to=2-6]
	\arrow["{q_{f,g}}"{pos=0.3}, from=5-1, to=4-3]
	\arrow["{p_{f,g}}"', from=5-1, to=6-2]
	\arrow[from=2-1, to=5-1]
	\arrow[curve={height=6pt}, from=3-2, to=2-6]
	\arrow["{\overline{q_{f,g}}}", from=2-1, to=1-3]
	\arrow["{\overline{p_{f,g}}}", dashed, from=2-1, to=3-2]
\end{tikzcd}
}
\]
the blue squares are pullbacks by the pasting law, and if we form pullback along the orange map we get that the top face of the cube is a pullback. Using this, plus the distributive law Proposition 8, Lecture 7 of \cite{lurie} again, we identify the pullback of our coequalizer as

\[\begin{tikzcd}
	{A'=\bigsqcup \limits_{i\xrightarrow{f}k\xleftarrow{g}j} pb(d_f^*,d_g^*)} && {X'=\bigsqcup _i d_i^*} && {c'} \\
	& pb && pb \\
	{A=\bigsqcup \limits_{i\xrightarrow{f}k\xleftarrow{g}j} pb(d_f,d_g)} && {X=\bigsqcup _i d_i} && c
	\arrow["{Q=\bigsqcup q_{f,g}}"', shift right=1, from=3-1, to=3-3]
	\arrow["{P=\bigsqcup p_{f,g}}", shift left=1, from=3-1, to=3-3]
	\arrow["{\Phi =\bigsqcup \varphi _i}", two heads, from=3-3, to=3-5]
	\arrow["{H=\alpha}", from=1-5, to=3-5]
	\arrow["{\Phi '=\bigsqcup \varphi _i^*}", two heads, from=1-3, to=1-5]
	\arrow["{H'=\bigsqcup \alpha _i'}"{description}, from=1-3, to=3-3]
	\arrow["{P'=\bigsqcup \overline{p_{f,g}}}", shift left=1, from=1-1, to=1-3]
	\arrow["{Q'=\bigsqcup \overline{q_{f,g}}}"', shift right=1, from=1-1, to=1-3]
	\arrow[from=1-1, to=3-1]
\end{tikzcd}\]

So we have to check that the first row is a coequalizer. We know that the image factorization $\langle P,Q \rangle : A\xtwoheadrightarrow{\psi }B\xhookrightarrow{\langle F,G\rangle }X\times X $  yields the kernel pair of $\Phi $, and we would like to see that the image factorization of $\langle P',Q'\rangle $ yields the kernel pair of $\Phi '$. This can be checked on

\[\begin{tikzcd}
	&&&& {X'} \\
	{A'} && {\widetilde{B}} &&&& {c'} \\
	&&& {X'} \\
	&&&& X \\
	A && B &&&& c \\
	&&& X
	\arrow["P"{description}, from=5-1, to=6-4]
	\arrow["\Phi"{description}, from=6-4, to=5-7]
	\arrow["Q"{description, pos=0.6}, from=5-1, to=4-5]
	\arrow["\Phi"{description}, from=4-5, to=5-7]
	\arrow["{Q'}"{description}, from=2-1, to=1-5]
	\arrow["{P'}"{description}, from=2-1, to=3-4]
	\arrow["{\Phi '}"{description, pos=0.7}, from=3-4, to=2-7]
	\arrow["{\Phi'}"{description}, from=1-5, to=2-7]
	\arrow["H"{description}, from=2-7, to=5-7]
	\arrow["{H'}"{description, pos=0.3}, from=1-5, to=4-5]
	\arrow["{H'}"{description, pos=0.2}, from=3-4, to=6-4]
	\arrow[from=2-1, to=5-1]
	\arrow["G"{description, pos=0.7}, from=5-3, to=4-5]
	\arrow["F"{description}, from=5-3, to=6-4]
	\arrow["{G'}"{description}, from=2-3, to=1-5]
	\arrow["{F'}"{description}, from=2-3, to=3-4]
	\arrow[dashed, from=2-3, to=5-3]
	\arrow["\psi"{description}, two heads, from=5-1, to=5-3]
	\arrow["{\widetilde{\psi}}"{description}, dashed, from=2-1, to=2-3]
\end{tikzcd}\]
Indeed, by the pasting law again, every face of the right inner cube is a pullback. The square with two dashed arrows commutes, as $\langle F,G\rangle $ is monic. Then this square is also a pullback and hence $\widetilde{\psi }$ is effective epi as we wished. 
\end{proof}

\begin{theorem}
\label{commfinlim}
Let $\mathcal{D}$ be a $\lambda $-pretopos. Then filtered colimits of size $<\lambda $ commute with finite limits.
\end{theorem}

\begin{proof}
Let $J$ be the Grothendieck-topology on $\mathcal{D}$ generated by the $<\lambda $ effective epimorphic families. This topology is subcanonical and the Yoneda-embedding $y:\mathcal{D}\to Sh(\mathcal{D},J)$ is a (fully faithful) $\lambda $-geometric functor (see e.g.~Proposition 6.1.4.~and 6.1.6.~of \cite{kamsmathesis}). By the above theorem it preserves and reflects finite limits and filtered colimits of size $<\lambda $. But in $Sh(\mathcal{D},J)$ finite limits commute with filtered colimits.
\end{proof}

We recall the following definition from \cite{barrex}:

\begin{definition}
A category $\mathcal{D}$ is $\kappa $-regular if it is regular, has $<\kappa $-limits and the transfinite cocomposition of $<\kappa $-many effective epimorphisms is an effective epi, i.e.~given an ordinal $\alpha <\kappa $ and a limit-preserving functor $F:\alpha ^{op}\to \mathcal{D}$, such that each $i\to j$ arrow in $\alpha $ is mapped to an effective epi $Fi \twoheadleftarrow Fj$, the limit cone $(Fi\leftarrow F\alpha )_{i<\alpha }$ consists of effective epis. 

$\mathcal{D}$ is $\kappa $-exact if it is $\kappa $-regular and Barr-exact (i.e.~internal equivalence relations have a quotient).
\end{definition}

\begin{proposition}
In a $\kappa $-regular category the product of $<\kappa $-many effective epimorphisms is an effective epi.
\end{proposition}

\begin{proof}
The following is a limit diagram:
\[\begin{tikzcd}
	&&& {x_0\times x_1\times x_2\times \dots} \\
	\\
	{y_0\times y_1\times y_2\times \dots} && {x_0\times y_1\times  y_2\times \dots} & {x_0\times x_1\times y_2\times \dots}
	\arrow["{f_0\times 1\times 1\times \dots }"', two heads, from=3-3, to=3-1]
	\arrow["{1\times f_1\times 1\times \dots}"', two heads, from=3-4, to=3-3]
	\arrow["{f_0\times f_1\times f_2\times \dots}"', curve={height=12pt}, from=1-4, to=3-1]
	\arrow["{1\times f_1\times f_2\times \dots}"{description, pos=0.4}, from=1-4, to=3-3]
	\arrow["{1\times 1\times f_2\times \dots}"{description, pos=0.7}, from=1-4, to=3-4]
\end{tikzcd}\]
and the horizontal maps are effective epimorphisms by pullback-stability. 
\end{proof}

We repeat Definition 2.1.2.~of \cite{inflog} but with allowing different cardinals for the $\exists $ and for the $\vee $ properties.

\begin{definition}
\label{klpretop}
Fix regular cardinals $\lambda \geq \kappa \geq \aleph _0$.

Let $P$ be a rooted tree such that the degree of any vertex is $<\lambda $ and the length of any branch is $<\kappa $. Call these $(\lambda ,\kappa )$-trees. We say that a diagram $F:P^{op}\to \mathcal{D}$ is proper if the $1$-neighbourhood of any vertex is mapped to a jointly epimorphic family and $F$ restricted to any branch is limit-preserving. We say that $F$ is completely proper if it is proper and the transfinite cocomposition of the branches form an effective epimorphic family on the root. $\mathcal{D}$ is $(\lambda ,\kappa )$-regular if for any $(\lambda ,\kappa )$-tree $T$, any proper diagram $F:T^{op}\to \mathcal{D}$ is completely proper.

Let $\mathcal{D}$ be a $\lambda $-geometric category (resp.~a $\lambda $-pretopos). It is $(\lambda ,\kappa )$-coherent (resp.~it is a $(\lambda ,\kappa )$-pretopos) if it has $<\kappa $-limits, for any object $x$ in $Sub_{\mathcal{D}}(x)$ the $<\kappa $ meets distribute over $<\lambda $ joins, and it is $(\lambda ,\kappa )$-regular.
\end{definition}

We proved:

\begin{theorem}
\label{klpretopult}
Let $\lambda \geq \kappa \geq \aleph _0$ be regular cardinals, such that $\gamma <\kappa \Rightarrow 2^{\gamma }<\lambda $. Let $\mathcal{D}$ be a weakly Boolean $(\lambda ,\kappa )$-pretopos, in which $\{0\}\subseteq Sub_{\mathcal{D}}(1)$ is a $\kappa $-prime ideal: i.e.~if the intersection of $<\kappa $-many subobjects of $1$ is $0$ then one of them was $0$. 

Then for any coherent category $\mathcal{C}$, the set $\mathbf{Coh}(\mathcal{C},\mathcal{D})\hookrightarrow \mathcal{D}^{\mathcal{C}}$ is closed under $<\kappa $-ultraproducts. Moreover given subobjects $A_i\hookrightarrow X_i\hookleftarrow B_i$ in $\mathcal{D}$, we have $\faktor{\prod _i A_i}{U}\leq \faktor{\prod _i B_i}{U}$ iff $\{i:A_i\leq B_i\}\in U$.
\end{theorem}

\begin{definition}
\label{semancplte}
Let $M,N:\mathcal{C}\to \mathcal{D}$ be coherent functors. We say that $M$ and $N$ are elementary equivalent if for any subobjects $a\hookrightarrow x\hookleftarrow b$ in $\mathcal{C}$: $Ma\subseteq Mb$ iff $Na\subseteq Nb$. $\mathcal{C}\not \simeq *$ is semantically complete if any two $\mathcal{C}\to \mathbf{Set}$ coherent functors are elementary equivalent.
\end{definition}

\begin{proposition}
Let $\mathcal{C}$ be Boolean and $M,N:\mathcal{C}\to \mathcal{D}$ be coherent functors. If there is an elementary natural transformation $\alpha : M\Rightarrow N$ then $M$ and $N$ are elementary equivalent.
\end{proposition}

\begin{proof}
Given any $x\in \mathcal{C}$, we have $Nx=\emptyset $ iff $Mx=\emptyset $ as 
\[\begin{tikzcd}
	{M(\exists _!x)} && M1 \\
	&  \\
	{N(\exists _!x)} && N1
	\arrow[hook, from=3-1, to=3-3]
	\arrow[Rightarrow, no head, from=1-3, to=3-3]
	\arrow[hook, from=1-1, to=1-3]
	\arrow["\cong"', from=1-1, to=3-1]
\end{tikzcd}\]
is a pullback. But as $\mathcal{C}$ is Boolean we can use $Na\subseteq Nb$ iff $N(a\cap \neg b)=\emptyset $.
\end{proof}

Our next goal is to show that if $\mathcal{C}$ is semantically complete then $\mathbf{Coh}_e(\mathcal{C},\mathbf{Set})$ has the joint embedding property.

\begin{theorem}
Let $M:\mathcal{C}\to \mathbf{Set} $ be a coherent functor and $U\subseteq \mathcal{P}(I)$ be a $\kappa $-regular ultrafilter ($\kappa \geq |\mathcal{C}|\cdot \aleph _0$). Then $\faktor{M^I}{U}$ is $\kappa ^+$-universal, i.e.~given $N:\mathcal{C}\to \mathbf{Set}$ which is elementary equivalent to $M$ and $|N|=\sum _{x\in \mathcal{C}}|Nx|\leq \kappa $, there is an elementary natural transformation $\alpha :N\Rightarrow M $.
\end{theorem}

\begin{proof}
Fix a filtration $\langle \mathbb{S}\subseteq Ob(\mathcal{C}), (x\xhookrightarrow{\iota _x} s_{x,1}\times \dots s_{x,n_x})_{x\in \mathcal{C}} \rangle $. The (positive) elementary diagram $\Delta _e(N)$ of $N$ (wrt.~$\mathbb{S}$) is the set of tuples 
\begin{multline*}
    \varphi (s_1\times \dots \times s_m,(u_i)_i,(\vec{a}_{\varepsilon , 1},\dots \vec{a}_{\varepsilon, n_{\varepsilon }} )_{\varepsilon })  = \\ \langle s_1\times \dots \times s_m, (u_i\hookrightarrow \vec{s})_{i<k}, (\vec{a}_{\varepsilon ,i}\in N\vec{s})_{\varepsilon \in \{*,c\}^k , 1\leq i\leq n_{\varepsilon }} \rangle 
\end{multline*}
with $s_i\in \mathbb{S}$,  $\vec{a}_{\varepsilon ,i} \in Nu_1^{\varepsilon (1)} \cap \dots \cap Nu_k^{\varepsilon (k)}$, where $Nu^*=Nu$ and $Nu^c=N\vec{s}\setminus Nu$. (I.e.~we fix finitely many definable subsets in $N\vec{s}$ and note the position of finitely many tuples in the resulting partition.) Clearly $|\Delta _e(N)|=|N|\cdot |\mathcal{C}|\cdot \aleph _0\leq \kappa $.

Let $E\subseteq U$ be a collection of ultrasets such that $|E|=\kappa $ and for all $i\in I$: $\{J\in E : i\in J \}$ is finite. Fix an injection $f:\Delta _e(N)\to E$. Then $\gamma (i)=\{\varphi \in \Delta _e(N) : i\in f\varphi \}$ is finite. Given a fixed $\varphi $, the set $\{i\in I:\varphi \in \gamma (i)\}=f\varphi $ is an ultraset. 

If $Ns=\emptyset $ for some $s$, then $\alpha _s=\emptyset :\emptyset \to \emptyset $. Fix $i\in I$. We concatenate all tuples in $\gamma (i)$ to get just one long sequence of elements $a_1,\dots a_t\in Ns_1\times \dots \times Ns_t$. Of course $a_i=a_j$ can happen, so to remember this we introduce a subobject $N\delta \subseteq Ns_1\times \dots \times Ns_t$ which is the corresponding diagonal. Now we pull back all the prescribed subsets to get subsets of $Ns_1\times \dots \times Ns_t$, e.g.~if $(a_1,a_3)\in Nu\subseteq Ns_1\times Ns_3$ is in $\gamma (i)$ then we replace it with $(a_1,a_2,a_3,\dots )\in Nu\times Ns_2\times Ns_4\times \dots $, and similarly if $(a_1,a_3)\not \in Nu\subseteq Ns_1\times Ns_3$ is in $\gamma (i)$ then we replace it with $(a_1,a_2,a_3,\dots )\not \in Nu\times Ns_2\times Ns_4\times \dots $. So we only have to care about the position of one sequence wrt.~finitely many definable subsets of $N\vec{s}$, let's say it lies in $Nu_1=N\delta , Nu_2,\dots Nu_r$ and it is in the complement of $Nv_1,\dots Nv_{r'}$. 

The existence of our sequence $\vec{a}$ proves $Nu_1\cap \dots \cap Nu_r \not \subseteq Nv_1\cup \dots \cup Nv_{r'}$ hence $Mu_1\cap \dots \cap Mu_r \not \subseteq Mv_1\cup \dots \cup Mv_{r'}$. Take any $\vec{b}\in Mu_1\cap \dots \cap Mu_r\cap Mv_1^c \cap \dots \cap Mv_{r'}^c$. We define $b_j=(\alpha _{s_j}(a_j))_i$. As any $\varphi \in \Delta _e(N)$ appears in ultra-many $\gamma (i)$'s, we have defined $(\alpha _{s_j}(a_j))_i$ for ultra-many $i$'s and it was defined for each $a\in N_s$ (for each $s$).

By Proposition \ref{filtration} we shall check that given any $x\in \mathcal{C}$ the dashed arrow in 

\[
\adjustbox{scale=0.9}{
\begin{tikzcd}
	Nx && {Ns_{x,1}\times \dots \times Ns_{x,n_x}} \\
	\\
	{\faktor{M^I}{U} \ x} && {\faktor{M^I}{U}\ s_{x,1}\times \dots \times \faktor{M^I}{U}\ s_{x,n_x}}
	\arrow["{N \iota _x}", hook, from=1-1, to=1-3]
	\arrow["{\alpha _{x,1}\times \dots \times \alpha _{x,n_x}}", from=1-3, to=3-3]
	\arrow["{\faktor{M^I}{U}\ \iota _x}"', hook, from=3-1, to=3-3]
	\arrow[dashed, from=1-1, to=3-1]
\end{tikzcd}
}
\]
exists and the square is a pullback. If some $Ns_{x,i}=\emptyset $ then there's nothing to check. Otherwise it follows from the fact that for any $\vec{a}\in N\vec{s}$, the prescription $a\in Nx$ or $a\in Nx^c$ appears in ultra-many $\gamma (i)$'s.
\end{proof}

As a corollary we got:

\begin{theorem}
\label{semancplthasjep}
If $\mathcal{C}$ is semantically complete then $\mathbf{Coh}_e(\mathcal{C},\mathbf{Set})$ has the joint embedding property. (When $\mathcal{C}$ is Boolean the converse also holds.)
\end{theorem}


\section{Measuring positive closedness}

Recall the definition of the type-space functor $S_{\mathcal{C}}:\mathcal{C}\xrightarrow{Sub_{\mathcal{C}}} \mathbf{DLat}^{op}\xrightarrow{Spec} \mathbf{Spec}$ associated to a coherent category $\mathcal{C}$ (cf.~Definition \ref{typespacedef}). Given a model (i.e.~a coherent functor) $M:\mathcal{C}\to \mathbf{Set}$, there is a natural transformation 
\[\begin{tikzcd}
	{\mathcal{C}} &&&& {\mathbf{Set}} \\
	&& {\mathbf{Spec}}
	\arrow[""{name=0, anchor=center, inner sep=0}, "M", curve={height=-18pt}, from=1-1, to=1-5]
	\arrow["{S_{\mathcal{C}}}"', curve={height=12pt}, from=1-1, to=2-3]
	\arrow["U"', curve={height=12pt}, from=2-3, to=1-5]
	\arrow["tp", shorten <=6pt, Rightarrow, from=0, to=2-3]
\end{tikzcd}\]
whose component $tp_x:M(x)\Rightarrow \overline{S_{\mathcal{C}}}(x)$ maps $a\in M(x)$ to $tp(a/\emptyset ):=S_{M,x}(\widehat{a})$ (where $U$ is the forgetful functor, $\overline{S_{\mathcal{C}}}$ is $U\circ S_{\mathcal{C}}$, $S_{M,x}$ maps an ultrafilter of $Sub_{\mathbf{Set}}(Mx)$ to its preimage along $M:Sub_{\mathcal{C}}(x)\to Sub_{\mathbf{Set}}(Mx)$, finally $\widehat{a}$ is the principal ultrafilter on $a$).

To check that $tp$ is a natural transformation we need $S_{M,y}(\widehat{M(f)(a)})=$ $(f^*)^{-1}(S_{M,x}(\widehat{a}))$ which is $M^{-1}(M(f)^*)^{-1}(\widehat{a})=(f^*)^{-1}M^{-1}(\widehat{a})$ which follows by the commutativity of
\[\begin{tikzcd}
	{Sub_{\mathcal{C}}(x)} && {Sub_{\mathcal{C}}(y)} \\
	\\
	{Sub_{\mathbf{Set}}(Mx)} && {Sub_{\mathbf{Set}}(My)}
	\arrow["{f^*}"', from=1-3, to=1-1]
	\arrow["{(Mf)^*}"', from=3-3, to=3-1]
	\arrow["M", from=1-3, to=3-3]
	\arrow["M"', from=1-1, to=3-1]
\end{tikzcd}\]

Note that for an elementary natural transformation $\alpha :M\Rightarrow N$ the following commutes:

\[\begin{tikzcd}
	Mx && {S_{\mathbf{Set}}(Mx)} \\
	&&&& {S_{\mathcal{C}}(x)} \\
	Nx && {S_{\mathbf{Set}}(Nx)}
	\arrow["{\alpha _x}"', from=1-1, to=3-1]
	\arrow["{\widehat{\phantom{x}}}", from=1-1, to=1-3]
	\arrow["{(\alpha _x^*)^{-1}}", from=1-3, to=3-3]
	\arrow["{\widehat{\phantom{x}}}", from=3-1, to=3-3]
	\arrow["{M^{-1}}", from=1-3, to=2-5]
	\arrow["{N^{-1}}"', from=3-3, to=2-5]
\end{tikzcd}\]
which means $tp_N \circ \alpha =tp_M$.

The converse is also true:

\begin{proposition}
$\alpha :M\Rightarrow N$ is elementary iff $tp_N \circ \alpha =tp_M$.
\label{converse}
\end{proposition}

\begin{proof}
If $\alpha $ is not elementary then there is a subobject $a\hookrightarrow x$ such that $\alpha _x^*(N(a\hookrightarrow x))\setminus M(a\hookrightarrow x)\neq \emptyset$. If $s$ is an element of it, then $N^{-1}(\alpha _x^*)^{-1}(\widehat{s})\neq M^{-1}(\widehat{s})$.
\end{proof}

We would like to classify all natural transformations $M\Rightarrow \overline{S_{\mathcal{C}}}$. 

\begin{definition}
Define $\mathbf{Coh}_e/\mathbf{Set}$ to be the strict 2-category whose objects are $\mathcal{C}\to \mathbf{Set}$ coherent functors, the 1-cells are pairs $(F:\mathcal{C}\to \mathcal{D}, \eta _F:M\Rightarrow NF)$ where $F$ is a coherent functor and $\eta _F$ is an arbitrary natural transformation, and the 2-cells are elementary natural transformations $\alpha :F\Rightarrow G$ making 

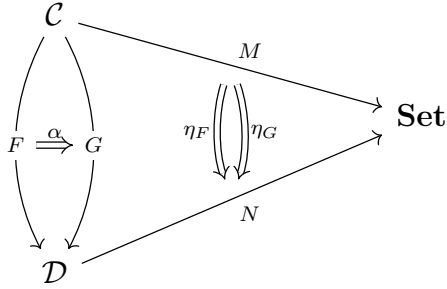
\[\begin{tikzcd}
	{\mathcal{C}} \\
	&&&& {\mathbf{Set}} \\
	\\
	{\mathcal{D}}
	\arrow[""{name=0, anchor=center, inner sep=0}, "F"{description}, curve={height=18pt}, from=1-1, to=4-1]
	\arrow[""{name=1, anchor=center, inner sep=0}, "G"{description}, curve={height=-18pt}, from=1-1, to=4-1]
	\arrow[""{name=2, anchor=center, inner sep=0}, "M", from=1-1, to=2-5]
	\arrow[""{name=3, anchor=center, inner sep=0}, "N"', from=4-1, to=2-5]
	\arrow["\alpha", shorten <=7pt, shorten >=7pt, Rightarrow, from=0, to=1]
	\arrow["{\eta _G}", curve={height=-6pt}, shorten <=7pt, shorten >=7pt, Rightarrow, from=2, to=3]
	\arrow["{\eta _F}"', curve={height=6pt}, shorten <=7pt, shorten >=7pt, Rightarrow, from=2, to=3]
\end{tikzcd}\]
commutative. We have an evident forgetful functor $\mathbf{Coh}_e/\mathbf{Set}\to \mathbf{Coh}_e$.
\end{definition}

\begin{definition}
\label{ldef}
We define a 2-functor $L:\mathbf{Coh}_e/\mathbf{Set} \to \mathbf{DLat}$ as follows: Given a 2-cell as pictured above, we take $LM$ to be the colimit
\[
colim((\int M)^{op}\to \mathcal{C}^{op}\to \mathbf{DLat})
\]
and $L\eta _F =L\eta _G$ to be the map induced by

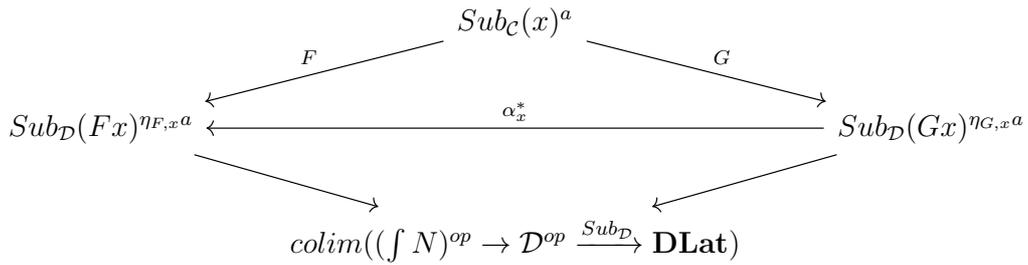
\[
\adjustbox{width=\textwidth}{
\begin{tikzcd}
	& {Sub_{\mathcal{C}}(x)^a} \\
	{Sub_{\mathcal{D}}(Fx)^{\eta _{F,x}a}} && {Sub_{\mathcal{D}}(Gx)^{\eta _{G,x}a}} \\
	& {colim((\int N)^{op}\to \mathcal{D}^{op}\xrightarrow{Sub_{\mathcal{D}}}\mathbf{DLat})}
	\arrow["F"', from=1-2, to=2-1]
	\arrow["G", from=1-2, to=2-3]
	\arrow["{\alpha _x^*}"', from=2-3, to=2-1]
	\arrow[from=2-1, to=3-2]
	\arrow[from=2-3, to=3-2]
\end{tikzcd}
}
\]
\end{definition}

\begin{remark}
$LM$ is the weighted colimit $colim^M Sub_{\mathcal{C}}=Lan_Y Sub_{\mathcal{C}}(M)$.
\end{remark}

\begin{proposition}
\label{lcdef}
$L$ restricted to the fiber over $\mathcal{C} $ is 
\[
L_{\mathcal{C}}:=Coh(\mathcal{C},\mathbf{Set})\hookrightarrow \mathbf{Set}^{\mathcal{C}}\xrightarrow{Lan_Y Sub_{\mathcal{C}}} \mathbf{DLat}
\]
\end{proposition}

\begin{proof}
Given $\eta :M\Rightarrow N$ we see it as the 1-cell $(1_{\mathcal{C}},\eta )$. As $Lan_YSub_{\mathcal{C}}$ preserves colimits, $Lan_YSub_{\mathcal{C}}(\mathcal{C}(x',-)\xRightarrow{f^*} \mathcal{C}(x,-))$ is $Sub_{\mathcal{C}}(x')\xrightarrow{f^*}Sub_{\mathcal{C}}(x)$, and since for each $(x,a)\in \int M$ we have a commutative square
\[\begin{tikzcd}
	{\mathcal{C}(x,-)} && M \\
	\\
	{\mathcal{C}(x,-)} && N
	\arrow["{1_x^*}", Rightarrow, no head, from=1-1, to=3-1]
	\arrow["{1_x\mapsto a}", Rightarrow, from=1-1, to=1-3]
	\arrow["\eta", Rightarrow, from=1-3, to=3-3]
	\arrow["{1_x\mapsto \eta_x(a)}", Rightarrow, from=3-1, to=3-3]
\end{tikzcd}\]
we get that $Sub_{\mathcal{C}}(x)^a\to L_{\mathcal{C}}M \xrightarrow{Lan_YSub_{\mathcal{C}}\eta } L_{\mathcal{C}}N$ and $Sub_{\mathcal{C}}(x)^a\to L_{\mathcal{C}}M \xrightarrow{L_{\mathcal{C}}\eta } L_{\mathcal{C}}N$ both equal $Sub_{\mathcal{C}}(x)^a\xrightarrow{id} Sub_{\mathcal{C}}(x)^{\eta _x(a)}\to L_{\mathcal{C}}N$ and therefore  $L_{\mathcal{C}}(\eta )=L((1_{\mathcal{C}},\eta ))=Lan_YSub_{\mathcal{C}}(\eta )$.
\end{proof}

\begin{proposition}
Given a coherent category $\mathcal{C}$ we have isomorphisms
\[
\mathbf{DLat}(LM,K)\cong Nat(M,\mathbf{DLat}(Sub_{\mathcal{C}}(-),K))
\]
natural in $M\in \mathbf{Set}^{\mathcal{C}}$ and $K\in \mathbf{DLat}$. Therefore $L_{\mathcal{C}}$ preserves filtered colimits (since $Coh(\mathcal{C},\mathbf{Set})\hookrightarrow \mathbf{Set}^{\mathcal{C}}$ preserves them and $Lan_YSub_{\mathcal{C}}$ is a left adjoint).
\end{proposition}

\begin{proof}
This can be found e.g.~as Theorem I.5.2.~of \cite{sheaves}.
\end{proof}

\begin{remark}
Taking $K=\mathbf{2}$ we get an isomorphism $Spec(LM)\cong Nat(M,\overline{S_{\mathcal{C}}})$, using the natural isomorphism $Spec\cong \mathbf{DLat}(-,\mathbf{2})$. Given a prime filter $p=(LM\xrightarrow{P} \mathbf{2})^{-1}(1)$ the corresponding natural transformation $\widehat{p}:M\Rightarrow \overline{S_{\mathcal{C}}}$ is defined as $\widehat{p}_x(a)=(Sub_{\mathcal{C}}(x)^a \to LM\xrightarrow{P}\mathbf{2})^{-1}(1)$.

In particular if $p\subseteq q \in Spec(LM)$ are prime filters then $\widehat{p}_x(a)\subseteq \widehat{q}_x(a)$ for all $x\in \mathcal{C}$ and $a\in Mx$.
\label{correspondence}
\end{remark}

\begin{proposition}
\label{lmcomp}
$LM$ is computed as follows: its underlying set is 
\[
\faktor{\bigsqcup _{a\in Mx}Sub_{\mathcal{C}}(x)}{\sim}
\]
where $u\hookrightarrow x^a \sim v\hookrightarrow y^b$ iff there exists $\varphi \hookrightarrow x\times y$ such that $(a,b)\in M\varphi $ and $\varphi \cap u\times y =\varphi \cap x\times v$.

Note that $u\hookrightarrow x^a\sim u\times y\hookrightarrow x\times y^{(a,b)}$ for any $y\in \mathcal{C}$ and $b\in My$ therefore any two elements of $LM$ have representatives in a common $Sub(x)^a$. The lattice operations are defined through these representatives.

Given  $(F,\eta _F)$ as above, $L\eta _F :\faktor{\bigsqcup _{a\in Mx}Sub_{\mathcal{C}}(x)}{\sim} \to \faktor{\bigsqcup _{b\in Ny}Sub_{\mathcal{D}}(y)}{\sim}$  takes the equivalence class of $u\hookrightarrow x^a$ to that of $F(u\hookrightarrow x)^{\eta _{F,x} a}$.
\end{proposition}

\begin{proof}
As $M$ preserves finite limits $(\int M)^{op}$ is filtered, hence $u\hookrightarrow x^a \sim v\hookrightarrow y^b$ iff there is $w\hookrightarrow z$, $c\in Mz$ and arrows $f:z\to x$, $g:z\to y$ with $Mf(c)=a$ and $Mg(c)=b$, such that $w=f^*u=g^*v$. Taking the effective epi-mono factorization of $(f,g):z\to x\times y$ and using that if $h$ is effective epi then $h^*$ is injective, completes the argument. 
\end{proof}

\begin{remark}
We can repeat the construction in the language of model theory. Let $T\subseteq \Sigma _{\omega \omega }^g $ be a coherent theory and $M$ be a model of $T$.  $LM$ is the following distributive lattice: its underlying set is $\faktor{\{\langle \psi (x_1,\dots x_k), (a_i\in M_{X_i}) \rangle \}}{\sim }$ where $\psi $ is a positive existential $\Sigma $-formula, $\vec{a}$ is a tuple of elements having the appropriate sorts, and $\langle \psi (x_1,\dots x_k), (a_i\in M_{X_i}) \rangle \sim \langle \chi (y_1,\dots y_l), (b_i\in M_{Y_i}) \rangle $ iff there's a positive existential $\varphi (x_1,\dots x_k,y_1,\dots y_l)$ such that $T\models \varphi \wedge \psi \Leftrightarrow \varphi \wedge \chi $ and $M\models \varphi (a_1,\dots a_k,b_1,\dots b_l)$. Here $\vec{x}$ and $\vec{y}$ are disjoint (so if $y_i=x_j$ then we replace it with a fresh variable $y_i'$). $[\langle \psi (x_1,\dots x_k), (a_i\in M_{X_i}) \rangle ]\wedge [\langle \chi (y_1,\dots y_l), (b_i\in M_{Y_i}) \rangle ]=[\langle \psi \wedge \chi (x_1,\dots x_k,y_1,\dots y_l), (a_1,\dots a_k,b_1,\dots b_l) \rangle ]$ and similarly for $\vee $. ($\vec{x}$ and $\vec{y}$ are still disjoint, the top element is $[\top ]$, the bottom element is $[\bot ]$. $[\langle \psi (\vec{x}),\vec{a} \rangle ]\sim [\top ]$ iff $M\models \psi (\vec{a})$.)
\end{remark}

\begin{example}
\label{noultrapr}
Given a distributive lattice $K$, every coherent functor $K\to \mathbf{Set}$ factors as $K\xrightarrow{P} \mathbf{2}\to \mathbf{Set}$, hence yields a prime filter $p=P^{-1}(1)\in Spec(K)$. Then $LP=\faktor{K}{p}=\faktor{K}{\sim _{p}}$ where $a\sim _{p}b$ iff $\exists x\in p: x\cap a =x\cap b$.

In particular $L_K$ does not preserve ultraproducts if $K$ has an infinite quotient $\faktor{K}{p}$: Given $(p_i\in Spec(K))_{i\in I}$ the prime filter corresponding to the ultraproduct structure (evaluation) is $p'=\faktor{\prod_i p_i}{U} $ given by $a\in p'$ iff $\{i:a\in p_i \}\in U$. Clearly $\faktor{\prod _i K/p_i }{U}\cong \faktor{K}{p'}$ does not hold in general, e.g.~$\faktor{(K/p)^I}{U}\not \cong \faktor{K}{p}$.
\end{example}

\begin{remark}
$LM\neq *$ for any $M:\mathcal{C}\to \mathbf{Set}$ coherent functor. 
\end{remark}

\begin{proposition}
$M:\mathcal{C}\to \mathbf{Set}$ is positively closed iff $LM=\mathbf{2}$.
\label{poscl2}
\end{proposition}

\begin{proof}
$\Rightarrow$ We claim that $u\hookrightarrow x^a \sim v\hookrightarrow y^b$ iff $(a\in Mu \leftrightarrow b\in Mv)$. This is enough as $x\hookrightarrow x^a \not \sim \emptyset \hookrightarrow x^a$ proves $|LM|\geq \mathbf{2}$ and since in any triple $u\hookrightarrow x^a$, $v\hookrightarrow y^b$, $w\hookrightarrow z^c$ two members are equivalent we get $|LM|\leq \mathbf{2}$.
 
To prove the 'only if' direction observe $a\in Mu \Rightarrow (a,b)\in M(\varphi \cap u\times y)=M(\varphi \cap x\times v)\Rightarrow b\in Mv$. To prove the converse first assume $a\in Mu$ and $b\in Mv$. Then $\varphi =u\times v$ proves $u\hookrightarrow x^a$ and $v\hookrightarrow y^b$ to be equivalent. If $a\not \in Mu$ and $b\not \in Mv$ then $(a,b)\not \in M(u\times y \cup x\times v)$ and since $M$ was positively closed there is $\varphi \hookrightarrow x\times y$ such that $\varphi \cap (u\times y \cup x\times v)=\emptyset $ and $(a,b)\in M\varphi $.

$\Leftarrow $ Take $u\hookrightarrow x$ with $a\not \in Mu$. By the above argument $u\hookrightarrow x^a$ cannot be equivalent to $x\hookrightarrow x^a$ hence it must be equivalent to $\emptyset \hookrightarrow x^a$ as $LM$ has two elements. That is, we have $\varphi \hookrightarrow x\times x$ with $(a,a)\in M\varphi $ and with $\varphi \cap u\times x=\emptyset $, therefore $\exists _{\pi }\varphi \cap u=\emptyset $ and $a\in M(\exists _{\pi }\varphi )$.
\end{proof}

\begin{corollary}
$M:\mathcal{C}\to \mathbf{Set}$ is positively closed iff $tp_M$ is the unique natural transformation $M\Rightarrow \overline{S_{\mathcal{C}}}$.
\end{corollary}

\begin{proposition}
\label{injsurj}
Let 
\[\begin{tikzcd}
	{\mathcal{C}} \\
	&& {\mathbf{Set}} \\
	{\mathcal{D}}
	\arrow["F"', from=1-1, to=3-1]
	\arrow[""{name=0, anchor=center, inner sep=0}, "M", from=1-1, to=2-3]
	\arrow[""{name=1, anchor=center, inner sep=0}, "N"', from=3-1, to=2-3]
	\arrow["\eta", shorten <=4pt, shorten >=4pt, Rightarrow, from=0, to=1]
\end{tikzcd}\]
be a 1-cell in $\mathbf{Coh}_e/\mathbf{Set}$. If $F$ is conservative, full wrt.~subobjects, and $\eta $ is elementary then $L((F,\eta ))$ is injective. If $F$ is full wrt.~subobjects, finitely covers its codomain (i.e.~for $y\in \mathcal{D}$ there's a finite effective epimorphic family $(\varphi x_i\to x)_{i<n}$), and $\eta $ is pointwise surjective then $L((F,\eta ))$ is surjective (if $N$ is $\kappa $-geometric then it is enough to assume that $F$ $\kappa $-covers its codomain).
\end{proposition}

\begin{proof}
Injectivity: $[(u\xhookrightarrow{i} x)^a]$ is mapped to $[(Fu\xhookrightarrow{Fi} Fx)^{\eta _x a}]$. Assume $(Fu\xhookrightarrow{Fi} Fx)^{\eta _x a}\sim (Fv\xhookrightarrow{Fj} Fy)^{\eta _y b}$ that is, there's $\chi \hookrightarrow Fx\times Fy$ with $\chi \cap (Fx\times Fv)= \chi \cap (Fu\times Fy)$ and with $(\eta _x(a),\eta _y(b))\in N\chi $. As $F$ is full wrt.~subobjects, $\chi =Fw$, as $F$ is conservative $w\cap (x\times v)= w\cap (u\times y)$ and as $\eta $ is elementary $(a,b)\in w$.

Surjectivity: As $F$ is full wrt.~subobjects and $\eta $ is pointwise surjective it is enough to prove that any $(\chi \hookrightarrow z)^d$ is equivalent to some $(\chi '\hookrightarrow Fx)^{d'}$. The effective epimorphic family $(h_i:Fx_i\to z)_i$ is mapped to a covering by $N$, hence for some $i$ there's $d'\in NFx_i$ with $Nh_i(d')=a$. Then $(h_i^*\chi \hookrightarrow Fx_i)^{d'}$ is equivalent to our original element. 
\end{proof}

We make use of Theorem 2.9.~of \cite{hidden}:

\begin{theorem}
A distributive lattice $K$ has Krull-dimension $\leq n$ iff for any $x_1,\dots x_{n+1}$ there are $a_1,\dots a_{n+1}$ such that
\begin{multline*}
    a_1\wedge x_1=0,\quad a_2\wedge x_2\leq a_1\vee x_1,\ \dots \\ a_{n+1}\wedge x_{n+1}\leq a_n\vee x_n ,\quad  a_{n+1}\vee x_{n+1}=1 
\end{multline*}
\end{theorem}

\begin{remark}
By Proposition \ref{poscl2} if $\mathcal{C}$ is Boolean then for every $M:\mathcal{C}\to \mathbf{Set}$ coherent functor $LM=\mathbf{2}$ is a Boolean algebra. By the above theorem if $K$ is the filtered colimit of the distributive lattices $K_i$ and each has Krull-dimension $dim(K_i)\leq n$ then $dim(K)\leq n$. Hence if $dim(Sub_{\mathcal{C}}(x))\leq n$ for each $x\in \mathcal{C}$ then $dim(LM)\leq n$.
\end{remark}

\begin{proposition}
$\{1\}\subseteq LM$ is a prime filter. The corresponding natural transformation is $tp_M$.
\end{proposition}

\begin{proof}
We need that $u \vee v \sim x\hookrightarrow x^a$ implies $u\sim x$ or $v\sim x$. Indeed, $u\vee v\sim x$ implies $a\in M(u)\vee M(v)$ from which $a\in M(u)$ or $a\in M(v)$.

By Remark \ref{correspondence} the $x$-component of the corresponding natural transformation takes $a\in Mx$ to $\{u\hookrightarrow x : u\hookrightarrow x^a \sim x\hookrightarrow x^a \}=\{u\hookrightarrow x : Mu\ni a \}$.
\end{proof}

\begin{corollary}
$dim(LM)=0$ $\Rightarrow $ $LM=\mathbf{2}$. 
\end{corollary}

\begin{corollary}
\label{4}
Given a coherent functor $M:\mathcal{C}\to \mathbf{Set}$ and a natural transformation $\tau :M\Rightarrow \overline{S_{\mathcal{C}}} $ we have $tp_{M,x}(a)\subseteq \tau _x(a)$ for all $x\in \mathcal{C}$ and $a\in Mx$. (By Remark \ref{correspondence}.)
\end{corollary}

\begin{theorem}
\label{posclequivalent}
Let $M:\mathcal{C}\to \mathbf{Set}$ be a coherent functor. The following are equivalent:
\begin{enumerate}
    \item $M$ is positively closed, i.e.~for all $u\hookrightarrow x\in \mathcal{C}$ and $a\in Mx\setminus Mu$ there is $v\hookrightarrow x$ such that $u\cap v=\emptyset $ and $a\in Mv$.
    \item $LM=\mathbf{2}$,
    \item $LM$ is a Boolean algebra,
    \item For all $x\in \mathcal{C}$ and $a\in Mx$ the prime filter $tp_{M,x}(a)$ is maximal,
    \item For any coherent functor $N:\mathcal{C}\to \mathbf{Set}$ each $M\Rightarrow N$ natural transformation is elementary.
\end{enumerate}
\end{theorem}

\begin{proof}
$1\Leftrightarrow 2\Leftrightarrow 3 \Rightarrow 5 $ was already proved. $4\Rightarrow 2$ follows by Corollary \ref{4}. We are left to show $5\Rightarrow 4$. Basically we repeat the proof of Theorem 8.2.4.~of \cite{hodges}. 

Let $M:\mathcal{C}\to \mathbf{Set}$ be a model and $p=\{\varphi \hookrightarrow x : a\in M\varphi  \}$ be the $x$-type realized by $a\in Mx$. Let $q\supseteq p$ be a maximal filter. By Proposition \ref{converse} it is enough to prove that there is a model $N:\mathcal{C}\to \mathbf{Set}$ and a natural transformation $\alpha :M\Rightarrow N$ such that $q=\{\psi \hookrightarrow x :\alpha _x(a)\in N\psi \} $.

Let $L_{\mathcal{C}}$ be the extended canonical language of $\mathcal{C}$ and add constant symbols $(c_b)_{b\in My}$ for each $y\in \mathcal{C}$. Let $L$ be the resulting signature. $T\subseteq L_{\omega \omega }^g$ is defined to be $Th(\mathcal{C})+Diag(M)+(R_{\psi }(c_a))_{\psi  \in q}$ where $Diag(M)$ is the diagram of $M$, i.e.~it is the set of all closed atomic formulas $\varphi (c_{a_1},\dots c_{a_n})$ for which $M\models \varphi (a_1, \dots a_n)$. It is enough to prove that $T$ is consistent: a model of $T$ yields a coherent functor $N:\mathcal{C}\to \mathbf{Set}$ together with specified elements $(c_b^N\in Ny)_{y\in \mathcal{C}, b\in My}$ such that $(\alpha _y:b\mapsto c_b^N)_{y\in \mathcal{C}}$ is a natural transformation with the required property.

Assume $T$ to be inconsistent. Then (using compactness and that $q$ is closed under finite meets) there are finitely many sentences in $Diag(M)$ and there is $x \hookleftarrow \psi \in q$ such that
\[
Th(\mathcal{C})\vdash \varphi _1(c_{a_{11}},\dots c_{a_{1j_1}})\wedge \dots \wedge \varphi _k(c_{a_{k1}},\dots c_{a_{kj_k}})\wedge R_{\psi }(c_a) \Rightarrow \bot
\]
Since the new constant symbols do not appear in $Th(\mathcal{C})$, we can replace them by free variables $r_1,\dots r_n, s$ and get
\[
Th(\mathcal{C})\vdash \exists r_1,\dots r_n [\varphi _1(r_1,\dots r_n,s)\wedge \dots \wedge \varphi _k(r_1,\dots r_n,s)]\wedge R_{\psi }(s) \Rightarrow \bot
\]

The left conjunct corresponds to a subobject $\chi \hookrightarrow x$ which is in $p$ (as $M$ makes this formula valid when we interpret $s$ to be $a$), and since the right conjunct corresponds to $\psi$, we get $q\ni \chi \cap \psi =\emptyset $ which is a contradiction.
\end{proof}

\begin{theorem}
\label{posclexists}
For any coherent category $\mathcal{C}\not \simeq *$ there is a positively closed model $N:\mathcal{C}\to \mathbf{Set}$.
\end{theorem}

\begin{proof}
Let $M:\mathcal{C}\to \mathbf{Set}$ be any model. Take all pairs $\varphi _i= (u_i\hookrightarrow x_i, a_i\in Mx_i\setminus Mu_i )_{i<\kappa }$. We will define a $\kappa $-chain of models and homomorphisms, whose colimit exists by \cite{rosicky}. $M_0=M$, in limit steps we take colimits. Given $M_i$, if there is an $N$ with a natural transformation $M_i\xRightarrow{\alpha } N$ such that $\alpha _{x_i}(a_i)\in Nu_i$ then we choose one, otherwise take $M_i\xRightarrow{id} M_i=M_{i+1}$. $M_{\kappa }=M'$. Now we do the same with $M'$ to get $M''$, etc. $N$ can be taken to be the colimit of $M\Rightarrow M'\Rightarrow M''\Rightarrow \dots $.

Given $N\xRightarrow{\beta } N'$, if it is not elementary, then for some $u\hookrightarrow x$ in $\mathcal{C}$, there is an element $a\in \beta _x^{-1}(N'u)\setminus Nu$. This is contained in some $M{''''}^{ \dots }{'}x$, and therefore also in $M{''''} ^{\dots }{''}u \subseteq Nu$ by construction.
\end{proof}

\begin{theorem}
\label{intsemancplt}
A coherent category $\mathcal{C}$ is semantically complete iff it is weakly Boolean and 2-valued (i.e.~$|Sub_{\mathcal{C}}(1)|=2$).
\end{theorem}

\begin{proof}
$\Rightarrow $. By the completeness theorem there is a $\mathcal{C}\to \mathbf{Set}^I$ conservative coherent functor. Given $a\hookrightarrow 1$ and a model $M:\mathcal{C}\to \mathbf{Set}$, either $Ma=\emptyset $ or $Ma=*$, and by our assumption this does not depend on the choice of $M$. So $a=0$ or $a=1$ and $|Sub_{\mathcal{C}}(1)|=2$.

Now pick a pair of subobjects $a\hookrightarrow x \hookleftarrow b$. By Theorem \ref{posclexists} there is a positively closed model $M:\mathcal{C}\to \mathbf{Set}$. By assumption if $a\not \leq b$ then $Ma\not \leq Mb$ and hence there is an element $p\in Ma\setminus Mb$. As $M$ is positively closed there is $v\hookrightarrow x$ such that $v\cap b=\emptyset $ and $p\in Mv$. Then we can choose $u=a\cap v$.

$\Leftarrow $. As $|Sub_{\mathcal{C}}(1)|=2$ any coherent functor $\mathcal{C}\to \mathbf{Set}$ reflects the initial object. Then our condition on the subobject lattices guarantees that either $Ma\leq Mb$ for every model $M$ or $Ma\not \leq Mb$ for every model $M$.
\end{proof}

\begin{remark}
The 2-functoriality of $L$ implies that given any coherent functor $F:\mathcal{C}\to \mathcal{D}$ we have a natural transformation
\[\begin{tikzcd}
	{\mathbf{Coh}(\mathcal{C},\mathbf{Set})} \\
	&&& {\mathbf{DLat}} \\
	{\mathbf{Coh}(\mathcal{D},\mathbf{Set})}
	\arrow["{F^*}", from=3-1, to=1-1]
	\arrow[""{name=0, anchor=center, inner sep=0}, "{L_{\mathcal{C}}}", curve={height=-6pt}, from=1-1, to=2-4]
	\arrow[""{name=1, anchor=center, inner sep=0}, "{L_{\mathcal{D}}}"', curve={height=6pt}, from=3-1, to=2-4]
	\arrow["{\Phi _F}", shorten <=5pt, shorten >=5pt, Rightarrow, from=0, to=1]
\end{tikzcd}\]
with $\Phi _{F,M}=L((F,id)):L_{\mathcal{C}}(MF)\to L_{\mathcal{D}}(M)$. By Proposition \ref{injsurj} when $F$ is a Morita-equivalence $\Phi _F$ is an isomorphism.
\end{remark}

Under certain assumptions on $F$ we also have:

\begin{proposition}
\label{commdiag}
Let $\mathcal{C}$ be a coherent category, $\widetilde{\mathcal{C}}$ be $\kappa $-geometric and let $\varphi :\mathcal{C}\to \widetilde{\mathcal{C}}$ be a coherent functor which is fully faithful, full wrt.~regular subobjects and for each $y\in \widetilde{\mathcal{C}}$ there are maps $(h_i:\varphi (x_i)\to y)_{i<\lambda <\kappa }$ with $y=\bigcup _{i<\lambda} \exists _{h_i}\varphi (x_i)$. Then we have the following natural isomorphisms:
\[\begin{tikzcd}
	{\mathbf{Coh}(\mathcal{C},\mathbf{Set})} && {\mathbf{Set}^{\mathcal{C}}} \\
	& \cong &&& {\mathbf{DLat}} \\
	{\mathbf{Geom_{\kappa}}(\widetilde{\mathcal{C}},\mathbf{Set})} && {\mathbf{Set}^{\widetilde{\mathcal{C}}}}
	\arrow["{\varphi ^*}", from=3-1, to=1-1]
	\arrow[hook, from=1-1, to=1-3]
	\arrow[hook, from=3-1, to=3-3]
	\arrow[""{name=0, anchor=center, inner sep=0}, "{Lan_YSub_{\mathcal{C}}}", from=1-3, to=2-5]
	\arrow[""{name=1, anchor=center, inner sep=0}, "{Lan_YSub_{\widetilde{\mathcal{C}}}}"', from=3-3, to=2-5]
	\arrow["{Lan_{\varphi}}"', from=1-3, to=3-3]
	\arrow["{\Psi }", "\cong "', shorten <=4pt, shorten >=4pt, Rightarrow, from=0, to=1]
\end{tikzcd}\]
where $\Psi _{\mathcal{C}(x,-)}$ is $Sub_{\mathcal{C}}(x)\xrightarrow{ \varphi } Sub_{\widetilde{\mathcal{C}}}(\varphi x)$.

In particular under these assumptions $\varphi ^*:\mathbf{Geom_{\kappa}}(\widetilde{\mathcal{C}},\mathbf{Set}) \to \mathbf{Coh}(\mathcal{C},\mathbf{Set})$ is fully faithful. 
\end{proposition}

\begin{proof}
We have
\[\begin{tikzcd}
	{\mathbf{Coh}(\mathcal{C},\mathbf{Set})} && {\mathbf{Set}^{\mathcal{C}}} \\
	&&&&& {\mathbf{DLat}} \\
	{\mathbf{Geom_{\kappa}}(\widetilde{\mathcal{C}},\mathbf{Set})} && {\mathbf{Set}^{\widetilde{\mathcal{C}}}}
	\arrow[""{name=0, anchor=center, inner sep=0}, "{\varphi ^*}", from=3-1, to=1-1]
	\arrow[hook, from=1-1, to=1-3]
	\arrow[hook, from=3-1, to=3-3]
	\arrow[""{name=1, anchor=center, inner sep=0}, "{Lan_YSub_{\mathcal{C}}}", from=1-3, to=2-6]
	\arrow[""{name=2, anchor=center, inner sep=0}, "{Lan_YSub_{\widetilde{\mathcal{C}}}}"', from=3-3, to=2-6]
	\arrow[""{name=3, anchor=center, inner sep=0}, "{\varphi ^*}", curve={height=-12pt}, from=3-3, to=1-3]
	\arrow[""{name=4, anchor=center, inner sep=0}, "{Lan_{\varphi }}", curve={height=-12pt}, from=1-3, to=3-3]
	\arrow["{=}"{description}, draw=none, from=0, to=3]
	\arrow["\dashv"{anchor=center, rotate=-180}, draw=none, from=4, to=3]
	\arrow["\Psi", shorten <=4pt, shorten >=4pt, Rightarrow, from=1, to=2]
\end{tikzcd}\]

Since $\varphi $ is fully faithful, $Lan_{\varphi }$ is also fully faithful and hence $\varphi ^* \circ Lan_{\varphi }\cong 1_{\mathbf{Set}^{\mathcal{C}}}$. Therefore it is enough to see that the full subcategory $\mathbf{Geom}_{\kappa }(\widetilde{\mathcal{C}},\mathbf{Set})$ lies in the essential image of $Lan_{\varphi }$. That is, given a $\kappa $-geometric functor $M:\widetilde{\mathcal{C}}\to \mathbf{Set}$ we should prove that it is the colimit of representables of the form $\widetilde{\mathcal{C}}(\varphi x,-)$. 

Write $M\varphi =colim _{(\int M\varphi )^{op}} \mathcal{C}(x_i,-)$. Then we can define a natural transformation $colim _{(\int M\varphi )^{op}} \widetilde{\mathcal{C}}(\varphi x_i,-)\Rightarrow M$ induced by the natural transformations $\widetilde{\mathcal{C}}(\varphi x_i,-)\Rightarrow M$ which map $1_{\varphi x_i}$ to $a_i\in M(\varphi x_i)$ at the object $(x_i,a_i\in M\varphi (x_i))$ of the indexing category. We would like to prove that this is an isomorphism.

To see that it is pointwise mono assume that we are given $f,g:\varphi x\to y$ with $Mf(a)=Mg(a)$ for some $a\in M\varphi x$. We should see that there is $h:z\to x$ and $c\in M\varphi z$ such that $M\varphi h(c)=a$ and $f\circ \varphi h =g\circ \varphi h$ (this is sufficient as $(\int M\varphi )^{op}$ is filtered). As $\varphi $ is full wrt.~regular subobjects the equalizer of $f$ and $g$ is of the form $\varphi h$ (and we can take $c=a$).

To see that it is pointwise epi pick $a\in M(y)$ and a jointly covering family $(h_i:\varphi x_i \to y)_{i<\lambda <\kappa }$. As $M$ is $\kappa $-geometric $(Mh_i)_{i<\lambda }$ is jointly covering so there is $i<\lambda$ and $b\in M\varphi x_i$ with $Mh_i(b)=a$. Hence $a$ is in the image of the $y$-component of $\widetilde{\mathcal{C}}(\varphi x_i,-)\Rightarrow M$ which map $1_{\varphi x_i}$ to $b$.

Finally, as $\varphi $ is bijective on objects (conservative and full wrt.~subobjects) the induced map between the colimits
\[
\adjustbox{width=\textwidth}{
\begin{tikzcd}
	&& M &&& {Lan_YSub_{\mathcal{C}}M} \\
	{\mathcal{C}(x,-)} & {\mathcal{C}(x',-)} && {Sub_{\mathcal{C}}(x)} & {Sub_{\mathcal{C}}(x')} \\
	&& {Lan_{\varphi }M} &&& {Lan_YSub_{\mathcal{C}}(Lan_{\varphi }M)} \\
	{\mathcal{C}(\varphi x,-)} & {\mathcal{C}(\varphi x',-)} && {Sub_{\mathcal{C}}(\varphi x)} & {Sub_{\mathcal{C}}(\varphi x')}
	\arrow["{f^*}", from=2-1, to=2-2]
	\arrow[curve={height=-12pt}, from=2-1, to=1-3]
	\arrow[from=2-2, to=1-3]
	\arrow["{\varphi f^*}", from=4-1, to=4-2]
	\arrow[curve={height=-12pt}, from=4-1, to=3-3]
	\arrow[from=4-2, to=3-3]
	\arrow["{f^*}", from=2-4, to=2-5]
	\arrow[curve={height=-12pt}, from=2-4, to=1-6]
	\arrow[from=2-5, to=1-6]
	\arrow["{\varphi f^*}", from=4-4, to=4-5]
	\arrow["\varphi"'{pos=0.4}, from=2-4, to=4-4]
	\arrow["\varphi"'{pos=0.4}, from=2-5, to=4-5]
	\arrow[curve={height=-12pt}, from=4-4, to=3-6]
	\arrow[from=4-5, to=3-6]
	\arrow["{\Psi _M}"', dashed, from=1-6, to=3-6]
\end{tikzcd}
}
\]
is an isomorphism.
\end{proof} 

\begin{theorem}
\label{lanfinpr}
Given a coherent category $\mathcal{C}$ the functor $Lan_YSub_{\mathcal{C}}$ preserves finite products of coherent functors.
\end{theorem}

\begin{proof}
Let $\varphi :\mathcal{C}\to \widetilde{\mathcal{C}}$ be the pretopos completion of the coherent category $\mathcal{C}$. Then $\varphi $ satisfies the conditions of Proposition \ref{commdiag} we have the above 2-cells. Since $\varphi ^*$ is an equivalence (hence essentially surjective), $Lan_{\varphi }$ preserves finite products by Example A4.1.10.~of \cite{elephant}, and $\Psi $ is iso, it is enough to prove that $Lan_YSub_{\widetilde{\mathcal{C}}}$ preserves finite products of $\kappa $-geometric functors.

As $\widetilde{\mathcal{C}}$ is a pretopos it has finite disjoint coproducts, and $Sub_{\widetilde{\mathcal{C}}}:\widetilde{\mathcal{C}}^{op}\to \mathbf{DLat}$ maps them to products. Then post-composing with the forgetful functor $\overline{Sub_{\widetilde{\mathcal{C}}}}:\widetilde{\mathcal{C}}^{op}\to \mathbf{Set}$ yields a finite product preserving functor whose left Kan-extension along the Yoneda-embedding $ Lan _Y \overline{Sub_{\widetilde{\mathcal{C}}}}$ is also finite product preserving as finite products form a sound doctrine in the sense of \cite{classification}. But if $M:\widetilde{\mathcal{C}}\to \mathbf{Set}$ is a lex functor then $Lan _Y \overline{Sub_{\widetilde{\mathcal{C}}}}(M)$ is the underlying set of $Lan _Y Sub_{\widetilde{\mathcal{C}}}(M)$ as $(\int M)^{op}$ is filtered and the forgetful functor $U:\mathbf{DLat}\to \mathbf{Set}$ preserves filtered colimits. As it reflects all limits it follows that $Lan_YSub_{\widetilde{\mathcal{C}}}$ preserves finite products of lex functors.
\end{proof}

\begin{question}
Can we find conditions on $\mathcal{C}$ which guarantee that $Lan_YSub_{\mathcal{C}}$ preserves $<\kappa $-products of coherent functors? As it preserves filtered colimits in this case $L_{\mathcal{C}}$ would preserve $<\kappa $-ultraproducts. We have seen in Example \ref{noultrapr} that this is not the case in general.
\end{question}

\printbibliography

\end{document}